\documentclass[11pt]{article}
\usepackage{DV}

\title{Asymptotic expansion for transport maps \\ between laws of multimatrix models}
\author[1]{David Jekel}
\author[2]{Evangelos A.\ Nikitopoulos}
\author[3]{F\'elix Parraud}

\affil[1]{Department of Mathematical Sciences, University of Copenhagen\protect\\
Universitetsparken 5, 2100 København (Denmark)\protect\\
{Email: \tt \href{mailto:daj@math.ku.dk}{daj@math.ku.dk}}\vspace{2mm}}
\affil[2]{Department of Mathematics, University of Michigan\protect\\
\noindent 530 Church Street, Ann Arbor, MI 48109-1043 (USA)\protect\\
{Email: {\tt \href{mailto:enikitop@umich.edu}{enikitop@umich.edu}}}\vspace{2mm}}
\affil[3]{Department of Mathematics and Statistics, Queen's University\protect\\
\noindent 48 University Avenue, Kingston, ON K7L 3N6 (Canada)\protect\\
{Email: {\tt \href{mailto:enikitop@umich.edu}{felix.parraud@gmail.com}}}}
\date{\vspace{-5ex}}

\begin{document}

\maketitle

\begin{abstract}
We study the large-$N$ behavior of random matrix tuples $Y^N = (Y_1^N,\dots,Y_d^N)$ with joint density proportional to $e^{-N^2 V}$ for some convex function $V$ in non-commuting variables satisfying certain bounds on its second derivative.  We give an asymptotic expansion in powers of $1/N^2$ of the trace of noncommutative smooth functions of $Y^N$.  We also give an asymptotic expansion for a family of maps $T^N$ that transport the law of a tuple of independent GUE random matrices to the law of $Y^N$ and, as a consequence, show strong convergence for the multimatrix models $Y^N$.  Our proof is based on an asymptotic expansion for the heat semigroup associated to the measure, which is expressed in terms of smooth functions of a matrix Brownian motion $(S^{N}_t)_{t \geq 0}$.  We introduce spaces of noncommutative smooth functions that unify and generalize the cases of polynomials and single-variable smooth functions and allow the systematic application of asymptotic expansion techniques to multimatrix models with convex interaction.
\end{abstract}

\newpage

\tableofcontents

\clearpage

\section{Introduction}

Write $\M_N(\C)$ for the set of $N \times N$ complex matrices and $\M_N(\C)_{\sa}$ for the real subspace of Hermitian matrices.
A key class of models in random matrix theory, called multimatrix models, are tuples of several random matrices $Y^N = (Y_1^N,\dots,Y_d^N)$ in $\M_N(\mathbb{C})_{\sa}^d$ with (joint) law
\[
d\mu_{V}^N(x) = \frac{1}{Z_V^N} e^{-N^2 \tr_N(V(x))}\,dx,
\]
where $dx$ is the Lebesgue measure on $\M_N(\C)_{\sa}^d$, $V(x)$ is some ``noncommutative function'' of $d$ variables (often a noncommutative polynomial) called the potential of the model, $\tr_N = N^{-1}\Tr_N$ is the normalized trace on $\M_N(\mathbb{C})$, and $Z_V^N$ is a normalizing constant.  These are analogs of Gibbs laws in the statistical mechanical approach to random matrix theory described in the single-matrix case in \cite{BdMPS}.  Besides their inherent interest in random matrix theory, these multimatrix models are a tool for studying the matrix integral $\int_{\M_N(\mathbb{C})_{\sa}^d} e^{-N^2 \tr_N(V(x))}\,dx$ using the relation
\[
\frac{d}{d\beta} \left[ -\frac{1}{N^2} \log \int e^{-\beta N^2 \tr_N(V)} \right] = \frac{\int \tr_N(V) e^{-\beta N^2 \tr_N(V)}}{\int e^{-\beta N^2 \tr_N(V)}},
\]
as well as understanding the large deviations theory of several independent GUE random matrices \cite{BCG2003}.  Moreover, the asymptotic expansion of $\mathbb{E} \tr_N[f(Y^{N})]$ in powers of $N$ relates to the combinatorial enumeration of maps, a fact that was applied in the context of topology by Harer and Zagier, in \cite{harer1986euler}, to compute the Euler characteristic of the moduli space of curves. We also refer to the survey \cite{zvonkin1997matrix} by  Zvonkin for an introduction to this topic. In physics,  the seminal works of t'Hooft \cite{t1974magnetic} and Br\'ezin, Parisi, Itzykson and Zuber \cite{brezin1978planar} related multimatrix models to the enumeration of maps of any genus, hence providing a purely analytical tool to solve these hard combinatorial problems. The idea is that one can view the free energy of multimatrix models of dimension $N$ as a formal power series in $N^{-1}$ whose coefficients are generating functions of maps on a surface of a given genus.

During the last two decades, the study of multimatrix models has been quite active.
In \cite{GMS2006} and \cite{guionnet2007second}, Guionnet and Maurel-Segala studied the first- and then second-order of the asymptotic expansion before giving a full expansion in \cite{maurel2006high}. More recently, in \cite{guionnet2022matrix}, they also studied the case of multimatrix models whose law is far from the quadratic potential. In \cite{parraud2025free}, the third author and Schnelli gave a new formula and an expansion to any order for any self-adjoint potential close to the quadratic one. Moreover, the unitary equivalents of these matrix models also have a long history starting with the Harish-Chandra--Itzykson--Zuber model, \cite{itzykson1980planar,matytsin1994large,zinn2003some}, which has since then been extended to more general potentials, \cite{collins2009asymptotics,guionnet2015asymptotics,figalli2016universality,buc2024topological}.

In particular, for suitable $V$ such that $\tr_N(V)$ is convex, the large-$N$ limit
\[
\mu_V(f(Y)) = \lim_{N \to \infty} \mathbb{E} \tr_N[f(Y^N)]
\]
exists for noncommutative polynomials $f$;
see, for example, \cite{GMS2006,GS2009,JekelEntropy,parraud2025free}. The resulting functional $\mu_V$ on noncommutative polynomials is known as a \emph{free Gibbs law}.  Free Gibbs laws have many analogous properties to classical Gibbs measures:
\begin{enumerate}
    \item They satisfy a noncommutative integration-by-parts relation (known as a Schwinger--Dyson equation):
    \[
    \mu_V(\mathcal{D}_j V(Y) f(Y)) = \mu_V \otimes \mu_V(\partial_j f(Y)),
    \]
    where $\partial_j$ and $\mathcal{D}_j$ are Voiculescu's free difference quotient and cyclic gradient, respectively.  See \cite{GMS2006}.
    \item They are stationary measures for a free Langevin stochastic differential equation:
    \[
    dY_t = dS_t - \frac{1}{2} \mathcal{D} V(Y_t)\,dt,
    \]
    where $S_t$ is a $d$-variable free Brownian motion and $\mathcal{D} V = (\mathcal{D}_1 V,\dots,\mathcal{D}_d V)$ is the free analog of $\nabla \tr_N(V)$ on matrices.  See \cite{BS2001,GS2009,Dabrowski2010}.
    \item They are characterized as maximizers of Voiculescu's free entropy in the microstate framework.  See \cite[\S 5]{GS2009}, \cite[\S 7--8.1]{JLS2022}.
\end{enumerate}
The analogous problems for non-convex $V$ are largely open, although this problem was indeed studied in \cite{guionnet2022matrix} and is related to the difficult problem of completing the large deviations principle for several GUE matrices in \cite{BCG2003}.

This paper will broaden the techniques from \cite{parraud2023asymptotic,parraud2023asymptotic2,parraud2025free,Jekel2025asymptotic} and extend the results of asymptotic expansions to a more general ensemble of functions (see Theorem \ref{3lessopti}). We then use this result to systematically study multimatrix models associated to convex potentials $V$.  We will study both the expectation of noncommutative functions in the large-$N$ limit as well as the asymptotic expansion of transport functions that push forward $\mu_{V}^N$ to the normalized Gaussian measure, studied earlier in \cite{JLS2022}. It is worth noting that the dimension-independence hypothesis on $V$ is purely to keep the already-heavy notational burden to a manageable level;
one could certainly track the dependence of the error terms on $V$ in our proofs if necessary.

The first theorem we mention here is an asymptotic expansion of $\mathbb{E} \tr_N[f(Y^{N})]$ in powers of $1/N^2$ for noncommutative polynomials $f$ with somewhat sharp error bounds, which works for $V$ and $f$ in a general class of noncommutative smooth functions.  Allowing $f$ to be smooth and not merely polynomial is crucial for applications to strong convergence, i.e., convergence of operator norms, such as Theorem \ref{cor2} below.  Indeed, if $f$ is a self-adjoint noncommutative polynomial, then to rule out outlying eigenvalues for $f(Y^{N})$, one needs to control $\E \tr_N [\varphi (f(Y^{N}))]$ up to order $1/N^2$ when $\varphi$ is a smooth function on $\R$.  The utility of the asymptotic expansion in \cite{parraud2024spectrum}, as opposed to the combinatorial expressions in terms of maps, is the analytic control it provides.
The terms in the asymptotic expansion are expressed in terms of derivatives of the input function.
Thus, the terms and error bounds in the formula make sense to apply to smooth functions to begin with and thereby facilitate effective bounds on the largest eigenvalue.

The aim of this work, however, is much more general than to study a single-variable smooth function composed with a noncommutative polynomial.  Indeed, such a composition should be seen as an element of a general class of noncommutative smooth functions that includes both single-variable smooth functions and multivariable noncommutative polynomials, and is closed under composition.  By introducing a general framework, we can thus avoid the ad hoc computations that arise from assuming the particular structure of $\varphi \circ f$.  Thus far, several classes of noncommutative smooth functions have been introduced in the literature for such purposes;
see \cite{DGS2021} and \cite{JLS2022}.  However, the seminorms on the derivatives from \cite{JLS2022} are not a priori strong enough for our purposes.  Indeed, the asymptotic expansion formula for a noncommutative polynomial requires looking at its fourth free difference quotient, which lives in a fourfold tensor product, and then applying the twisted multiplication $a \otimes b \otimes c \otimes d \mapsto badc$.  The multiplication operation is not bounded with respect to any $\mathrm{C}^*$-tensor norm, for instance, nor for the norms in \cite{JLS2022} that view the derivatives as multilinear maps and hence are more similar to injective tensor norms.  Our approach here is thus more similar to \cite{DGS2021}, but we use Banach-space projective tensor norms rather than symmetrized Haagerup tensor norms.

The asymptotic expansion for the expectation is stated as Theorem \ref{cor1} below; see further below for an explanation of the space $\cC^k(X_1,\dots,X_d)$ of noncommutative smooth functions of $d$ variables.  In fact, we will deduce Theorem \ref{cor1} from a stronger result, Theorem \ref{mainthm} below, which gives an asymptotic expansion for the transport maps constructed as in \cite{DGS2021,JLS2022}.  The assumption on the potential $V$ below is that $V(X) = \frac{1}{2} \sum_{i=1}^d X_i^2 + W(X)$, where $W$ is a ($4k+4$)-regular noncommutative smooth function of smoothness order $16k+23$;
loosely speaking, the ($4k+4$)-regularity hypothesis means that the second derivatives of $W$ satisfy a bound of order $1/k$.

\begin{theorem}[Asymptotic expansion of expectations]\label{cor1}
Let $V \in\cC^{16k+23}(X_1,\dots,X_d)$ be a $(4k+4)$-regular potential (Definition \ref{definition:kregular}), $f\in\cC^{4k+4}(X_1,\dots,X_d)$, and $Y^N$ be a multimatrix model with potential $V$ (Definition \ref{def:randompot}).
There exist functions $f_i\in\cC^{k-i+1}(X_1,\dots,X_d)$ ($i = 0,\ldots,k$) such that if $x$ a $d$-tuple of free semicircular variables, $K > 0$, and $N \in \N$ satisfies $N>2K$, then
$$ \left| \E\left[\frac{1}{N}\Tr\left( f(Y^N) \right) \right] - \tau\big( f_0(x) \big) - \frac{\tau\big(f_1(x)\big)}{N^2} - \dots - \frac{\tau\big(f_k(x)\big)}{N^{2k}} \right| \leq \frac{C_{V,k,K} \norm{f}_{\cL^k,K}}{N^{2k+2}},$$
where $\norm{f}_{\cL^k,K} \coloneqq \sup \big\{\norm{f}_{\cC^k,K} e^{-KR^2} : R > 0 \big\}$ (Definition \ref{def.NCsmoothbasic}).
\end{theorem}

Our asymptotic expansion for $\mathbb{E} \tr_N[f(Y^N)]$ significantly extends past work in terms of the smoothness assumptions on both $f$ and $V$.  Many works on the asymptotic expansion often assumed that $V$ is polynomial or analytic \cite{GMS2006,maurel2006high,parraud2025free}.  Others such as \cite{figalli2016universality,JekelEntropy,JekelExpectation,JLS2022} addressed more general $V$ but did not give a full asymptotic expansion.  The main limitation in our work is that we need to assume a bound of order $1/k$ on the second derivatives of $V(X) - \frac12\sum_{i=1}^d X_i^2$ in order to get the asymptotic expansion up to order $k$.  Based on \cite{GS2009,JekelExpectation,DGS2021,JLS2022}, one might expect that all the results hold provided that the second derivative minus the identity is bounded by a constant less than $1$ (independent of $k$), but this seems to be out of reach with our method.

One of the key ingredients for our approach is that the measure $\mu_V^N$ is the stationary solution for the matrix Langevin SDE
\begin{equation}
    \label{skdvjnslkdvn}
    dY_t^{N} = dS_t^{N} - \frac{1}{2} \nabla [\tr_N(V)](Y_t^{N})\,dt,
\end{equation}
and also gives the long-time behavior when the initial condition is $0$, for instance; this allows us ultimately to study the heat semigroup associated to $\mu_V^N$ and construct transport of measure as we discuss further below.  The solution $(Y_t^{N})_{t \geq 0}$ is a function of the Brownian motion $(S_t^{N})_{t \geq 0}$.  We aim to show that $X_t^{N}$ is in fact a (dimension-independent) noncommutative \emph{smooth} function applied to the Brownian motion $(S_t^{N})_{t \geq 0}$, i.e., a noncommutative smooth function of an infinite family of GUE variables, to which we can apply the asymptotic expansion of \cite{parraud2023asymptotic}.  The bulk of the technical work in this paper is to make sense of this notion of smooth functions of the Brownian motion and show that the solution is indeed such a smooth function. To do so, we unfortunately rely on the simple structure of the SDE \eqref{skdvjnslkdvn}. In particular, our method does not allow for a more complicated diffusion term at the moment.

We remark that another possible approach would be to generalize the proof of \cite{parraud2023asymptotic} by using the integration by parts formula for $\mu_V^N$ in place of the Gaussian integration by parts.  In this case, the coefficients in the expansion would be expressed in terms of expectations with respect to the free Gibbs law $\mu_V$. One would then need to evaluate the expectation of these terms under the free Gibbs law, but our approach here avoids this by expressing the terms as functions of an infinite family of GUE matrices to begin with. Yet another approach would be to view our multimatrix model as a perturbation of the law of a $d$-tuple of independent GUE random matrices and then work with the usual Gaussian integration by parts; this idea is used in \cite{parraud2025free} but is limited to analytic functions and potentials.  

Our spaces of noncommutative smooth functions are defined in \S \ref{sec.NCfunc} as follows.  Note that although the statement of Theorem \ref{cor1} only references smooth functions in finitely many variables, the proof uses infinitely many variables to describe the Brownian motion.  Hence, it is convenient to introduce variables $x_e$ indexed by a Hilbert space with orthonormal basis $\mathcal{B}$ as in \cite{Jekel2025asymptotic}, since this formalism fits well with the classical and free Gaussian functors.  The vector formalism is also partly inspired by Malliavin calculus, but the seminorms are comparatively much stronger than the norms on Cameron--Martin space.  Like in \cite{DGS2021,JLS2022}, our $\cC^k(\mathcal{B})$ space is defined as a completion of the space of \emph{noncommutative trace polynomials} (expressions such as $x_{e_1} x_{e_2} \tr(x_{e_3} x_{e_2}) + x_{e_2} x_{e_3} \tr(x_{e_1} x_{e_3}) + x_{e_1}$), which have a long history in free probability (see e.g., \cite{Cebron2013,DHK2013}).  We consider the free difference quotient operation $\partial$ acting on the noncommutative polynomial part (the terms outside the traces) and mapping into a tensor product, and also the derivative $\tilde{\partial}$ which differentiates each trace terms $\tr(f)$ and replaces it with $\otimes \mathcal{D} f$ where $\mathcal{D}$ is Voiculescu's cyclic derivative.  With each of the derivatives we keep track of which variable has been differentiated by tensoring on the corresponding vector in the Hilbert space as in \cite{Jekel2025asymptotic}.  For each $k$th-order derivative of $f$ and for $R > 0$, we define a seminorm as the supremum of the Banach space projective tensor norm of the derivative evaluated on tuples $x$ in a tracial $\mathrm{C}^*$-algebra satisfying $\sup_{e \in \mathcal{B}} \norm{x_e} / c_e \leq R$ where $c_e$ are chosen positive weights.

One of the most important features of these noncommutative smooth function spaces is that they are closed under all the operations that we need to use.  They are closed under composition and satisfy a chain rule (Theorem \ref{chainrule}).  They are closed under various permuted multiplication maps that combine different tensorands (Proposition \ref{proppourvoila}).  They are closed under plugging in free semicirculars into some of the variables and then taking the conditional expectation onto the algebra generated by the other variables (Proposition \ref{prop:condexp}); this is where it is essential that we include trace polynomials and not only noncommutative polynomials.  We can also prove an asymptotic expansion for noncommutative smooth functions of GUE matrices (Theorem \ref{3lessopti}) by a direct generalization of the techniques from \cite{parraud2023asymptotic}.

These properties are crucial for the construction of transport in our main result (Theorem \ref{mainthm}), which gives an asymptotic expansion for certain smooth functions $T^N$ which express $\mu_V^N$ as a pushforward of the GUE measure; in fact, we will deduce Theorem \ref{cor1} from Theorem \ref{mainthm}.  The construction of transport functions has played a significant role in random matrices and free probability.  Notably, \cite{figalli2016universality,bekerman2015transport} used approximate transport functions to study universality phenomena for the eigenvalues, and Guionnet and Shlyakhtenko \cite{GS2014} used transport to establish that the von Neumann algebras associated to free Gibbs laws are isomorphic to the von Neumann algebra of free semicirculars, a result which was subsequently generalized in \cite{Nelson2015a,Nelson2015b,DGS2021,JekelExpectation,JLS2022}.  We follow the same construction of transport maps as in \cite{DGS2021,JLS2022}.  The idea is to fix a path of convex potentials $V_t$ that interpolates between the quadratic potential at time $0$ and the desired potential at time $1$.  We then construct a family of maps $T_{s,t}^{(N)}$ by solving
\[
\partial_s T_{s,t}^{N} = \nabla (\Psi_{V_s}^{N} \tr_N(\dot{V}_s)) \circ T_{s,t}^{N},
\]
where $\Psi_{V_t}^N$ is the pseudo-inverse of the Laplacian operator $L_{V_t}^N$ associated to the $\mu_{V_t}^N$.  It is a standard technique from transport equations that $T_{s,t}^{(N)}$ will push $\mu_{V_t}^{N}$ forward to $\mu_{V_s}^{N}$, and in fact, this construction of transport is infinitesimally optimal.

The gradient of the pseudo-inverse $\nabla \Psi_{V_t}^{N}F$ for some function $F$ is in turn expressed as
\[
\nabla (\Psi_{V_t}^N F) = \int_0^\infty \nabla \left[e^{sL_{V_t}^{N}}F\right]\,ds 
\]
where $e^{sL_{V_t}^{N}}H$ is the heat semigroup given by
\[
e^{sL_{V_t}^{N}}F(x) = \mathbb{E}_{X_0 = x}[F(Y_s)],
\]
where $Y_t^N$ is the solution to the Langevin equation associated to $\mu_{V_t}^N$.  Hence, after proving that $Y_t^{N}$ is a noncommutative smooth function of the Brownian motion, and then showing appropriate tail bounds on the heat semigroup for large time, we can express $\nabla (\Psi_{V_t}^{N} F)$ as a noncommutative smooth function and obtain an asymptotic expansion. We refer to \S\ref{sec.transport} for more details.

We therefore obtain the following result giving asymptotic expansions for smooth transport to and from a GUE tuple.

\begin{theorem}
\label{mainthm}
    Let $V\in\cC^{4k+7}(X_1,\dots,X_d)$ be a $k$-regular potential (see Definition \ref{definition:kregular}), and $Y^N$ a multimatrix model with potential $V$ (see Definition \ref{def:randompot}). There exists a transport map $T^N$ between the law of $Y^N$ and the one of a $d$-tuple of independent GUE random matrices $X^N$, i.e., $T^N$ such that for all measurable functions $f$,
    $$ \E\left[ f(Y^N) \right] = \E\left[ f\circ T^N(X^N) \right].$$
    
    \noindent Moreover, there exists $T^0 \in (\cC^k(X_1,\dots,X_d))^d$, as well as $T^m\in (\cC^{k+1-m}(X_1,\dots,X_d))^d$ for $m\in [1,k]$, such that, for all $N\geq 1$,
    \begin{equation}
        \label{eq:main2}
        \sup_{H\in\M_N(\C)_{sa}^d}\ \norm{T^N(H) - \sum_{i=0}^k \frac{T^i(H)}{N^{2i}} } \leq \frac{C_{V,k}}{N^{2k+2}},
    \end{equation}
    for some constant $C_{V,k}$. Besides, there exists another constant $C_{V,k}$ such that
    \begin{align}
        \label{eqkvjnso2}
        \norm{T^0 - \id }_{\cC^k,\infty}\leq C_{V,k}, \quad\quad \norm{T^m}_{\cC^{k+1-m},\infty}\leq C_{V,k},\quad 1\leq m\leq k.
    \end{align}
\end{theorem}

Previous works on free transport already showed that their finite-dimensional transport maps will approximate the corresponding free transport maps in various settings.  Guionnet and Shlyakhtenko \cite[Theorem 4.7]{GS2014} showed that their free monotone transport maps approximate the finite-dimensional optimal transport maps, but only in an $L^2$ sense.  In \cite{JekelExpectation}, the transport maps were constructed using the evolution of the density along the heat flow, and the error between the finite-dimensional transport maps and the limiting transport map was measured in normalized Hilbert--Schmidt norm uniformly on operator norm balls, and the hypotheses on the finite-dimensional approximations for $V$ also used the same form of asymptotic approximation.  In \cite[\S 8.2]{JLS2022}, the notion of asymptotic approximation was upgraded to look at the error in operator norm uniformly on each operator norm ball.  Although this is not explicitly stated there, it would not be hard to deduce from the arguments in \cite[\S 8.3-8.4]{JLS2022} that the resulting finite-dimensional transport maps will approximate the free transport map in operator norm on each ball.  However, now in Theorem \ref{mainthm}, we have much stronger control over the finite-dimensional transport maps, for potentials satisfying correspondingly stronger assumptions.

As a consequence, we also obtain strong convergence in distribution for the multimatrix models from convex potentials, that is, convergence of the operator norms of smooth functions in the matrix tuple $Y^N$.  This extends the result of \cite[\S 6]{GS2009} which proved strong convergence for polynomial $V$ satisfying certain convexity assumptions by approximating $Y^N$ and $y$ by functions of the GUE/semicircular Brownian motion in the Langevin equation, and then applying Haagerup and Thorbj{\o}rnsen's result for the GUE/semicircular case \cite{haagerup2005new}.  We argue by transferring the strong convergence from the GUE model to the other model using the transport maps from Theorem \ref{mainthm}.  We remark that for this purpose, it would be sufficient to prove by any method that the finite-dimensional transport maps are close to the free transport maps in operator norm on balls.  In particular, to deduce strong convergence from Theorem \ref{mainthm}, we only need the zero-order term and first-order error bound in the asymptotic expansion of $T^N$.  We also note that the first proof of strong convergence for Haar unitary matrices by Collins and Male \cite{ColMal2014unitary} used a similar idea, showing that the Haar unitary matrix could be coupled closely with a continuous function applied to a GUE matrix.

\begin{theorem}\label{cor2}
Given $V\in\cC^{7}(X_1,\dots,X_d)$ a $0$-regular potential (see Definition \ref{definition:kregular}), and $f\in\cC^{0}(X_1,\dots,X_d)$. Let $Y^N$ be a multimatrix model with potential $V$ (see Definition \ref{def:randompot}). Then, almost surely, 
$$ \lim_{N\to\infty} \norm{f(Y^N)} = \norm{f\circ T^0(x)}, $$
where $x$ is a $d$-tuple of free semicircular variables, and $T^0$ is the map from Theorem \ref{mainthm}.
\end{theorem}

Convergence of operator norms of polynomials in random matrices was introduced in the seminal work of Haagerup and Thorbj{\o}rnsen on GUE matrices \cite{haagerup2005new}, and subsequently generalized for many other matrix ensembles \cite{Schultz2005,ColMal2014unitary}. Recent research has made substantial progress in studying more and more diverse models of random matrices, motivated in particular by Hayes' result that strong convergence for tensors of independent GUE matrices implies the Peterson--Thom conjecture on the structure of free group von Neumann algebras \cite{HayesPT}. While the original approach of Haagerup and Thorbj{\o}rnsen relied on carefully studying the solutions of the Schwinger--Dyson Equations satisfied by the model of random matrices, several new approaches were introduced in the last few years. 
\begin{itemize}
    \item First, the non-backtracking random walk approach, \cite{bordenave2019eigenvalues,bordenave2023norm} allowed the authors to control the norm of any polynomial in independent random matrices whose law where the Haar measure on compact matrix groups, by combining careful estimate on the law of the entries of our random matrix, often provided by the Weingarten calculus \cite{collins2022weingarten}, with the concept of non-backtracking operators.

    \item Another approach introduced in \cite{collins2022operator}, further refined in \cite{parraud2022operator,parraud2023asymptotic,parraud2023asymptotic2,parraud2024spectrum,banna2025strong,parraud2025free}, consisted in interpolating our random matrices with free operators, with the help of the free stochastic calculus. This theory of stochastic calculus, developed by Biane and Speicher in \cite{biane1998stochastic}, replaces independent random variables by free operators, and has turned out to be very useful to compute matrix integrals. This is the approach that was chosen in this paper.

    \item At the same time, a similar interpolation was developed in \cite{bandeira2023matrix}.  The authors studied general Gaussian matrices, expressed as a linear combination of deterministic matrices times classical Gaussian coefficients.  They then interpolated between matrices obtained by replacing each coefficient with a Gaussian diagonal matrix and a Gaussian Wigner matrix respectively, and computed the derivative of moments and resolvents of these operators. These results were extended to other models in \cite{brailovskaya2024universality,bandeira2024matrix} and applied to several matrices in \cite{JLNP2025strong}.

    \item Finally, one can directly interpolate along the dimension of our random matrices. The so-called polynomial method was first applied to random matrices in \cite{CGVTvT2026new} and further developed in \cite{CGVvH2024strong2,magee2026strong} for example, yields especially sharp results for any family of random matrices whose expected traces of polynomials in them, i.e., $\E[\tr_N(P(Y^N))]$, is a rational function evaluated in $N^{-1}$.
\end{itemize}

We also refer to \cite{van2025strong} for a survey of the recent developments on strong convergence.  We emphasize that many of the other techniques listed above would not directly apply to the multimatrix models in Theorem \ref{cor2} because they rely on more algebraic structure.  For instance, the non-backtracking random walk approach is better suited to unitary random matrix models and requires a good understanding of the asymptotic behavior of the law of the coefficients of our random matrices, which is simply not available for our model.  The polynomial method relies on $\mathbb{E} \tr_N(P(Y^N))$ being an analytic function in $1/N$, which we will certainly not be able to guarantee when $V$ is merely a $\cC^k$ function, and is even unclear if $V$ is polynomial, as was the case in \cite{parraud2025free}.  Our approach ultimately relies on applying the interpolation method to the GUE Brownian motion and free Brownian motion.  But within the context of interpolation techniques, we also emphasize that it does not seem feasible to use computations with only noncommutative polynomials, resolvents, the linearization method, and the like, in order to understand the behavior of smooth functions.  The multimatrix $Y^N$ for instance are not obtained in any obvious algebraic way from GUE matrices or even general Gaussian matrices.


Finally, it is natural in light of the strong convergence literature to ask to what extent the theorems and methods in our work can be upgraded to handle tensoring with deterministic matrices.  This might require incorporating operator-space like structure into our spaces of smooth functions.  We leave this problem as an interesting direction for future research.

The rest of the paper is organized as follows:
\begin{itemize}[itemsep=0pt]
    \item \S \ref{3deffree} recalls some standard notation and facts from free probability and random matrix theory.
    \item \S \ref{sec.NCfunc} defines spaces of non-commutative smooth functions and describe the setting that we will need to study transport maps.
    \item \S \ref{sec.asymptoticexpansion} extends the asymptotic expansion results from \cite{parraud2023asymptotic} to these spaces.
    \item \S \ref{sec.SDE} analyzes the Langevin SDE using an integral-equation framework.
    \item \S \ref{sec.transport} completes the construction of transport maps and proves Theorem \ref{mainthm}.
    \item \S \ref{sec.otherproofs} proves Theorems \ref{cor1} and \ref{cor2}.
\end{itemize}

\subsection*{Funding}

DJ was partially supported by the Danish Independent Research Fund, grant 1026-00371B; and an EU Horizon Marie Sk{\l}odowska Curie Action, FREEINFOGEOM, grant id: 101209517.
EAN was partially supported by NSF grant DGE 2038238.

\subsection*{Acknowledgments}

Early motivation for this work included discussions of DJ with Alice Guionnet and Yoann Dabrowski in Lyon in March 2020, supported by a Simons Visiting Professor grant from Mathematisches Forschungsinstitut Oberwolfach before the workshop ``Real Algebraic Geometry with a View Toward Hyperbolic Programming and Free Probability.''  Our collaboration began during the visit of FP to UC San Diego in May 2023, supported by the Linda Peetre Memorial Fund, and continued during the ``Workshop on Operator Algebras and Applications: Free Probability'' at the Fields Institute in November of 2023.  We thank Todd Kemp for his mentorship during the San Diego visit and Bruce Driver for the use of his office.

\section{Standard definitions}
\label{3deffree}

\subsection{Free probability}

We start by recalling some definitions from free probability.  See also \cite{VDN1992,AGZ2009,MS2017} for more background.

\begin{definition} ~
	\label{3freeprob}
	\begin{enumerate}[font=\normalfont,label=(\roman*)]
		\item A $\mathrm{W}^*$-probability space $(\cA,\tau) $ is a unital von Neumann algebra $\cA$ (with operator norm $\norm{\cdot}$) endowed with a faithful, normal, tracial state $\tau \colon \cA \to \C$, i.e., a $\sigma$-WOT--continuous linear map $\tau \colon \cA \to \C$ satisfying $\tau(1_{\cA})=1$, $\tau(ab) = \tau(ba)$ for all $a,b \in \cA$, and $\tau(a^*a)\geq 0$ for all $a\in \cA$, with equality if and only if $a=0$.
        An element of $\cA$ is called a (noncommutative) random variable.
		\item Unital $\ast$-subalgebras $\cA_1,\dots,\cA_n$ of $\cA$ are said to be freely independent, or free for short, if for all $k \in \N$ and all $a_1\in\cA_{j_1},\ldots,a_k \in \cA_{j_k}$ such that $j_1\neq j_2, j_2\neq j_3,\ldots,j_{k-1}\neq j_k$,
		\begin{equation}
			\label{osijs}
			\tau\Big( (a_1-\tau(a_1))(a_2-\tau(a_2))\dots (a_k-\tau(a_k)) \Big) = 0.
		\end{equation}
		Families of noncommutative random variables are said to be free if the $\ast$-subalgebras they generate are free. Note, in particular, that if $X$ and $Y$ are free, then
		\begin{equation}
			\label{oidscjosc}
			\tau(XY)=\tau(X)\tau(Y).
		\end{equation}

        \item Write $\C \la \cB \ra$ for the free unital $\C$-algebra on the set $\cB$, i.e., the set of noncommutative (complex) polynomials in the indeterminates $X = (X_e)_{e \in \cB}$.
        Two families $x=(x_e)_{e \in \cB} \in \cA^{\cB}$ and $y=(y_e)_{e \in \cB} \in \cA^{\cB}$ of noncommutative random variables are said to have the same noncommutative law or distribution if for all $P \in \C\la \cB \ra$,
        $$ \tau(P(x)) = \tau(P(y)).$$
		
		\item A family $x = (x_e)_{e \in \cB}$ of noncommutative random variables is called 
		a free semicircular system when the noncommutative random variables are free, 
		$x_e$ is self-adjoint (i.e., $x_e=x_e^*$) for all $e \in \cB$, and for all $k \in \N$ and all $e \in \cB$,
        $$
		\tau( x_e^k) = \begin{cases}
		    c_{\frac{k}{2}} & \mbox{if } k \text{ is even}, \\
			0 & \text{otherwise,}
		\end{cases}
		$$
		where $c_n$ is the $n^{\text{th}}$ Catalan number for all $n \in \N$.
        See \cite[Proposition 7.15]{nica_speicher} for a construction of such a system.
	\end{enumerate}
\end{definition}

We also introduce some general notation.

\begin{notation}~
	\label{3tra}
	\begin{enumerate}[font=\normalfont,label=(\roman*)]
		\item $e_1,\ldots,e_N$ is the standard basis of $\C^N$.
		\item $\Tr_N \colon A \mapsto \sum_{u=1}^N e_u^* A e_u $ is the non-normalized trace on $\M_N(\C)$, and $\tr_N \colon A \mapsto \frac1N \Tr_N(A)$ is the normalized trace on $\M_N(\C)$.
		\item $E_{r,s}=e_r^*e_s$ is the matrix with $(i,j)$ entry zero if $(i,j) \neq (r,s)$ and one if $(i,j) = (r,s)$.
	\end{enumerate}
\end{notation}

Let $(x^i_{r,s})_{1\leq i\leq d, 1\leq r\leq s\leq N}\cup (y^i_{r,s})_{1\leq i\leq d, 1\leq r< s\leq N} $ be a free semicircular system, fix $\cA_N=\M_N(\cA)$, and define $x_i^N\in\cA_N$ by
$$
\sqrt{N}\ (x_i^N)_{r,s} = \left\{
\begin{array}{ll}
	\frac{x_{r,s}^i+\ii y_{r,s}^i}{\sqrt{2}} & \mbox{if } r<s, \\
	x^i_{r,s} & \mbox{if } r=s, \\
	\frac{x_{s,r}^i-\ii y_{s,r}^i}{\sqrt{2}} & \mbox{if } r>s.
\end{array}
\right.
$$
Endowing $\cA_N$ with the involution ${(a_{i,j})_{1\leq i,j\leq N}}^* = (a_{j,i}^*)_{1\leq i,j\leq N}$ and the trace $\tau_N : A \mapsto\tau\left(\tr_N(A)\right)$, one has the following result.

\begin{proposition}
	\label{dojvdsomv}
	With the trace and the involution defined as above, $\cA_N$ is a $\mathrm{W}^*$-probability space. Moreover, the family $x^N=(x_1^N,\ldots,x_d^N)$ is a free semicircular system, and it is free from the $\ast$-subalgebra $\M_N(\C) \subseteq \cA_N$.
\end{proposition}

See \cite[Proposition 2.3]{parraud2025free} in the $\mathcal{C}^*$ case;
the $\mathrm{W}^*$ case is a simple extra argument.
In the sequel, we will drop the $N$ and simply write $x_i$ instead $x_i^N$.
Thanks to the previous proposition, the trace of a noncommutative polynomial in $x^N$ does not depend on $N$.
In particular, thanks to Equation \eqref{oidscjosc}, if $y^N$ is defined just like $x^N$ but with a family of semicircular variables free from those used to build $x^N$, then $x^N$ and $y^N$ have the same noncommutative distribution, and thanks to Equation \eqref{oidscjosc}, with $Z^N$ a family of matrices,

\begin{align}
\label{sgwlefnnf}
    \tau_N\left( P(x^N,Z^N) E_{a,b} \right) \tau_N\left( P(x^N,Z^N) E_{c,d} \right) &= \frac{1}{N^2} \tau\left( e_b^* P(x^N,Z^N) e_a \right) \tau\left( e_d^* P(x^N,Z^N) e_c \right) \\
    &= \frac{1}{N} \tau_N\left( e_b^* P(x^N,Z^N) e_a e_d^* P(y^N,Z^N) e_c \right) \nonumber \\
    &= \frac{1}{N} \tau_N\left( P(x^N,Z^N) E_{a,d} P(y^N,Z^N) E_{c,b} \right). \nonumber
\end{align}

\subsection{Random matrix theory}

In this subsection, we introduce the random matrix ensembles of interest for this paper and state a few useful properties about them. We begin with the definition of a Gaussian Unitary Ensemble (GUE) random matrix, which will be used in much of the rest of the paper.

\begin{definition}
	\label{3GUEdef}
	A GUE random matrix $X^N$ of size $N$ is an $N \times N$ Hermitian matrix whose coefficients are random variables with the following laws:
	\begin{enumerate}[font=\normalfont,label=(\roman*)]
		\item For $1\leq i\leq N$, the random variables $\sqrt{N} X^N_{i,i}$ are independent, centered Gaussian random variables of 
		variance $1$.
		\item For $1\leq i<j\leq N$, the random variables $\sqrt{2N} \cRe{X^N_{i,j}}$ and $\sqrt{2N} \cIm{X^N_{i,j}}$ are independent, 
		centered Gaussian random variables of variance $1$, independent of  $\big(X^N_{1,1},\ldots,X^N_{N,N}\big)$.
	\end{enumerate}
\end{definition}

The aim of this paper is to study more general random matrix ensembles.
Thus, we need to introduce some additional definitions.
Beware:
Some of the objects below, e.g., the space $\cC^k(X_1,\dots,X_d)$ of noncommutative smooth functions (see Definition \ref{def.NCsmoothbasic}), are introduced later.

\begin{definition}
    $W\in\cC^1(X_1,\dots,X_d)$ is said to be self-adjoint if for any $\sigma$-WOT--separable $\mathrm{W}^*$-probability space $(\cA,\tau)$ and all self-adjoint elements $x_1,\dots,x_d\in \cA$, $W(x)$ is also self-adjoint.
\end{definition}

We later introduce a number $\norm{W}_{1,\infty}$ measuring the ``size of the first derivative'' of a function $W \in \cC^1(X_1,\ldots,X_d)$.
Thanks to the discussion after Definition \ref{def.CkLip}, there is some $C < \infty$ such that if $\norm{W}_{1,\infty} < \infty$ and $H = (H_1,\ldots,H_d) \in \MnC_{\sa}^d$ is a $d$-tuple of $N \times N$ Hermitian matrices, then
$$ \norm{W(H) - W(0)} \leq C \norm{W}_{1,\infty} \norm{H}, $$
where $\norm{H} = \max\left\{\norm{H_i} : i=1,\ldots,d\right\}$.
Consequently, for such a $W$ that is also self-adjoint, if $X^N$ is a $d$-tuple of GUE random matrices, there exists a constant $Z \in (0,\infty)$ such that
\begin{align}
\label{slkdnvspm}
    \int_{\M_N(\C)_{\sa}^d} e^{-N \left( \frac{1}{2} \sum_{i=1}^d \Tr(H_i^2) + \Tr(W(H))\right)} dH &= Z\ \E\left[ e^{-N \Tr(W(X^N))} \right] \\ \nonumber
    &\leq Z e^{-N \Tr(W(0))} \E\left[ e^{C \norm{W}_{1,\infty} N^2 \norm{X^N}} \right], \\ \nonumber
    &= Z e^{-N \Tr(W(0))} \int_0^{\infty} \P\left( \norm{X^N} \geq \frac{\ln(t)}{C \norm{W}_{1,\infty} N^2} \right) dt,\numberthis \label{eq.layercake}
\end{align}
where $dH$ denotes integration against the Lebesgue measure on $\M_N(\C)_{\sa}^d$.
Now, thanks to, for example, \cite[Lemma 2.3.3]{AGZ2009} and the fact that $\sup\left\{\E\big[\norm{Y^N}\big] : N \in \N\right\} < \infty$ if $Y^N$ is a GUE random matrix (see \cite[Lemma 3.3.2]{AGZ2009}), there exists a constant $L$ such that for any GUE random matrix $X^N$ of size $N$ and any $\delta\geq 0$,
\begin{equation}
    \label{lvndlsvosndn}
    \P\left( \norm{X^N} \geq L+\delta \right) \leq 2 e^{- \frac{\delta^2 N}{2}}.
\end{equation}
Consequently, the quantity in \eqref{eq.layercake} is finite and thus so is the left-hand side of \eqref{slkdnvspm}.
Thus, we may define the following random matrix ensemble.

\begin{definition}
    \label{def:randompot}
    Let $W\in\cC^1(X_1,\dots,X_d)$ be self-adjoint and such that $\norm{W}_{1,\infty}$ is finite. We define the potential $V(X) = \frac{1}{2} \sum_{i=1}^d X_i^2 + W(X)$ and the probability measure $\mu_V^N$ on $\M_N(\C)_{\sa}^d$ by
    $$ d\mu_V^N(x) = \frac{e^{-N \left( \tr_N(V(x))\right)}}{Z_V^N} dx, $$
    where $dx$ is the Lebesgue measure on $\M_N(\C)_{\sa}^d$ and $Z_V^N$ is the (finite) number on the left-hand side of \eqref{slkdnvspm}. A $d$-tuple of random matrices with joint law $\mu_V^N$ will be referred to as a \textbf{multimatrix model with potential $V$}.
\end{definition}

Note that a multimatrix model with $W=0$ is just a $d$-tuple of independent GUE random matrices  Finally, in order to study multimatrix models, we will need an important notion of regularity for the potential $V$. 

\begin{definition}
\label{definition:kregular}
    We say $V$ is a $k$-regular potential if $V\in\cC^{4k+7}(X_1,\dots,X_d)$ is self-adjoint and $W\coloneqq V-\frac{1}{2}\sum_{i=1}^d X_i^2$ is such that
    $$\kappa \coloneqq \max_{1\leq p\leq d}\ \left\{ \sum_{i=1}^d\left( \norm{\partial_i\bD_pW}_{0,\infty} + \norm{\widetilde{\partial}_i \bD_p W }_{0,\infty} \right)\right\} < \frac{1}{4k+5} \; \text{ and } \; \max_{1\leq j\leq 4k+7} \norm{W}_{j,\infty} <\infty, $$ 
    where the $\partial_i$ and $\tilde{\partial}_i$ operations are defined in the next section, $\bD_p \coloneqq \cD_p + \id\otimes\tr \circ \widetilde{\partial}_p$ is defined with the help of the operators introduced in Proposition \ref{proppourvoila}, and $\norm{\cdot}_{0,\infty}$ is given by Definition \ref{def.norms-of-derivatives}.
\end{definition}

\section{Noncommutative smooth functions and derivatives} \label{sec.NCfunc}

For the duration of this section, fix a complex Hilbert space $H$, an orthonormal basis $\cB$ of $H$, a family $c = (c_e)_{e \in \cB} \in [1,\infty)^{\cB}$ of weights, and the set $\W$ of isomorphism classes of $\sigma$-WOT--separable $\mathrm{W}^*$-probability spaces.

\subsection{Preliminary definitions and an important example}

Recall that $\C \la \cB \ra$ is the set of noncommutative polynomials in the indeterminates $X = (X_e)_{e \in \cB}$.

\begin{notation}\label{nota.Hc,seminormAR}
Let $m \in \N$, $\ell \in \N_0 = \N \cup \{0\}$, $R > 0$, and $(\cA,\tau)$ be a $\mathrm{W}^*$-probability space.
\begin{enumerate}[label=(\roman*),font=\normalfont]
    \item Write
    \[
    \norm{h}_c \coloneqq \sum_{e \in \cB}\left|\ip{e,h}\right|c_e \in [0,\infty] \qquad (h \in H),
    \]
    and $H_c \coloneqq \{h \in H : \norm{h}_c < \infty\}$.
    Note that $H_c$ is a Banach space with the norm $\norm{\cdot}_c$.\label{item.Hc}
    \item Write
    \[
    \norm{x}_{\cA,c} \coloneqq \sup_{e \in \cB}\frac{\norm{x_e}}{c_e} \in [0,\infty] \qquad \left(x = (x_e)_{e \in \cB} \in \cA^{\cB}\right),
    \]
    and $\cA_c^{\cB} \coloneqq \big\{ x\in \cA^{\cB} : \norm{x}_{\cA,c} < \infty\big\}$.
    Note that $\cA_c^{\cB}$ is a Banach space with the norm $\norm{\cdot}_{\cA,c}$.
    Also, write $\cA_{c,\sa}^{\cB} \coloneqq \big\{x \in \cA_c^{\cB} : x_e \in \cA_{\sa}$ for all $e \in \cB\big\}$, where $\cA_{\sa} = \{a \in \cA : a^*=a\}$.\label{item.AcX}
    
    \item Write $\cA_{c,\sa}^{\cB,d}$ for the sets of all $(x_1,\ldots,x_d) \in (\cA_{c,\sa}^{\cB})^d$ such that for all $i,j=1,\ldots,n$, $x_i$ and $x_j$ have the same noncommutative law.
    Endow $\cA_{c,\sa}^{\cB,d}$ with the norm
    \[
    \norm{x}_{\cA,c,d} \coloneqq \max_{1 \leq i \leq d}\norm{x_i}_{\cA,c}.
    \]
    
    \item Write $\cA^{\ell,m} \coloneqq H_c^{\potimes \ell}\potimes \cA^{\potimes m}$, where $\potimes$ is the Banach space projective tensor product.
    For a function $f \colon \cA_{c,\sa}^{\cB,d} \to \cA^{\ell,m}$ and $R>0$, write
    \[
    \norm{f}_{\cA,c,R} \coloneqq \sup \left\{\norm{f(x)}_{\cA^{\ell,m}} : x \in \cA_{\sa,c}^{\cB,d}, \; \norm{x}_{\cA,c,d} \leq R\right\} \in [0,\infty],
    \]
    where $\norm{\cdot}_{\cA^{\ell,m}}$ is the Banach space projective tensor norm on $\cA^{\ell,m}$.
    Note that we are suppressing $\ell$ and $m$ in the notation $\norm{\cdot}_{\cA,c,R}$.
    We shall often use the notation $\norm{\cdot}$ instead of $\norm{\cdot}_{\cA^{\ell,m}}$ when the context is clear. \label{item.seminormAcR}

    \item Write $\partial \colon H_c^{\otimes \ell} \otimes \C\la \cB \ra^{\otimes m} \to H_c^{\otimes (\ell+1)} \otimes \C\la \cB \ra^{\otimes( m+1)}$ for the unique linear map that sends the pure tensor $h_1 \otimes \cdots \otimes h_{\ell} \otimes P_1 \otimes \cdots \otimes P_{m}$ to
    \[
    h_1 \otimes \cdots \otimes h_{\ell} \otimes \sum_{e \in \cB} e \otimes \partial_e P_1 \otimes P_2 \otimes \cdots \otimes P_{m},
    \]
    where $\partial_e \colon \C\la \cB \ra \to \C \la \cB \ra \otimes \C \la \cB \ra$ is the free difference quotient or noncommutative derivative with respect to the indeterminate $X_e$, i.e., the linear map from $\C \la \cB \ra$ to $\C \la \cB \ra \otimes \C \la \cB \ra$ such that if $M\in \C \la \cB \ra$ is a monomial, then
    $$ \partial_e M = \sum_{M=AX_eB} A\otimes B. $$
    Also, define the \emph{cyclic derivative} $\cD\colon \C\la \cB \ra \to H_c \otimes \C\la \cB \ra$ by $\cD \coloneqq m\circ\partial$, where $m \colon H_c \otimes \C\la\cB\ra \otimes \C\la\cB \ra \to H_c \otimes \C\la\cB\ra$ is the unique linear map sending the pure tensor $h \otimes A \otimes B$ to $h\otimes BA$.
    \label{item.nablaXi}

    \item For $y \in \cA_c^{\cB}$, write
    \[
    \cA^{1,2}  \ni T \mapsto T\sh y \in \cA^{0,1} = \cA
    \]
    for the unique bounded linear map sending the pure tensor $h \otimes a_1 \otimes a_2$~to
    \[
    \sum_{e \in \cB}\la e,h\ra \,a_1y_ea_2.
    \]
    Similarly, write
    \[
    \cA^{1,2}  \ni T \mapsto T\sh (y\otimes 1) \in \cA^{0,2} = \cA \potimes \cA
    \]
    for the unique bounded linear map sending the pure tensor $h \otimes a_1 \otimes a_2$~to
    \[
    \sum_{e \in \cB}\la e,h\ra \,a_1y_e\otimes a_2.
    \]
    \label{sharpdefi}
    
\end{enumerate}
\end{notation}

\begin{remark}\label{rem.sharpwelldef}
The element $\sum_{e \in \cB}\la e,h\ra \,a_1y_ea_2 = a_1\big(\sum_{e \in \cB}\la e,h\ra \,y_e\big)a_2$ above is well defined because
\[
\sum_{e \in \cB} \|\la e,h\ra \,y_e\| \leq \norm{y}_{\cA,c} \sum_{e \in \cB} \left|\la e,h\ra\right| c_e = \norm{y}_{\cA,c}\norm{h}_c < \infty
\]
by definition of $\norm{\cdot}_{\cA,c}$ and $\norm{\cdot}_c$.
Also, it follows from this estimate that the operator norm of $T \mapsto T \sh y$ is at most $\norm{y}_{\cA,c}$.
Similar comments apply to the operation $T \mapsto T \sh (y \otimes 1)$.
\end{remark}

One of the most important choices of $H$ and $\cB$ in this paper is the $L^2$ space $H = L^2(\R_+)$ of $\R_+ = [0,\infty)$ with respect to the Lebesgue measure and the Haar basis $\cB$, defined below, along with the weights of interest to us.
These particular choices will enable us to study the (asymptotic) behavior of certain matrix- and operator-valued stochastic differential equations.

\begin{definition}
\label{def:HaarBasis}
Define
\begin{equation*}
	\psi(x) \coloneqq \begin{cases}
		1 & \text{if } 0\leq x< \frac12,\\
		-1 & \text{if } \frac12\leq x< 1, \\
		0 & \text{otherwise.}
	\end{cases}
\end{equation*}
Then, for $j \in \N_0$ and $k =0,\ldots,2^j-1$, set
$$\psi_{j,k}(x) \coloneqq 2^\frac{j}{2} \psi(2^jx-k) \qquad (x \geq 0).$$
Next, let $\{\phi_i : i \in \N\}$ be the set of functions in $L^2([0,1])$ given by $\phi_1 \coloneqq \1_{[0,1]}$ and $\phi_i \coloneqq \psi_{j,k}$ whenever $i=2^j +k+1$ for $j \in \N_0$ and $k=0,\ldots,2^j-1$.\footnote{$\{\phi_i : i \in \N\}$ is an orthonormal basis of $L^2([0,1])$, usually referred to as the \emph{Haar basis}.}
Finally, for each $i,n \in \N$, define
\begin{align*}
    e_i^n(x) & \coloneqq \phi_i(x-(n-1)) \qquad (x \geq 0),  \\
    c_{e_i^n} & \coloneqq \sqrt{1+\ln\left(i\right) + \ln\left(n\right)}.
\end{align*}
We call $\cB \coloneqq \{e_i^n : i,n\in \N\}$ the Haar basis of $H \coloneqq L^2(\R_+)$ and $c = (c_e)_{e \in \cB}$ the Brownian weights on $\cB$.
The set $\cB$ is, in fact, an orthonormal basis of $H = L^2(\R_+)$.
\end{definition}

Later, we shall explain how to view Brownian motion as a family of random variables indexed by the family $(\1_{[0,t]})_{t\in\R_+}$ of elements of $H = L^2(\R_+)$.
Due to the importance of $H_c \subseteq H$ in our operations and spaces, it will be crucial to our development to show that $\1_{[0,t]}$ belongs to $H_c$.

\begin{proposition}\label{prop.indicatorHaarweightnorm}
Let $H$, $\cB$, and $c = (c_e)_{e \in \cB}$ be as in Definition \ref{def:HaarBasis}.
For all $s,t \geq 0$, $\1_{[s,t]}\in H_c$.
Furthermore, for every $\alpha \in (0,1/2)$, there exists a constant $C_{\alpha} < \infty$ such that for all $t\geq s\geq 0$,
    $$ \norm{\1_{[s,t]}}_c \leq \sqrt{1+\ln\left(t+1\right)}\ C_{\alpha} \max\left\{t-s,(t-s)^{\alpha}\right\}.$$
\end{proposition}

\begin{proof}
    First, assume that $\lfloor t\rfloor = \lfloor s\rfloor$.
    In this case, if $\phi \coloneqq \1_{[s,t]}$, $n\neq \lfloor s \rfloor$, and $i \in \N$, then $\langle \phi, e_i^n\rangle =0$.
    Furthermore, if $j \in \N_0$ and $k=0,\ldots,2^j-1$, then
    \begin{itemize}
        \item either $s$ or $t$ belongs to $\left[\frac{k}{2^j}+\lfloor s \rfloor,\frac{k+1}{2^j}+\lfloor s \rfloor\right)$, in which case
        $$ \left|\left\langle \phi, e_{2^j+k+1}^{\lfloor s \rfloor}\right\rangle\right| = \left|\int_{\max\left\{\frac{k}{2^j},s\right\}}^{\min\left\{\frac{k+1}{2^j},t\right\}} e_{2^j+k+1}^{\lfloor s \rfloor}(s)\,ds \right| \leq 2^{\frac{j}{2}} \min\left\{2^{-j},t-s\right\} \leq 2^{-j\left(\frac{1}{2} -\alpha\right)} (t-s)^{\alpha},$$
        \item or else $\langle \phi, e_{2^j+k+1}^n\rangle=0$.
    \end{itemize}
    Consequently, 
    \begin{align*}
        \sum_{i,n=1}^{\infty} c_{e_i^n} \left|\langle \phi, e_i^n\rangle\right| &=  \sqrt{1+\ln\left(\lfloor s \rfloor +1\right)} \left|\left\langle \phi, e_1^{\lfloor s \rfloor} \right\rangle\right| \\
        &\quad\quad+ \sum_{j=0}^{\infty} \sum_{k=0}^{2^j-1} \sqrt{1+\ln\left(2^j+k+1\right)+\ln\left(\lfloor s \rfloor+1\right)} \left|\left\langle \phi, e_{2^j+k+1}^{\lfloor s \rfloor +1} \right\rangle\right| \\
        &\leq \sqrt{1+\ln\left(\lfloor s \rfloor+1\right)} \left( 1 + 2\sum_{j=0}^{\infty} \sqrt{\frac{1+\ln\left(2^{j+1}\right)}{2^{j(1-2\alpha)}} } \right) (t-s)^{\alpha}\\
    \end{align*}
    Thus, $\1_{[s,t]}\in H_c$, and there exists a universal constant $D_{\alpha} < \infty$ such that
    \[
    \norm{\1_{[s,t]}}_c \leq \sqrt{1+\ln(\lfloor s \rfloor +1)} D_{\alpha} (t-s)^{\alpha}.
    \]
    Now, in general,
    $$ \1_{[s,t]} = \1_{[s,\lceil s \rceil]} + \sum_{n=\lceil s \rceil+1}^{\lfloor t \rfloor} e_1^n + \1_{[\lfloor t \rfloor,t]}.$$
    Consequently, $\1_{[s,t]}\in H_c$, and 
    \begin{align*}
        \norm{\1_{[s,t]}}_c &\leq \sqrt{1+\ln\left(\lfloor s \rfloor+1\right)} D_{\alpha} (\lceil s \rceil-s)^{\alpha} \\
        &\qquad + \sum_{n=\lceil s \rceil+1}^{\lfloor t \rfloor} \sqrt{1+\ln\left(\lfloor n \rfloor+1\right)} D_{\alpha} + \sqrt{1+\ln\left(\lfloor t \rfloor+1\right)} D_{\alpha} (t-\lfloor t \rfloor)^{\alpha} \\
        &\leq 3\sqrt{1+\ln\left(t+1\right)} D_{\alpha} \max\left\{t-s,(t-s)^{\alpha}\right\},
    \end{align*}
    as desired.
\end{proof}

\subsection{The space of noncommutative trace polynomials}

\begin{definition}
    Define $\tr\left(\C \la \cB \ra\right)$ to be the vector space 
    $$ \bigslant{\C \la \cB \ra}{\spn\left\{ PQ-QP : P,Q\in \C \la \cB \ra \right\}}. $$
    The equivalence class in $\tr\left(\C\la\cB\ra\right)$ of an element $P \in \C\la\cB\ra$ is denoted by $\tr\left(P\right)$.
    Also, define $\Tr^0(\cB)$ to be the symmetric tensor algebra of $\tr\left(\C \la \cB \ra\right)$ modulo the relation $\tr\left(1\right) = 1$.
    Finally, define $\Tr^{\ell}(\cB) \coloneqq H_c^{\otimes \ell-1}\otimes \C \la \cB \ra^{\otimes \ell} \otimes \Tr^0(\cB)$ for all $\ell \in \N$.
\end{definition}

Let $\ell \in \N$ and $(\cA,\tau)$ be a $\mathrm{W}^*$-probability space.
Note that each element $Q \in \Tr^{\ell}(\cB)$ induces, via evaluation, a continuous function $Q^{\cA} \colon \cA_{c,\sa}^{\cB,\ell} \to \cA^{\ell-1,\ell}$ with $\norm{Q^{\cA}}_{\cA,c,R} < \infty$ for all $R > 0$. Indeed, suppose
\[
Q = h_1 \otimes \cdots \otimes h_{\ell-1} \otimes P_1 \otimes \cdots \otimes P_{\ell} \otimes \tr\left(R_1\right)\otimes \cdots\otimes \tr\left(R_m\right)
\]
for some $h_1,\ldots,h_{\ell-1} \in H_c$ and $P_1,\ldots,P_{\ell},R_1,\dots,R_m \in \C\la \cB \ra$.
For $x = (x_1,\ldots,x_{\ell}) \in \cA_{c,\sa}^{\cB,\ell}$, define
\[
Q^{\cA}(x) = Q(x) \coloneqq h_1 \otimes \cdots \otimes h_{\ell-1} \otimes P_1(x_1) \otimes \cdots \otimes P_{\ell}(x_{\ell}) \times \tau(R_1(x_1))\dots \tau(R_{m}(x_{1})).
\]
Since $x_i$ and $x_j$ have the same noncommutative law for all $i,j=1,\ldots,\ell$, it does not matter which $x_i$ we plug into $R_1,\ldots,R_{m}$.
We leave it to the reader to verify that $Q^{\cA}$ is continuous and $\norm{Q^{\cA}}_{\cA,c,R} < \infty$ for all $R > 0$. Finally, beware that we shall often suppress the superscript $\cA$ in $Q^{\cA}$ and simply write $Q$ in order to lighten the notational burden.

\begin{definition}[Noncommutative differential of a trace polynomial]

    We first set $ \partial : \C\la \cB\ra \to H_c\otimes \C\la \cB\ra^{\otimes 2} $ by
    $$ \partial P = \sum_{e\in\cB} e\otimes \partial_e P, $$
    and
    $$ \partial^{(s)} \coloneqq (\id_{H_c^{\otimes (s-1)}}\otimes \partial \otimes\id_{\C\la\cB\ra^{\otimes (s-1)}})\circ\cdots\circ (\id_{H_c}\otimes\partial\otimes\id_{\C\la\cB\ra})\circ \partial, $$
    and for $\ell\geq 1$, we define $\partial^{(s)}:\Tr^{\ell}(\cB)\to \Tr^{\ell+s}(\cB)$ by 
    \begin{align*}
        &\partial^{(s)} \left(h_1\otimes\cdots\otimes h_{\ell} \otimes P_1\otimes \cdots\otimes P_{\ell}\otimes \tr(R_1)\otimes \cdots\otimes \tr(R_m) \right)\\
        &\coloneqq h_1\otimes\cdots\otimes h_{\ell} \otimes (\partial^{(s)} P_1) \otimes P_2\otimes \cdots\otimes P_{\ell}\otimes \tr(R_1)\otimes \cdots\otimes \tr(R_m).
    \end{align*}
    We also define $\widetilde{\partial}$ by 
    \begin{align*}
        &\widetilde{\partial} \left(h_1\otimes\cdots\otimes h_{\ell} \otimes P_1\otimes \cdots\otimes P_{\ell}\otimes \tr(R_1)\otimes \cdots\otimes \tr(R_m) \right)\\
        &\coloneqq \sum_{i=1}^m h_1\otimes\cdots\otimes h_{\ell} \otimes \cD R_i \otimes P_1 \otimes P_2\otimes \cdots\otimes P_{\ell} \otimes \tr(R_1)\otimes \cdots \tr(R_{i-1})\otimes \tr(R_{i+1})\cdots\otimes \tr(R_m).
    \end{align*}
    And finally we set $\widetilde{\partial}^{(s)}\coloneqq \partial^{(s-1)}\circ\widetilde{\partial}$.
\end{definition}

\begin{remark}
    We will sometimes need the usual free different quotient with respect to the variable $X_e$, i.e., the map $\partial_e := \la e| \circ \partial \colon \Tr^1(\cB) \to \Tr^1(\cB)$, where $\langle e| (h\otimes A\otimes B) = \la e, h\ra A\otimes B$. Similarly, we define $\cD_e$. Finally, when $\cB = \{e_1,\dots,e_d\}$ is finite, we denote by $\partial_i$ the map $\partial_{e_i}$.
\end{remark}

Note that the operator $\widetilde{\partial}$ is well-defined.
Indeed, for any $P,Q\in\C\la\cB\ra$, $\cD(PQ)=\cD(QP)$, therefore the operator $\cD$ is well-defined on $\tr\left(\C \la \cB \ra\right)$.

\begin{definition}\label{def.NCsmooth}
Let $\cB_p$ be the orthonormal basis of the direct sum $H^p = H^{\oplus p}$ induced by the orthonormal basis $\cB$ of $H$, $k \in \N_0$, $R > 0$, and $Q \in \Tr^1(\cB_p)$.
Write
\[
\norm{Q}_{j,c,R} \coloneqq \sup_{(\cA,\tau) \in \mathbb{W}}\ \sup_{n\geq 0}\ \sup_{j_0+\dots+j_n=j} \norm{\Big(\widetilde{\partial}^{(j_n)}\circ\cdots\circ \widetilde{\partial}^{(j_1)}\circ \partial^{(j_0)}Q\Big)^{\cA}}_{\cA,c,R},
\]
where we view $\Big(\widetilde{\partial}^{(j_n)}\circ\cdots\circ \widetilde{\partial}^{(j_1)}\circ \partial^{(j_0)}Q\Big)^{\cA}$ as a function from $\cA_{c,\sa}^{\cB,p\times (j+1)}$ to $\cA$ where each of the $j+1$ tensorand is evaluated in different variables.
We then set
\[
\norm{Q}_{\cC^k,R} \coloneqq \sup_{0\leq j\leq k} \norm{Q}_{j,c,R}.
\]
Define $\cC^k_p(\cB)$ to be the Fr\'echet-space completion of $\Tr^1(\cB_p)$ with respect to the family $\big\{\norm{\cdot}_{\cC^k,R} : R > 0\big\}$ of seminorms and $\cC^{\infty}_p(\cB)$ to be the Fr\'echet-space completion of $\Tr^1(\cB_p)$ with respect to the family $\big\{\norm{\cdot}_{\cC^k,R} : R > 0, \, k \in \N\big\}$ of seminorms. And finally, we set $\cC^k(\cB)$ the union of every $\cC^k_p(\cB)$.
\end{definition}

Here is a useful concrete description of $\cC^k(\cB)$ for $k \in \N_0 \cup \{\infty\}$.
Consider the space $\cS$ of $\W$-tuples $\mathbf{f} = \big(\mathbf{f}^{\cA}\big)_{(\cA,\tau) \in \W}$, where
\[
\mathbf{f}^{\cA} = \Big( f_{j_0,\dots,j_n}^{\cA} \colon \cA_{c,\sa}^{\cB,d\times(j_0+\dots+j_n+1)} \to \cA^{j_0+\dots+j_n,j_0+\dots+j_n+1}\Big)_{0 \leq j_0+\dots+j_n \leq k},
\]
such that the functions $f_{j_0,\dots,j_n}^{\cA}$ are continuous and satisfy
\[
\sup_{(\cA,\tau) \in \W}\big\|f_{j_0,\dots,j_n}^{\cA} \big\|_{\cA,c,R} < \infty \qquad (R > 0).
\]
(We use the notation $f^{\cA}$ for the function corresponding to the choice $j=0$.)
This space is a Fr\'echet space with the topology induced by the seminorms
\[
\mathbf{f} \mapsto \sup_{(\cA,\tau) \in \W}\big\|f_{j_0,\dots,j_n}^{\cA} \big\|_{\cA,c,R} \qquad (0 \leq j_0+\dots+j_n \leq k).
\]
If $Q \in \C\la \cB_p \ra$ and $\mathbf{Q} = \big( \mathbf{Q}^{\cA}\big)_{(\cA,\tau) \in \W}$ is defined by
\[
\mathbf{Q}^{\cA} \coloneqq \bigg( \widetilde{\partial}^{(j_n)}\circ\cdots\circ \widetilde{\partial}^{(j_1)}\circ \partial^{(j_0)}Q^{\cA} \colon \cA_{c,\sa}^{\cB,d\times (j_0+\dots+j_n+1)} \to \cA^{j_0+\dots+j_n,j_0+\dots+j_n+1} \bigg)_{0 \leq j_0+\dots+j_n \leq k},
\]
where we assume that each tensorand is evaluated in a different set of variables indexed by $\cB_p$, then $\mathbf{Q} \in \cS$.
The space $\cC^k_p(\cB)$ is the closure of $\big\{\mathbf{Q} : Q \in \C\la \cB\ra \big\}$ in $\cS$. 
We therefore write
\begin{equation}
    \widetilde{\partial}^{(j_n)}\circ\cdots\circ \widetilde{\partial}^{(j_1)}\circ \partial^{(j_0)}f^{\cA} \coloneqq f_{j_0,\dots,j_n}^{\cA} \qquad \left(\mathbf{f} = \big(f^{\cA}\big)_{(\cA,\tau) \in \W} \in \cC^k(\cB)\right).\label{eq.partialXnotation}
\end{equation}
As a consequence of this description, it is clear that if $k_1 \geq k_2$, then $\cC^{k_1}(\cB) \hookrightarrow \cC^{k_2}(\cB)$ in such a way that there is no notation conflict in \eqref{eq.partialXnotation}. Finally, note that outside of this section, we will often suppress the superscript $\cA$ in $f^{\cA}$, and simply write $f$, in order to lighten notations.

\begin{remark} \label{rem:finitecase}
    Note that if $\cB$ is finite, i.e., $H=\C^d$ for some $d$, then the choice of weight $c=(c_e)_{e\in\cB}$ does not matter, in the sense that the resulting space $\cC^k(\cB)$ will not depend on $c$. Therefore, we introduce the following self-contained definition. Note that the space $\cC^k(X_1,\dots,X_d)$ defined below corresponds to the space $\cC^k_1(\cB)$ of Definition \ref{def.NCsmooth} with $\cB$ the standard basis of $\C^d$.
\end{remark}

\begin{definition}
    \label{def.NCsmoothbasic}
Let $\Tr^1(X_1,\dots,X_d)$ be the space of trace polynomials in $d$ variables, $k \in \N_0$, $R > 0$, and $Q \in \Tr^1(X_1,\dots,X_d)$. Write
\[
\norm{Q}_{j,R} \coloneqq \sup_{(\cA,\tau) \in \mathbb{W}}\ \sup_{n\geq 0}\ \sup_{j_0+\dots+j_n=j} \norm{\Big(\widetilde{\partial}^{(j_n)}\circ\cdots\circ \widetilde{\partial}^{(j_1)}\circ \partial^{(j_0)}Q\Big)^{\cA}}_{\cA,R},
\]
where
$$ \norm{f}_{\cA,R} = \sup \left\{ \norm{f(x_1,\dots,x_{j+1})}_{\cA^{j,j+1}} \colon x_a\in\cA, \norm{x_a}\leq R, \forall a,b, x_a\sim x_b \right\}, $$
where one says that $x\sim y$ if $x$ and $y$ have the same noncommutative law, and given $P_1,\dots,P_{j+1}$ trace polynomials, we set
$$ \left( P_1\otimes\dots\otimes P_{j+1}\right)(x_1,\dots,x_{j+1}) = P_1(x_1)\otimes\dots\otimes P_{j+1}(x_{j+1}).  $$
We then set
\[
\norm{Q}_{\cC^k,R} \coloneqq \sup_{0\leq j\leq k} \norm{Q}_{j,R},
\]
and we define $\cC^k(X_1,\dots,X_d)$ to be the Fr\'echet-space completion of $\Tr^1(X_1,\dots,X_d)$ with respect to the family $\big\{\norm{\cdot}_{\cC^k,R} : R > 0\big\}$ of seminorms and $\cC^{\infty}(X_1,\dots,X_d)$ to be the completion with respect to the family $\big\{\norm{\cdot}_{\cC^k,R} : R > 0, \, k \in \N\big\}$ of seminorms.
\end{definition}

At several points, we will need to directly reference the norms of individual derivatives of noncommutative smooth functions.  Note that, since the derivatives are functions which output values in a tensor product, they are not elements of the original space.  Hence, we also make the following definition.

\begin{definition} \label{def.norms-of-derivatives}
Given $\ell,m \in \N$ and a function $\cf$ such that 
\[
\mathbf{f} = \Big( \cf^{\cA} \colon \cA_{c,\sa}^{\cB,d} \to \cA^{\ell,m}\Big)_{(\cA,\tau)\in\W},
\]
we define
\[
\norm{\cf}_{0,c,R} \coloneqq \sup_{(\cA,\tau) \in \mathbb{W}}\ \norm{\cf^{\cA}}_{\cA,c,R}.
\]
Similarly, in the case where $c_e = 1$ for all $e$ (such as the case of finite $\cB$) discussed above), we write simply $\norm{\cdot}_{0,R}$ rather than $\norm{\cdot}_{0,c,R}$.  Note that the dependence on $\ell,m$ is suppressed in the notation.
\end{definition}

Thus, in particular, Definition \ref{def.NCsmooth} can be expressed as
\[
\norm{Q}_{j,c,R} = \sup_{n\geq 0}\ \sup_{j_0+\dots+j_n=j} \norm{\Big(\widetilde{\partial}^{(j_n)}\circ\cdots\circ \widetilde{\partial}^{(j_1)}\circ \partial^{(j_0)}Q\Big)}_{0,c,R}.
\]

Next, let us explain how the abstract derivatives in our space of noncommutative functions relate to ordinary Fr{\'e}chet derivatives.

\begin{proposition}[Interpretation of $\partial$ and $\#$]\label{prop.nablaXiinterp}
If $Q \in \Tr^1(\cB)$, $(\cA,\tau)$ is a tracial von Neumann algebra, then $Q^{\cA} \in C^{\infty}\big(\cA_{c,\sa}^{\cB}; \cA\big)$ with respect to $\norm{\cdot}_{\cA,c}$ and $\norm{\cdot}_{\cA^{0,1}}$, and
\begin{equation}
    DQ^{\cA}(x)[y] = \partial Q^{\cA}(x,x)\sh y + \tau\otimes\id_{\cA}\left( \widetilde{\partial}Q^{\cA} \sh (y\otimes 1) \right) \qquad \left(x, y \in \cA_{c,\sa}^{\cB}\right),\label{eq.polyderiv}
\end{equation}
where $D$ denotes the Fr\'echet derivative---in particular, $D^0Q^{\cA}=Q^{\cA}$ and,
$$ D^jQ^{\cA}(x)[y_1,\dots,y_j] = \lim_{t\to 0} \frac{D^{j-1}Q^{\cA}(x+ty_j)[y_1,\dots,y_{j-1}] - D^{j-1}Q^{\cA}(x)[y_1,\dots,y_{j-1}]}{t}. $$
Besides, there exists a universal constant $C_j$ such that for all $x\in \cA^{\cB}_{c,\sa}$ with $\norm{x}_{\cA,c}\leq R$,
\begin{equation}
    \label{eq:frechetdiff}
    \norm{D^jQ^{\cA}(x)[y]}_{\cA^{0,1}} \leq C_j \norm{Q}_{j,c,R} \norm{y_1}_{\cA,c}\dots \norm{y_j}_{\cA,c}.
\end{equation}
\end{proposition}

\begin{proof}
Recall that the operator $\#$ was defined in Notation \ref{nota.Hc,seminormAR}. For $y \in \cA$, we also write $u\in \cA \potimes \cA \mapsto u\sh y \in \cA$ for the bounded linear map sending $a \otimes b$ to $ayb$, and $u\in \cA \potimes \cA \mapsto u\sh (y\otimes 1) \in \cA \potimes \cA$ for the bounded linear map sending $a \otimes b$ to $ay\otimes b$. (In particular, the symbol $\sh$ may mean something different from one line to another.)
If $Q \in \C\la \cB \ra$, then there exists a finite set $F \subseteq \cB$ such that $Q \in \C\la F \ra$.
Thus, if $x,y \in \cA_{c,\sa}^{\cB}$, then 
\begin{align*}
    \frac{\d}{\d t}\Big|_{t=0} Q(x+ty) & = \frac{\d}{\d t}\Big|_{t=0} Q((x_e)_{e \in F} + t(y_e)_{e \in F}) \\
    & = \sum_{e \in F} \partial_eQ((x_e)_{e \in F},(x_e)_{e \in F})\sh y_e + \tau\otimes\id_{\cA}\left( \widetilde{\partial}_eQ((x_e)_{e \in F},(x_e)_{e \in F})\sh y_e\otimes 1\right) \\
    & = \sum_{e \in \cB} \partial_eQ(x,x)\sh y_e + \tau\otimes\id_{\cA}\left( \widetilde{\partial}_eQ(x,x)\sh y\otimes 1\right)
\end{align*}
because $\partial_e Q = 0$ for all $e \in \cB \setminus F$.
Also,
\[
\partial Q(x,x)\sh y = \sum_{e \in \cB} (e \otimes \partial_eQ(x,x))\sh y = \sum_{e \in \cB}\sum_{f \in \cB} \underbrace{\la f,e\ra}_{\delta_{ef}} \partial_eQ(x,x)\sh y_f = \sum_{e \in \cB} \partial_eQ(x,x)\sh y_e + \widetilde{\partial}_eQ(x,x)\sh y_e.
\]
And similarly,
$$ \tau\otimes\id_{\cA}\left( \widetilde{\partial}Q \sh (y\otimes 1) \right) = \sum_{e \in \cB} \tau\otimes\id_{\cA}\left( \widetilde{\partial}_eQ(x,x)\sh y\otimes 1\right). $$
Thus Equation \eqref{eq.polyderiv} holds. Since the map
\[
\cA_{c,\sa}^{\cB} \ni x \mapsto \Big(y \mapsto \partial Q(x,x)\sh y + \tau\otimes\id_{\cA}\left( \widetilde{\partial}Q \sh (y\otimes 1) \right) \Big) \in B\big(\cA_{c,\sa}^{\cB};\cA\big)
\]
is continuous, we conclude from \cite[Fact 1.73]{HJ2014} that $Q^{\cA} \in C^1\big(\cA_{c,\sa}^{\cB};\cA\big)$ and
\[
DQ^{\cA}(x)[y] = \partial Q(x,x)\sh y \qquad \left(x,y \in \cA_{c,\sa}^X\right),
\]
as desired.

To obtain the higher-order derivatives and in particular to prove \eqref{eq:frechetdiff}, we proceed by induction. Let $m_x$ be such that 
\begin{align}
\label{lskjdvnskvn}
    &m_x\left(e_1\otimes \dots\otimes e_{j-1}\otimes A_1\otimes\dots\otimes A_j \otimes \tr(P_1)\otimes\dots\otimes \tr(P_m)\right) \\
    &= \tau(A_1(x) y_{e_1}\dots A_{i_1}(x) y_{e_{i_1}})\dots \tau(A_{i_{l-1}+1}(x) y_{e_{i_{l-1}+1}} \dots A_{i_l}(x) y_{e_{i_l}} ) \times \nonumber \\
    &\quad A_{i_l+1}(x) y_{e_{i_l+1}} A_{i_l+2}(x)\dots y_{e_{j-1}} A_j(x) \tau(P_1(x))\dots \tau(P_m(x)), \nonumber
\end{align}
for some integer $l$. We now want to compute
\begin{equation}
\label{sldfnvslknv}
    \frac{\d}{\d t}\Big|_{t=0} m_{x+ty}\left( \widetilde{\partial}^{(j_n)}\circ\cdots\circ \widetilde{\partial}^{(j_1)}\circ \partial^{(j_0)}Q \right),
\end{equation}
where $j_0+\dots+j_n = j-1$. First, let us remark that
$$ \frac{\d}{\d t}\Big|_{t=0} \tau(A_1(x) y_{e_1}\dots A_n(x) y_{e_n}) = \sum_{s=1}^n \tau(A_1(x) y_{e_1}\dots \partial A_s(x)\sh y \dots A_n(x) y_{e_n}), $$
$$ \frac{\d}{\d t}\Big|_{t=0} A_1(x) y_{e_1}\dots y_{e_{n-1}} A_n(x) = \sum_{s=1}^n A_1(x) y_{e_1}\dots \partial A_s(x)\sh y \dots y_{e_{n-1}} A_n(x), $$
$$ \frac{\d}{\d t}\Big|_{t=0} \tau(P_1(x))\dots \tau(P_m(x)) = \sum_{s=1}^m \tau(P_1(x))\dots \tau(\partial P_s \sh y) \dots \tau(P_m(x)).$$
Thus with 
$$ m_x^1\left(A_1 \otimes \dots \otimes A_n \right) = \tau(A_1(x) y_{e_1}\dots A_n(x) y_{e_n}), $$
$$ m_x^2\left(A_1 \otimes \dots \otimes A_n \right) = A_1(x) y_{e_1}\dots y_{e_{n-1}} A_n(x), $$
$$ m_x^3\left(\tr(P_1) \otimes \dots \otimes \tr(P_m) \right) = \tau(P_1(x))\dots \tau(P_m(x)),$$
we get that for some maps $m^{a,b}$ of the form defined in Equation \eqref{lskjdvnskvn},
$$ \frac{\d}{\d t}\Big|_{t=0} m_{x+ty}^1\left(\widetilde{\partial}^{(n)}\tr(P) \right) = \sum_{s=1}^n m_x^{1,s}\left(\id_{\C\la\cB\ra^{\otimes s-1}}\otimes\partial\otimes \id_{\C\la\cB\ra^{\otimes n-s}} \circ \widetilde{\partial}^{(n)}\tr(P) \right), $$
$$ \frac{\d}{\d t}\Big|_{t=0} m_{x+ty}^2\left(\partial^{(n-1)} P \right) = \sum_{s=1}^n m_x^{2,s}\left(\id_{\C\la\cB\ra^{\otimes s-1}}\otimes\partial\otimes \id_{\C\la\cB\ra^{\otimes n-s}} \circ \partial^{(n-1)}P \right), $$
$$ \frac{\d}{\d t}\Big|_{t=0}m_{x+ty}^3\left(\tr(P_1) \otimes \dots \otimes \tr(P_m) \right) = m^{3,1}\left( \widetilde{\partial}\left( \tr(P_1) \otimes \dots \otimes \tr(P_m)\right) \right).$$
However, note that for $M$ any monomial,
\begin{align*}
    &\id_{\C\la\cB\ra^{\otimes s-1}}\otimes\partial\otimes \id_{\C\la\cB\ra^{\otimes n-s}} \circ \partial^{(n-1)}M \\
    &=\sum_{f,e_1,\dots,e_{s-1}} \sum_{M=A_1X_{e_1}\dots A_s X_{f} A_{s+1}\dots X_{e_{n-1}} A_{n+1}} e_{n-1}\otimes \dots\otimes e_1\otimes f\otimes A_1 \otimes \dots \otimes A_{n+1} \\
    &=\sum_{f,e_1,\dots,e_{s-1}} \sum_{M=A_1X_{e_1}\dots X_{e_{s-1}} A_s X_{f} A_{s+1} X_{e_s}\dots X_{e_{n-1}} A_{n+1}} \\
    &\quad\quad\quad\quad\quad\quad\quad\quad\quad\quad\quad\quad\quad T_s\left(e_{n-1}\otimes \dots e_s \otimes f\otimes e_{s-1} \dots\otimes e_1\otimes f\otimes A_1 \otimes \dots \otimes A_{n+1}\right) \\
    &= T_s\circ \partial^{(n)}M,
\end{align*}
where $T_s$ is an operator permuting tensorands. Thus by combining those equations one can compute \eqref{sldfnvslknv} as a finite linear combination of terms involving different maps $m_x$ defined in \eqref{lskjdvnskvn}, permutations operator, and 
$$ \widetilde{\partial}^{(i_n)}\circ\cdots\circ \widetilde{\partial}^{(i_1)}\circ \partial^{(i_0)}Q,$$
where $i_0+\dots+i_n=j$. Besides, $T_s$ as a linear operator on $\cA^{j,j+1}$ has norm $1$, and as long as $\norm{x}_{\cA,c}\leq R$,
$$ \norm{m_x(H)} \leq \norm{H}_{\cA,c,R} \norm{y_1}_{\cA,c}\dots \norm{y_j}_{\cA,c}. $$
Thus by combining those estimates, we get Equation \eqref{eq:frechetdiff} and we conclude that $Q^{\cA} \in C^k\big(\cA_{c,\sa}^{\cB};\cA\big)$ for any $k$.
\end{proof}

As in the polynomial case, the $\partial^{(j)}$ operators enable the calculation of the derivative(s) of ``functions'' in $\cC^k(\cB)$.

\begin{proposition} 
\label{skjodncskncv}
If $\mathbf{f} \in \cC^k_1(\cB)$, then $f^{\cA} \in C^k\big(\cA_{c,\sa}^{\cB};\cA\big)$, and
\begin{equation}
    Df^{\cA}(x)[y] = \partial f^{\cA}(x,x)\sh y + \tau\otimes\id_{\cA}\left( \widetilde{\partial}f^{\cA} \sh (y\otimes 1) \right) \qquad \left(x, y \in \cA_{c,\sa}^{\cB}\right), \label{idsncsncds}
\end{equation}
where $D$ denotes the Fr\'echet derivative---in particular, $D^0f^{\cA}=f^{\cA}$ and,
$$ D^jf^{\cA}(x)[y_1,\dots,y_j] = \lim_{t\to 0} \frac{D^{j-1}f^{\cA}(x+ty)[y_1,\dots,y_{j-1}] - D^{j-1}f^{\cA}(x)[y_1,\dots,y_{j-1}]}{t}. $$
Besides, there exists a universal constant $C_j$ such that for all $x\in \cA^{\cB}_{c,\sa}$ with $\norm{x}_{\cA,c}\leq R$,
\begin{equation}
    \label{eq:frechetdiff2}
    \norm{D^jf^{\cA}(x)[y]}_{\cA^{0,1}} \leq C_j \norm{f}_{j,c,R} \norm{y_1}_{\cA,c}\dots \norm{y_j}_{\cA,c}.
\end{equation}
\end{proposition}

\begin{proof}
Let $(Q_n)_{n \in \N}$ be a sequence in $\C\la \cB \ra$ converging to $\mathbf{f}$ in $\cC^k_1(\cB)$. Therefore, for any $(\cA,\tau) \in\W$, thanks to Equation \eqref{eq:frechetdiff}, the sequences $(D^jQ^{\cA}_n)_{n\in\N}$ for $j\in [0,k]$ converge uniformly on bounded sets.
The conclusion then follow from Proposition \ref{prop.nablaXiinterp} and \cite[Theorem 1.85]{HJ2014}.
\end{proof}

\begin{corollary} \label{cor:taylorineq}
If $f \in \cC^k_1(\cB)$, then there exists functions $f^j\in \cC^{k-j}(\cB^{j+1})$, and a constant $C_k$ such that if $(\cA,\tau)\in\W$, $y\in\cA^{\cB}_{c,\sa}$, 
\begin{equation}
    \label{vjnsvn}
    \sup_{x\in \cA_{\sa},\ \norm{x}_{\cA,c}\leq R}\ \norm{f(x+y) - f(x) - f^1(x,y) \dots - f^{k-1}(x,y,\dots,y) } \leq C_k \norm{f}_{k,c,R+\norm{y}_{\cA,c}} \left( \norm{y}_{\cA,c}\right)^{k}. 
\end{equation}
Besides, for all $x,y_1,\dots,y_j,z\in\cA^{\cB}_{c,\sa}$, $i\in [1,j]$,
$$ f^j(x,y_1,\dots,y_i+z,\dots,y_j) = f^j(x,y_1,\dots,y_i,\dots,y_j) + f^j(x,y_1,\dots,z,\dots,y_j).$$
And finally, if we define
$$    \norm{f^j}_{\cA,c,R,T} \coloneqq \sup \left\{\norm{f^j(x,y_1,\dots,y_j)}_{\cA^{\ell,m}} : x,y_1,\dots,y_j \in \cA_{\sa,c}^{\cB,d}, \; \norm{x}_{\cA,c,d} \leq R, \norm{y_i}_{\cA,c,d} \leq T\right\}.
$$
then by similarly defining $\norm{\cdot}_{j,c,R,T}$ as in Definition \ref{def.NCsmooth} but with $\norm{\cdot}_{\cA,c,R,T}$ instead of $\norm{\cdot}_{\cA,c,R}$, then for all $i\leq k-j$ there exists a constant $C_i$ such that
\begin{equation}
\label{dvjnslnvsdv}
    \norm{f^j}_{i,c,R,T} \leq C_i \norm{f}_{\cC^k,R} (T+1)^i,
\end{equation}
\end{corollary}

\begin{proof}
    Equation \eqref{vjnsvn} is a direct corollary of Equation \eqref{eq:frechetdiff2} and Taylor's formula in a Banach space which states that
    $$ f(x+y) = f(x) + Df(x)[y]+\dots + \frac{1}{k!} D^kf(x)[y,\dots,y] + \frac{1}{k!} \int_0^1 D^kf(x+ty)[y,\dots,y] (1-t)^kdt. $$
    Besides, as shown in the proof of Proposition \ref{prop.nablaXiinterp}, $Df^j$ is defined with the maps $m_x$ defined by
\begin{align*}
    &m_x\left(e_1\otimes \dots\otimes e_{j-1}\otimes A_1\otimes\dots\otimes A_j \otimes \tr(P_1)\otimes\dots\otimes \tr(P_m)\right) \\
    &= \tau(A_1(x) y_{e_1}\dots A_{i_1}(x) y_{e_{i_1}})\dots \tau(A_{i_{l-1}+1}(x) y_{e_{i_{l-1}+1}} \dots A_{i_l}(x) y_{e_{i_l}} ) \times \nonumber \\
    &\quad A_{i_l+1}(x) y_{e_{i_l+1}} A_{i_l+2}(x)\dots y_{e_{j-1}} A_j(x) \tau(P_1(x))\dots \tau(P_m(x)), \nonumber
\end{align*}
the differentials of order $j$, i.e 
$$ \widetilde{\partial}^{(j_n)}\circ\cdots\circ \widetilde{\partial}^{(j_1)}\circ \partial^{(j_0)}f, $$
as well as permutations operators. Therefore the fact that $f^j\in\cC^{k-j}(\cB^{j+1})$ is a consequence of Proposition \ref{proppourvoila}. 
\end{proof}

\subsection{Example of $\cC^k$-functions}

The aim of this subsection is to give concrete example of $\cC^k$-functions, which combined with other results of this paper will provide a large array of functions that we can study.

\begin{example}\label{ex.Sasfunc}
Let  $h \in H_c$ and $F \subseteq \cB$ be a non-empty, finite set.
If
\[
P_{h,F}(X) \coloneqq \sum_{e \in F} \ip{e,h}\,X_e \in \C\ip{\cB} \; \text{ and } \; h_F \coloneqq \sum_{e \in F} \ip{e,h} \,e \in H_c,
\]
then
\[
\widetilde{\partial}^{(j_n)}\circ\cdots\circ \widetilde{\partial}^{(j_1)}\circ \partial^{(j_0)} P_{h,F}(X) = \begin{cases}
    \partial P_{h,F}(X) = h_F \otimes 1 \otimes 1 & \text{if } j_0=1 \text{ and } n=0, \\
    0 & \text{else}.
\end{cases}
\]
Consequently, if $k \in \N$, $\emptyset \neq F_0 \subseteq F$, and $R > 0$, then
\begin{align*}
    \norm{P_{h,F} - P_{h,F_0}}_{\cC^k,R} & = \norm{P_{h,F} - P_{h,F_0}}_{0,c,R} + \norm{P_{h,F} - P_{h,F_0}}_{1,c,R} \\
    & \leq R\sum_{e \in F \setminus F_0} |\ip{e,h}| \,c_e + \norm{h_F-h_{F_0}}_{H_c} \\
    & = (R+1)\sum_{e \in F \setminus F_0} |\ip{e,h}|\,c_e.
\end{align*}
Since $h \in H_c$, it follows that $F \mapsto P_{h,F}$ is a Cauchy net in $\cC^{\infty}(\cB)$.
Its limit is denoted by $\mathbf{f}_h \in \cC^{\infty}(\cB)$, and it is clear from the calculations above that if $(\cA,\tau) \in \W$, then
\begin{align*}
    f_h^{\cA}(x) & = \sum_{e \in \cB} \ip{e,h} \, x_e \in \cA \; \text{ for all } x \in \cA_{c,\sa}^{\cB}, \; \text{ and} \\
    \widetilde{\partial}^{(j_n)}\circ\cdots\circ \widetilde{\partial}^{(j_1)}\circ \partial^{(j_0)}f_h^{\cA} & \equiv \begin{cases}
    h \otimes 1 \otimes 1 & \text{if } j_0=1 \text{ and } n=0, \\
    0 & \text{else}.
\end{cases}
\end{align*}
Furthermore, $\norm{\mathbf{f}_h}_{\cC^k,R} \leq (R+1)\norm{h}_{H_c}$ for all $R>0$ and $k \in \N_0$.
In particular, the linear map $H_c \ni h \mapsto \mathbf{f}_h \in \cC^{\infty}(\cB)$ is continuous.

This example will be crucial in a couple of parts of our development;
see, e.g., Section \ref{sec.SDE}.
\end{example}

\begin{example}\label{ex.functionalcalculus}
Suppose $H = \C$ and therefore $\cB$ only has one element. Also, take $c = c_1 \coloneqq 1$ for the weight(s).
Now, let $f \colon \R \to \C$ be a continuous function.
For each unital $\mathrm{C}^*$-algebra $\cA$, $f$ may be promoted to a map $f^{\cA} \colon \cA_{\sa} \to \cA$ defined via the continuous functional calculus, and it is easy to see that
\[
\mathbf{f} \coloneqq \left(f^{\cA}\right)_{(\cA,\tau) \in \W} \in \cC^0(X_1).
\]
\end{example}

More generally, this can be improved to the case of functions differentiable in the usual sense of the term.

\begin{proposition}
    Let $f \colon \R \to \C$ be a function $k+1$ times differentiable in the usual sense of the term. Then $f$ may be promoted to a map $f^{\cA} \colon \cA_{\sa} \to \cA$ defined via the continuous functional calculus. Besides if $f^{(j)}$ is the $j$-th differential in the classical sense, assuming that $f\in L^1(\R)\cap L^2(\R) $ and $f^{(k+1)}\in L^2(\R)$, then
    \[
    \mathbf{f} \coloneqq \left(f^{\cA}\right)_{(\cA,\tau) \in \W} \in \cC^k(X_1).
    \]
\end{proposition}

\begin{proof}

We define 
$$ \widehat{f}\colon y\mapsto \frac{1}{2\pi} \int_{\R} f(x) e^{\ii x y} dx. $$
Note that for all $j\in [0,k]$, there exists constants $C_1,C_2$ such that
\begin{align*}
\int_{\R} |y|^{j} |\widehat{f}(y)| dy &\leq C_1 \int_{\R} \frac{1 +|y|^{k+1}}{1+|y|} |\widehat{f}(y)| dy \\
&= C_1 \int_{\R} \frac{ |\widehat{f}(y)| +|\widehat{f^{(k+1)}}(y)|}{1+|y|} dy \\
&\leq C_2\left( \norm{\widehat{f}}_{L^2(\R)} + \norm{\widehat{f^{(k+1)}}}_{L^2(\R)} \right) \\
&= C_2\left( \norm{f}_{L^2(\R)} + \norm{f^{(k+1)}}_{L^2(\R)} \right),
\end{align*}
where we used Plancherel's theorem in the last line. Consequently, if $x=(x_0,\dots,x_j)$, then one can define
\begin{align*}
    \partial^{(j)} f^{\cA} \colon x\mapsto \int_{\R} \partial e^{\ii y X}(x)\ \widehat{f}(y)\ dy,
\end{align*}
where
\begin{align*}
    \partial e^{\ii y X}(x) = (\ii y)^j \int_{[0,1]^j}\alpha_1^{j-1}\alpha_2^{j-2}\dots \alpha_{j-1}\  & e^{\ii\alpha_j\dots \alpha_1 y x_0}\otimes e^{\ii(1-\alpha_{j})\alpha_{j-1}\dots \alpha_1 y x_1}\otimes \cdots \\
&\cdots \otimes e^{\ii(1-\alpha_2)\alpha_1 y x_{j-1}} \otimes e^{\ii(1-\alpha_1)y x_j}\ d\alpha_1\dots d\alpha_j
\end{align*}
and if $n>0$,
$$ \widetilde{\partial}^{(j_n)}\circ \cdots \circ\widetilde{\partial}^{(j_1)}\circ \partial^{(j_0)} f^{\cA} = 0.$$
Note that since $H=H_c=\C$, and $\C\otimes \cA \simeq \cA$, we can define $\partial^{(j)} f^{\cA}(x)$ as an element of $\cA^{\otimes j+1}$ instead of $H_c^{\otimes j}\otimes \cA^{\otimes j+1}$.
We have (see \cite[\S 4.5]{parraud2024spectrum}) that
\begin{align*}
    &\frac{(n_1-1)! (n_2-n_1-1)!\dots (n_j-n_{j-1}-1)! (n-n_j)! }{n!} \\
	&= \int_{[0,1]^j} \left(\prod_{i=1}^j\alpha_i^{j-i}\right) \prod_{i=1}^j \left((1-\alpha_i)\prod_{s=1}^{i-1}\alpha_s\right)^{n_i-n_{i-1}-1} (\alpha_1\dots\alpha_j)^{n-n_j} d\alpha_1\dots d\alpha_j.
\end{align*}
Therefore,
\begin{align*}
    \sum_{k\geq 0} \frac{(\ii y)^n (\partial^{(j)}X^n)(x)}{n!} = \partial e^{\ii y X}(x).
\end{align*}
Consequently, if we set
$$ Q_m^y = \sum_{0\leq n\leq m} \frac{(\ii y X)^n}{n!}, $$
then for some constant $C$,
$$ \norm{Q_m^y - e^{\ii y X} }_{j,R} \leq C \frac{y^j e^{y R}}{(m-j)!}.$$
Hence by setting 
$$ Q_m = \int_{[-\sqrt{m},\sqrt{m}]} Q_m^y \widehat{f}(y)\ dy, $$
with the help of the fact that for all $j\in [0,k]$,
$$ \lim_{m\to\infty} \int_{|y|\geq \sqrt{m}} |y|^j |\widehat{f}(y)|\ dy = 0 $$
we have that for all $R>0$,
$$ \lim_{m\to\infty} \norm{\mathbf{f}-Q_m}_{\cC^k,R} =0.$$
Hence the conclusion.
\end{proof}

\subsection{Stability of the noncommutative smooth functions under various operations}

The aim of this subsection is to prove that our space $\cC^k(\cB)$ is stable under different operations, such as the composition or taking a conditional expectation. We begin by a proposition describing several simple but important operations.

\begin{proposition}
\label{proppourvoila}
 Consider the following maps:
    
    \begin{enumerate}[font=\normalfont,label=(\roman*)]
    
        \item $T_{\sigma,\nu}:\cA^{\ell,m}\mapsto\cA^{\ell,m}$ for $\sigma,\nu$ permutations, such that 
        $$ T_{\sigma,\nu}(h_1\otimes\cdots\otimes h_l\otimes a_1\otimes \cdots \otimes a_m) = h_{\sigma(1)}\otimes\cdots\otimes h_{\sigma(l)}\otimes a_{\nu(1)}\otimes\cdots\otimes a_{\nu(m)}, $$
        
        \item $\langle f|:\cA^{\ell,m}\mapsto\cA^{\ell-1,m}$ such that 
        $$\langle f|(h_1\otimes\cdots\otimes h_l\otimes a_1\otimes\cdots\otimes a_m) = \langle f| h_1\rangle h_2\otimes\cdots\otimes h_l \otimes a_{1}\otimes\cdots\otimes a_{m}, $$
        
        \item given $\cU$ a subset of $\cB$ and $p_{\cU}$ the orthogonal projection on $\spn(\cU)$, $\eta_{\cU}: \cA^{\ell,m}\mapsto\cA^{\ell-2,m}$ such that 
        $$\eta_{\cU}(h_1\otimes\cdots\otimes h_l\otimes a_1\otimes\cdots\otimes a_m) = \langle p_{\cU}(h_1)| p_{\cU}(h_2)\rangle h_3\otimes\cdots\otimes h_l\otimes a_{1}\otimes\cdots\otimes a_{m},$$

        \item $m:\cA^{\ell,m}\mapsto\cA^{\ell,m-1}$ such that 
        $$ m(h_1\otimes\cdots\otimes h_{\ell}\otimes a_1\otimes \cdots \otimes a_m) = h_1\otimes\cdots\otimes h_{\ell}\otimes a_1a_2\otimes a_3\otimes\cdots\otimes a_m, $$

        \item $\tr:\cA^{\ell,m}\mapsto\cA^{\ell,m-1}$ such that 
        $$ \tr(h_1\otimes\cdots\otimes h_{\ell}\otimes a_1\otimes \cdots \otimes a_m) = \tau(a_1) h_1\otimes\cdots\otimes h_{\ell}\otimes a_2\otimes\cdots\otimes a_m. $$

    \end{enumerate}

    \noindent If $\mathbf{f} \in \cC^k(\cB)$, and $\cJ \colon \cA^{j,j+1} \mapsto \cA$ is a linear map built out of the maps above, then given $j_0,\dots,j_n$ such that $j_0+\dots+j_n = j$, we have that $\cJ\circ \widetilde{\partial}^{(j_n)}\circ\cdots\circ \widetilde{\partial}^{(j_1)}\circ \partial^{(j_0)} \mathbf{f} \in \cC^{k-j}(\cB)$. Besides there exists a constant $C_{\cJ}$ such that for any $R>0$,
    \begin{equation}
    \label{skljdvnsdk}
        \norm{\cJ\circ \widetilde{\partial}^{(j_n)}\circ\cdots\circ \widetilde{\partial}^{(j_1)}\circ \partial^{(j_0)} \mathbf{f} }_{\cC^{k-j}(\cB),R} \leq C_{\cJ} \norm{\mathbf{f}}_{\cC^k(\cB),R}.
    \end{equation}
\end{proposition}

\begin{proof}

First, given $h_1,\dots,h_{s-1}\in H_c$, $A_1,\dots,A_s\in\C\la \cB_p\ra$ monomials, $r\in [1,s]$, let us define 
\begin{align*}
    &\partial^{[r]} (h_1\otimes\dots\otimes h_{s-1}\otimes A_1\otimes \dots \otimes A_s) \\
    &= \sum_{e\in \cB_p} \sum_{A_n = U X_e V} h_1\otimes\dots\otimes h_{s-1}\otimes e\otimes A_1\otimes \dots \otimes A_{r-1}\otimes U\otimes V \otimes A_{r+1} \otimes \dots \otimes A_s.
\end{align*}
Then, given $i_s\in [1,s]$ for $s\in [1,j]$, by induction we can find a permutation $\sigma$ such that for any $M\in \C\la \cB_p\ra$ a monomial, 
\begin{align*}
    \partial^{[i_j]}\circ \dots \circ \partial^{[i_1]} M &= \sum_{M=A_0X_{e_1}A_1 \dots X_{e_j} A_j} e_{\sigma(1)}\otimes \cdots\otimes e_{\sigma(j)}\otimes A_0\otimes \dots \otimes A_j \\
    &= T_{\sigma,\id} \circ \partial^{(j)}(M).
\end{align*}
Thus, if $i=i_0+\dots+i_m$, $j=j_0+\dots+j_n$, given a trace polynomial $P\in\Tr^1\la \cB_p\ra$, $i_0,\dots,i_m\in \N$, one can write 
$$ \widetilde{\partial}^{(i_m)}\circ\cdots\circ \widetilde{\partial}^{(i_1)}\circ \partial^{(i_0)} \left( \cJ\circ \widetilde{\partial}^{(j_n)}\circ\cdots\circ \widetilde{\partial}^{(j_1)}\circ \partial^{(j_0)} P\right)$$
as a finite linear combination of terms defined with the help of the first $i+j$ derivatives of $P$, as well as the maps $T_{\sigma,\nu}$, $\langle f|$, $\eta_{\cU}$, $ m$ and $\tr$. Since those maps are contractions for the projective tensor norm on $\cA^{\ell,m}$, given a net $(Q_{\lambda})_{\lambda\in\Lambda}$ of elements of $\Tr^1( \cB_p)$ converging towards $\mathbf{f}$, then $\left(\cJ\circ \widetilde{\partial}^{(j_n)}\circ\cdots\circ \widetilde{\partial}^{(j_1)}\circ \partial^{(j_0)}Q_{\lambda}\right)_{\lambda\in\Lambda}$ is a Cauchy net in $\cC_{dj}^{k-j}(\cB)$ therefore converges. 
\end{proof}

Next, we prove that taking the conditional expectation of a sufficiently smooth function also yields a $\cC^k(\cB)$-function. Note that since our definition of a smooth function is different from the usual one which only involves the Fr{\'e}chet derivatives (see for example \cite{JLS2022}), it makes the following proposition fairly non-trivial. In particular, it is the reason why we need to assume that in the Definition of $\cA_{c,\sa}^{\cB,d}$ (see Definition \ref{nota.Hc,seminormAR}) the noncommutative law of every $x_i$ is the same. 

\begin{proposition}
\label{prop:condexp}
    Given $f\in \cC^{3k}_p(\cB^1\cup \cB^2)$, with $\cB^1$ and $\cB^2$ Hilbert bases of $H_1$ and $H_2$. We set
    $y= (y_e)_{e\in\cB^2_p}$ a family of semicircular variables, if $\cB^2_p = (g_{e,i})_{e\in\cB^2,i\in [1,p]}$, we assume that 
    \begin{itemize}
        \item if $e\neq f$, then $y_{g_{e,i}}$ and $y_{g_{f,j}}$ are free for any $i,j$,
        \item otherwise, either $y_{g_{e,i}}=y_{g_{e,j}}$ for all $e\in\cB^2_p$, or $y_{g_{e,i}} $ and $y_{g_{e,j}}$ are free for all $e\in\cB^2_p$.
    \end{itemize}
    We define $\tr\big(f|\cB^1\big)$ the conditional expectation of $f$ knowing $y$ by
    $$ \left( \tr\big(f|\cB^1\big) \right)^{\cA} \colon x\in \cA_{c_1,\sa}^{\cB^1,p} \mapsto \tau\big(f^{\cA}(x,y)|\cA\big),$$
    where $y$ is taken to be free from $\cA$, and $\tau\big(\cdot|\cA\big)$ is the conditional expectation on $\cA$. Then $\tr\big(f|\cB^1\big)$ is an element of $\cC^k_p(\cB^1)$. Besides, there is a constant $C_{k,p}$ such that for all $R>0$,
    \begin{equation}
    \label{ksknvdojfbdo}
        \norm{\tr\big(f|\cB^1\big)}_{\cC^k(\cB^1),R} \leq C_{k,p} \norm{f}_{\cC^{3k}(\cB^1\cup\cB^2),R}.
    \end{equation}
\end{proposition}

\begin{proof}

    Let us first assume that $f$ is a polynomial. We begin by computing the differentials of $\tr\big(f|\cB^1\big)$. First, given $(\cA,\tau)$ a tracial von Neumann algebra, $M$ a monomial, if $g\in \cB^2_p$, then for all $a\in\cA$,
    $$ \tau\left( a M(x,y) y_g \right) = \sum_{M=AX_hB,\ h\in\cB^2_p} \tau\left( a A(x,y) \right) \tau\left( B(x,y) \right) \times \tau(y_hy_g), $$
    thanks to the Schwinger--Dyson equations. Thus,
    \begin{equation*}
        \tau\big(M(x,y)y_g|\cA\big) = \sum_{M=AX_hB,\ h\in\cB^2_p}\ \tau\big(A(x,y)|\cA\big)\ \tau\big(B(x,y)\big) \times \tau(y_hy_g),
    \end{equation*}
    which implies
    \begin{equation}
        \label{skljdvneb}
        \tr\big(M X_g|\cB^1\big) = \sum_{M=AX_hB,\ h\in\cB^2_p}\ \tr\big(A|\cB^1\big) \tr_1\big(B\big) \times \tau(y_hy_g),
    \end{equation}
    where $\tr_1 B = \tr\left( \tr\left( B | \cB^1 \right) \right)$. Note that if $g\in \cB^1_p$ then by definition of the conditional expectation $\tr\left( MX_g | \cB^1 \right) = \tr\left( M | \cB^1 \right)X_g$, thus when combined with Equation \eqref{skljdvneb}, this yields an explicit construction of $\tr\left( M | \cB^1 \right)$ as an element of $\Tr^1(\cB^1_p)$ (by induction on the degree of $M$).
    
    We then show by induction on the degree of $M$ that with $\id_{H_1}$ the projection on the vector space generated by $\cB_p^1$,
    \begin{equation}
        \label{sodvnsk}
        \partial \tr\big(M|\cB^1\big) = \left( \id_{H_1} \otimes \tr\big(\cdot|\cB^1\big)^{\otimes 2} \right) \circ\partial M
    \end{equation}
    If $M=1$, then the formula above is a tautology. Otherwise, if $M=RX_g$ with $g\in\cB_p^1$, then by definition of the conditional expectation, $\tau\big(M(x,y)|\cA\big) = \tau\big(R(x,y)|\cA\big) x_g$, thus 
    \begin{equation}
        \label{skjdvnsknv}
        \tr\big(M|\cB^1\big) = \tr\big(R|\cB^1\big) X_g,
    \end{equation}
    and thanks to our induction hypothesis, 
    \begin{align*}
        \partial \tr\big(M|\cB^1\big)(x) &= \left( \partial \left(\tr\left( M|\cB^1\right)\right) \times 1\otimes X_g + \id_{H_1}(g)\otimes \tr\left( R | \cB^1 \right) \right)(x) \\
        &= \left( \left( \left( \id_{H_1} \otimes \tr\big(\cdot|\cB^1\big)^{\otimes 2} \right) \circ\partial R\right) \times 1\otimes X_g + \id_{H_1}(g)\otimes \tr\left( R | \cB^1 \right) \right)(x) \\
        &= \id_{H_1} \otimes \tr\big(\cdot|\cB^1\big)^{\otimes 2} \circ\partial M.
    \end{align*}
    On the other hand, if $M=RX_g$ with $g\in\cB_p^2$, then thanks to Equation \eqref{skljdvneb},
    \begin{align*}
        \partial \tr\big(M|\cB^1\big) &= \sum_{R=AX_hB,\ h\in\cB^2_p}\ \partial\tr\big(A|\cB^1\big)\times \tr_1\big(B\big) \times \tau(y_hy_g) \\
        &= \sum_{R=AX_hB,\ h\in\cB^2_p}\ \id_{H_1} \otimes \tr\big(\cdot|\cA\big)^{\otimes 2} \circ\partial A \times \tr_1\big(B\big) \times \tau(y_hy_g) \\
        &= \sum_{R=A_1X_fA_2X_hB,\ f,h\in\cB^2_p}\ f\otimes \tr\big(A_1|\cB^1\big) \otimes \tr\big(A_2|\cB^1\big)  \times \tr_1\big(B\big) \times \tr(y_hy_g)\\
        &= \sum_{M=A_1X_fA_2,\ h\in\cB^2_p}\ f\otimes\tr\big(A_1|\cB^1\big) \otimes \tr\big(A_2|\cB^1\big) \\
        &= \id_{H_1} \otimes \tr\big(\cdot|\cB^1\big)^{\otimes 2} \circ\partial M.
    \end{align*}
    Hence the proof of Equation \eqref{sodvnsk}. Consequently, we have by induction that
    \begin{equation}
    \label{sdfjvnskvdn}
        \partial^{(s)} \tr\big(M|\cB^1\big) = \id_{H_1^{\otimes s}} \otimes \tr\big(\cdot|\cB^1\big)^{s+1} \circ\partial^{(s)} M.
    \end{equation}
    
    \noindent Besides, if $M\in\C\la \cB^1_p\ra$, one has that
    \begin{equation}
            \label{osjdgvn}
        \widetilde{\partial} \tr(M) = \sum_{M=AX_f B} BA = \cD M.
    \end{equation}
    We will now prove that if $M\in\C\la \cB^1_p,\cB_p^2\ra$,
    \begin{equation}
        \label{skjdvnsl0}
        \widetilde{\partial} \tr(M) = \id_{H_1}\otimes \tau\big(\cdot|\cA\big) \circ \cD M.
    \end{equation}
    We proceed by induction on the degree of $M$ in $(X_{g})_{g\in\cB_p^2}$. If this degree is equal to $0$, then it is simply Equation \eqref{osjdgvn}. Otherwise, by traciality one can assume that $M= R X_g $ with $g\in \cB_p^2$, and thanks to Equation \eqref{skljdvneb}, after taking the trace,
    \begin{align*}
        \tr_1\left( M \right) = \sum_{R=AX_hB,\ h\in\cB^2_p} \tr_1\left(A \right) \tr_1\left( B \right) \times \tau(y_hy_g).
    \end{align*}
    Thus, thanks to our induction hypothesis,
    \begin{align*}
        \widetilde{\partial} \tr_1\left( M \right) &= \sum_{R=AX_hB,\ h\in\cB^2_p} \widetilde{\partial}\tr_1\left(A\right) \times \tr_1\left( B \right) \times \tau(y_hy_g) + \tr_1\left(A \right)\times \widetilde{\partial}\tr_1\left( B \right) \times \tau(y_hy_g) \\
        &=  \sum_{R=AX_hB,\ h\in\cB^2_p} \id_{H_1}\otimes \tr\big(\cdot|\cB^1\big) \circ \cD A \times \tr_1\left( B\right) \times \tau(y_hy_g) \\
        &\quad\quad\quad\quad\quad\quad\quad + \tr_1\left( A \right)\times \id_{H_1}\otimes \tr\big(\cdot|\cB^1\big) \circ \cD B \times \tau(y_hy_g)\\
        &=  \sum_{f\in\cB_p^1} f\otimes \Bigg( \sum_{R=A_1X_fA_2X_hB,\ h\in\cB^2_p}  \tr\big(A_2A_1|\cB^1\big) \times \tr_1\left( B \right) \times \tau(y_hy_g) \\
        &\quad\quad\quad\quad\quad\quad + \sum_{R=AX_hB_1X_fB_2,\ h\in\cB^2_p} \tr_1\left( A\right)\times \tr\big(B_2B_1|\cB^1\big) \times \tau(y_hy_g) \Bigg) \\
        &=  \sum_{f\in\cB_p^1} \sum_{R=AX_fB} f\otimes \Bigg( \sum_{B=B_1X_hB_2,\ h\in\cB^2_p}  \tr\big(B_1A|\cB^1\big) \times \tr_1\left( B_2\right) \times \tau(y_hy_g) \\
        &\quad\quad\quad\quad\quad\quad\quad\quad\quad + \sum_{A=A_1X_hA_2,\ h\in\cB^2_p} \tr_1\left( A_1 \right)\times \tau\big(BA_2|\cB^1\big) \times \tau(y_hy_g) \Bigg) \\
        &= \id_{H_1}\otimes \tau\big(\cdot|\cA\big) \circ \cD M,
    \end{align*}
    where, given that $g\in\cB_p^2$, we used in the last line that 
    \begin{align*}
        \tr\big(BX_gA|\cB_1\big) =&\ \sum_{B=B_1X_hB_2,\ g\in\cB^2_p} \tr\big(B_1A|\cB^1\big) \tr_1\big(B_2\big) \times \tau(y_hy_g) \\
        &+  \sum_{A=A_1X_hA_2,\ g\in\cB^2_p} \tr_1\big(A_1\big) \tr\big(BA_2|\cB^1\big) \times \tau(y_hy_g).
    \end{align*}
    Note that the proof of this equation is the same as the one of Equation \eqref{skljdvneb}. Hence by induction, Equation \eqref{skjdvnsl0} is true. Therefore, we get thanks to Equation \eqref{sdfjvnskvdn} that,
    \begin{equation}
        \label{skjdvnsl}
        \widetilde{\partial}^{(s)} \tr_1(M) = \left(\id_{H_1^{\otimes s}}\otimes \tr\big(\cdot|\cB^1\big)^{\otimes s} \right) \circ \left(\widetilde{\partial}^{(s)} \tr( M)\right),
    \end{equation}
    where $\widetilde{\partial}^{(s)}$ on the right hand side is the differential in $\Tr^1\left(\cB^1_p,\cB^2_p\right)$. Finally, we are going to prove that 
    \begin{equation}
    \label{ksjvnsl}
        \widetilde{\partial}^{(s)} \tr\big(M|\cB^1\big) = m \circ\left( \eta\otimes\id \otimes \right(\widetilde{\partial}^{(s)}\circ \tr_1\left) \otimes \id \right)\circ \partial^{(2)} M,
    \end{equation}
    where $\eta(g\otimes h)$ is equal to $1$ if $g,h\in\cB_2$ and $\tau(y_hy_g)=1$, or $0$ otherwise, and 
    $$m(A\otimes B\otimes C) = B \otimes \left(\tr\left(A | \cB^1 \right) \tr\left(C | \cB^1 \right)\right) .$$
    Note that for a monomial $M$, Equation \eqref{ksjvnsl} is the same as 
    $$ \widetilde{\partial}^{(s)} \tr\big(M|\cB^1\big) = \sum_{g,h\in\cB_p^2} \tau(y_hy_g) \sum_{M=AX_gBX_hC} \left( \widetilde{\partial}^{(s)}\tr_1(B)\right)\otimes \tr\left( A | \cB^1 \right) \tr\left( C | \cB^1 \right). $$
    We will prove this equation by induction on the degree of $M$ with respect to $(X_g)_{g\in\cB_p^2}$. In particular, if this degree is equal to $0$ or $1$ then $\widetilde{\partial}^{(s)} \tr\big(M|\cB^1\big)=0$ and so is the right hand side of the equation above. Note thanks to Equation \eqref{skjdvnsknv}, we can assume that $M= R X_g $ with $g\in \cB_p^2$. Thus, thanks to Equation \eqref{skljdvneb}, by using our induction hypothesis, we have that 
    \begin{align*}
        &\widetilde{\partial}^{(s)} \tr\big(M|\cB^1\big) \\
        &= \sum_{M=AX_hBX_g,\ h,g\in\cB^2_p} \widetilde{\partial}^{(s)}\tr\big(A|\cB^1\big) \times \tr_1(B) \times \tau(y_hy_g) \\
        &\quad + \sum_{M=AX_hBX_g,\ h,g\in\cB^2_p} \widetilde{\partial}^{(s)}\tr_1(B)\otimes \tr\big(A|\cB^1\big) \times \tau(y_hy_g) \\
        &= \sum_{g,h,g',h'\in\cB_p^2} \tau(y_hy_g)\tau(y_{h'}y_{g'})  \sum_{M=A_1X_{h'}A_2X_{g'}A_3X_hBX_g} \widetilde{\partial}^{(s)} \tr_1(A_2) \otimes \Big( \tr\left( A_1 | \cB^1 \right) \tau\left( A_3 | \cB^1 \right) \Big) \times \tr_1(B) \\
        &\quad + \sum_{M=AX_hBX_g,\ h,g\in\cB^2_p} \widetilde{\partial}^{(s)}\tr_1(B)\otimes \tr\big(A|\cB^1\big) \times \tau(y_hy_g).
    \end{align*}
    Thus by using \eqref{skljdvneb}, we get that
    \begin{align*}
        \widetilde{\partial}^{(s)} \tr\big(M|\cB^1\big) =&\ \sum_{g,h\in\cB_p^2} \tau(y_hy_g) \sum_{M=AX_hBX_gC,\ C\neq 1} \widetilde{\partial}^{(s)}\tr_1(B) \otimes \tr\left( A | \cB^1 \right) \tr\left( C | \cB^1 \right) \\
        &+ \sum_{g,h\in\cB_p^2} \tau(y_hy_g) \sum_{M=AX_hBX_g} \widetilde{\partial}^{(s)}\tr_1(B)\otimes \tr\big(A|\cB^1\big) \tr\big(1|\cB^1\big).
    \end{align*}
    Hence, Equation \eqref{ksjvnsl} is proved. Note that thanks to our assumption on the freeness of the semicirculars $y_{e,i}$, the quantity $\kappa_{i,j}=\tau(y_{e,i}y_{e,j})$ does not depend on $e$, and
    $$\eta(v\otimes w) = \sum_{1\leq i,j\leq p} \kappa_{i,j} \sum_{e\in\cB^2} \langle g_{e,i} | v\rangle \langle g_{e,j} | w\rangle.$$
    In particular, there is a constant $C_p$ such that $|\eta(v\otimes w)|\leq C_p \norm{v}_2 \norm{w}_2$. Finally,
    if
    $$ M = R\otimes \tr P_1 \otimes \cdots \otimes \tr P_m,$$
    then by Equation  \eqref{skjdvnsl} and \eqref{ksjvnsl},
    $$ \widetilde{\partial}^{(s)} \tr\big( M | \cB^1 \big) = m \circ\left( \eta\otimes\id \otimes \right(\widetilde{\partial}^{(s)}\circ \tr_1\left) \otimes \id \right)\circ \partial^{(2)} M + \left(\id_{H_1^{\otimes s}}\otimes \tr\big(\cdot|\cB^1\big)^{\otimes s} \right) \circ \left(\widetilde{\partial}^{(s)} M\right).$$
    
    So by combining this result with Equation \eqref{sdfjvnskvdn} and \eqref{skjdvnsl} again, one can express the first $k$ differentials of $\tr\big(f|H_1\big)$ with the help of the first $3k$ differentials of $f$ as well as maps defined in Proposition \ref{proppourvoila}. The conclusion follows by a density argument. In particular, Equation \eqref{skljdvnsdk} yields Equation \eqref{ksknvdojfbdo}.
\end{proof}

Our last necessary tool will be a version of the chain rule adapted to our space of functions.

\begin{theorem}[Chain rule]
\label{chainrule}
Let $\mathbf{J} \in\cC^k(X_1,\dots,X_d)$ and $\mathbf{f}_1,\dots,\mathbf{f}_d\in\cC^k(\cB)_{\sa}$, and write $\mathbf{f} \coloneqq (\mathbf{f}_1,\ldots,\mathbf{f}_d)$. Let $\mathbf{g} \coloneqq \mathbf{J} \circ \mathbf{f}$ be defined by
\begin{align*}
    g^{\cA} &\coloneqq J^{\cA} \circ f^{\cA} = \big(x \mapsto J^{\cA}(f_1^{\cA}(x),\ldots,f_d^{\cA}(x))\big),
\end{align*}
then $\mathbf{J} \circ \mathbf{f} \in \cC^k(\cB)$.\footnote{More precisely, if $(P_n(X_1,\ldots,X_d))_{n \in \N}$ is a sequence of trace polynomials converging in $\cC^k(X_1,\ldots,X_d)$ to $\mathbf{J}$ and $(Q_n(X)) = (Q_{1,n}(X),\ldots,Q_{d,n}(X))_{n \in \N}$ is a sequence of $d$-tuples of self-adjoint trace polynomials converging in $\cC^k(\cB)_{\sa}^d$ to $\mathbf{f}$, then $(P_n(Q_n(X)))_{n \in \N}$ converges in $\cC^k(\cB)$ to a limit, $\mathbf{J} \circ \mathbf{f}$, independent of the choice of approximating sequences.}
Besides, if we define
$$ \Lambda^s_R(\cf) = \sup_{\substack{1\leq i_1,\dots,i_m\leq d \\ s_1+\dots+s_m = s \\ s_l\geq 1,\ m\geq 2 }}\ \prod_{l=1}^m\ \norm{\cf_{i_l}}_{s_l,c,R}\ ,\quad\quad S\coloneqq \max_{1\leq i\leq d} \norm{\cf_i}_{0,c,R}, $$
then there exists a constant $C_k$ such that, if $j_0>0$, $j\coloneqq j_0+\dots+j_n$,
\begin{align}
\label{chainestimate}
    &\norm{\widetilde{\partial}^{(j_n)}\circ\cdots\circ\widetilde{\partial}^{(j_1)}\circ\partial^{(j_0)}\mathbf{g} - \sum_{i=1}^d  (\partial_i \mathbf{J})\circ\cf\ \sh_1 \left(\widetilde{\partial}^{(j_n)}\circ\cdots\circ\widetilde{\partial}^{(j_1)}\circ\partial^{(j_0)}\mathbf{f_i} \right)}_{0,c,R} \\
    &\leq C_k\times \Lambda^j_R(\cf) \times \sup_{s\in [2, j]} \norm{\mathbf{J}}_{s,S}, \nonumber
\end{align}
and if $j_0=0$,
\begin{align}
\label{chainestimate2}
    &\Bigg\| \widetilde{\partial}^{(j_n)}\circ\cdots\circ\widetilde{\partial}^{(j_1)}\mathbf{g} - \sum_{i=1}^d  (\partial_i \mathbf{J})\circ\cf\ \sh_2 \left(\widetilde{\partial}^{(j_n)}\circ\cdots\circ\widetilde{\partial}^{(j_1)}\mathbf{f_i} \right) - \sum_{i=1}^d  (\widetilde{\partial}_i \mathbf{J} )\circ\cf\ \sh_3 \left(\widetilde{\partial}^{(j_n)}\circ\cdots\circ\widetilde{\partial}^{(j_1)} \mathbf{f_i} \right) \Bigg\|_{0,c,R} \\
    &\leq C_k\times \Bigg( \Lambda_R^j(\cf) \times \sup_{s\in [2, j]} \norm{\mathbf{J}}_{s,S} + \sup_i\ \norm{\widetilde{\partial}^{(j_n)}\circ\cdots\circ\widetilde{\partial}^{(j_2)} \circ \partial^{(j_1)}\mathbf{f_i} }_{0,c,R} \times \sup_i\ \norm{\widetilde{\partial}_i\mathbf{J}}_{0,S}\Bigg). \nonumber
\end{align}
Where if we set,
$$ \cX = A\otimes B \otimes \tr(R_1)\otimes \cdots\otimes \tr(R_m), $$
$$ \cY = h_1\otimes\dots\otimes h_{k-1}\otimes C_1\otimes \dots\otimes C_k \otimes \tr(R_{m+1})\otimes \cdots\otimes \tr(R_n), $$
then
\begin{align*}
    &\cX\sh_1\cY = h_1\otimes\dots\otimes h_{k-1}\otimes AC_1\otimes C_2\otimes\dots\otimes C_{k-1}\otimes C_k B \otimes \tr(R_1)\otimes \cdots\otimes \tr(R_n), \\
    &\cX\sh_2\cY = h_1\otimes\dots\otimes h_{k-1}\otimes C_1\otimes \dots\otimes C_{k-1}\otimes A C_k B \otimes \tr(R_1)\otimes \cdots\otimes \tr(R_n), \\
    &\cX\sh_3\cY = h_1\otimes\dots\otimes h_{k-1}\otimes C_1\otimes \dots\otimes C_{k-1} \otimes B \otimes \tr(AC_k) \otimes \tr(R_1)\otimes \cdots\otimes \tr(R_n).
\end{align*}
\end{theorem}

As a corollary of the proof of the theorem above, we also get the following.

\begin{corollary}
\label{ksvnslvslv}
With the notations of Theorem \ref{chainrule},
\begin{equation}
    \label{skvjnlsdvkn}
    \partial \cg = \sum_{i=1}^d \left( (\partial_i \mathbf{J})\circ \cf \right)\sh_1 \left(\partial \mathbf{f_i} \right).
\end{equation}
Besides, if given $ \cX = A\otimes B \otimes \tr(R_1)\otimes \cdots\otimes \tr(R_m)$  and $ \cY = h\otimes C_1 \otimes C_2 \otimes \tr(R_{m+1})\otimes \cdots\otimes \tr(R_n)$, we set
$$ \cX\sh_4\cY = h\otimes C_2AC_1 \otimes B \otimes \tr(R_1)\otimes \cdots\otimes \tr(R_n),$$
then
\begin{equation}
    \label{skvjnlsdvkn2}
    \widetilde{\partial} \cg = \sum_{i=1}^d \left( (\partial_i \mathbf{J})\circ \cf \right)\sh_2 \left(\widetilde{\partial} \mathbf{f_i} \right) + \sum_{i=1}^d \left( (\widetilde{\partial}_i \mathbf{J})\circ \cf \right)\sh_3 \left(\widetilde{\partial} \mathbf{f_i} \right) + \sum_{i=1}^d \left( (\widetilde{\partial}_i \mathbf{J})\circ \cf \right)\sh_4 \left(\partial \mathbf{f_i} \right).
\end{equation}

\end{corollary}

\begin{proof}[Proof of Theorem \ref{chainrule} and Corollary \ref{ksvnslvslv}] 
    Let us first assume that $\bJ,\cf_1,\dots,\cf_d$ are trace polynomials, one can define $\cg$ by using the following computation rules,
    $$ P\otimes \tr(R_1)\otimes\dots\otimes \tr(R_m)\ \times\ Q\otimes \tr(S_1)\otimes\dots\otimes \tr(S_n) = (PQ) \otimes \tr(R_1)\otimes\dots\otimes \tr(R_m) \otimes \tr(S_1)\otimes\dots\otimes \tr(S_n),$$
    $$ \tr\left( P\otimes \tr(R_1)\otimes\dots\otimes \tr(R_m) \right) = 1\otimes \tr(P) \otimes \tr(R_1)\otimes\dots\otimes \tr(R_m).$$
    First, let us prove Equation \eqref{skvjnlsdvkn} by induction. It is obvious if $\bJ=1$. Otherwise if $\bJ = X_s R$ with $R$ a monomial, then by using our induction hypothesis,
    \begin{align*}
        \partial \cg &= \partial\cf_s \times 1\otimes R\circ\cf + \cf_s\otimes 1\times \sum_{i=1}^d  \left( (\partial_i R)\circ \cf \right) \sh_1 \partial \mathbf{f_i} \\
        &= \sum_{i=1}^d \left( (\partial_i \mathbf{J})\circ \cf \right)\sh_1 \left(\partial \mathbf{f_i} \right).
    \end{align*}
    Hence the conclusion. Besides, if we now assume that $\cg$ is a trace polynomial then Equation \eqref{skvjnlsdvkn} still holds.
    
    Let us now prove Equation \eqref{skvjnlsdvkn2}. We first assume that $\bJ$ is a monomial, then 
    $$ \widetilde{\partial} \cg = \sum_{i=1}^d \left( (\partial_i \mathbf{J})\circ \cf \right)\sh_2 \left(\widetilde{\partial} \mathbf{f_i} \right), $$
    which we prove by induction on the degree of $\bJ$. Indeed, once again it is obvious if $\bJ=1$. Otherwise, if $\bJ = X_s R$ with $R$ a monomial, then by using our induction hypothesis,
    \begin{align*}
        \widetilde{\partial} \cg &= \widetilde{\partial}\cf_s \times 1\otimes R\circ\cf + 1\otimes \cf_s \times \sum_{i=1}^d \left( (\partial_i R)\circ \cf \right)\sh_2 \left(\widetilde{\partial} \mathbf{f_i} \right) \\
        &= \sum_{i=1}^d \left( (\partial_i \mathbf{J})\circ \cf \right)\sh_2 \left(\widetilde{\partial} \mathbf{f_i} \right).
    \end{align*}
    Let us now prove Equation \eqref{skvjnlsdvkn2} by induction on the number of traces. By writing $\bJ = R\otimes \tr(Q)$,
    \begin{align*}
        \widetilde{\partial} \cg &= (\widetilde{\partial}(R\circ \cf)) \otimes \tr(Q\circ\cf) + \widetilde{\partial}(\tr (Q\circ\cf))\otimes R\circ \cf \\
        &= (\widetilde{\partial}(R\circ \cf)) \otimes \tr(Q\circ\cf) + \sum_{i=1}^d ((\cD_i Q)\circ \cf\otimes 1)\sh_4 \partial \cf_i \times 1 \otimes R\circ \cf \\
        &\quad + \sum_{i=1}^d \left( (\cD_i Q)\circ \cf \otimes 1\right)\sh_3 \left(\widetilde{\partial} \mathbf{f_i} \right) \times 1 \otimes R\circ \cf \\
        &= \left( \sum_{i=1}^d \left( (\partial_i R)\circ \cf \right)\sh_2 \left(\widetilde{\partial} \mathbf{f_i} \right) + \sum_{i=1}^d \left( (\widetilde{\partial}_i R)\circ \cf \right)\sh_3 \left(\widetilde{\partial} \mathbf{f_i} \right) + \sum_{i=1}^d \left( (\widetilde{\partial}_i R)\circ \cf \right)\sh_4 \left(\partial \mathbf{f_i} \right)\right)\otimes \tr(Q\circ\cf) \\
        &\quad + \sum_{i=1}^d ((\cD_i Q)\circ \cf\otimes 1)\sh_4 \partial \cf_i \times 1 \otimes R\circ \cf + \sum_{i=1}^d \left( (\cD_i Q)\circ \cf \otimes 1\right)\sh_3 \left(\widetilde{\partial} \mathbf{f_i} \right) \times 1 \otimes R\circ \cf \\
        &=\sum_{i=1}^d \left( (\partial_i \mathbf{J})\circ \cf \right)\sh_2 \left(\widetilde{\partial} \mathbf{f_i} \right) + \sum_{i=1}^d \left( (\widetilde{\partial}_i \mathbf{J})\circ \cf \right)\sh_3 \left(\widetilde{\partial} \mathbf{f_i} \right) + \sum_{i=1}^d \left( (\widetilde{\partial}_i \mathbf{J})\circ \cf \right)\sh_4 \left(\partial \mathbf{f_i} \right).
    \end{align*}
    Hence the proof of Equation \eqref{skvjnlsdvkn2}.

    Thus by induction, we can express the higher order derivative of $\cg$ with the help of those of $\bJ$ and $\cf$ as well as the operators defined in Proposition \ref{proppourvoila}. In particular, thanks to Equation \eqref{skljdvnsdk}, we get Equations \eqref{chainestimate} and \eqref{chainestimate2}.
\end{proof}

This theorem yields the following corollaries.

\begin{corollary} \label{cor:complip}
    There exists a constant $C_k < \infty$ such that for all $\mathbf{K} \in \cC_1^{k+1}(X_1,\ldots,X_d)$, $R > 0$, and $\mathbf{f},\mathbf{g} \in \cC^k(\cB)_{\sa}^d$,
\[
\norm{\mathbf{K} \circ \mathbf{f} - \mathbf{K} \circ \mathbf{g}}_{\cC^k,R} \leq C_k(S+1)^k \sum_{j=1}^{k+1}\norm{\mathbf{K}}_{j,S} \times \norm{\mathbf{f} - \mathbf{g}}_{\cC^k,R},
\]
where
\begin{align*}
    & \norm{\mathbf{h}}_{\cC^k,R} \coloneqq \sum_{i=1}^d \norm{\mathbf{h}_i}_{\cC^k,R} \quad\qquad \big(\mathbf{h} = (\mathbf{h}_1,\ldots,\mathbf{h}_d) \in \cC_1^k(\cB)^d\big), \\
    & S \coloneqq \max\left\{ \norm{\mathbf{f}}_{\cC^k,R},\norm{\mathbf{g}}_{\cC^k,R} \right\}.
\end{align*}
Also, here, $\cC^k(\cB)_{\sa}$ is the set of $\mathbf{f} \in \cC_1^k(\cB)$ such that $f^{\cA}(x) \in \cA_{\sa}$ for all $(\cA,\tau) \in \W$ and $x \in \cA_{\sa}$.
\end{corollary}

\begin{proof}

First, note that as explained in the proof of Theorem \ref{chainrule}, by iterating the Equations obtained in Corollary \ref{ksvnslvslv}, we can express the differential of $\mathbf{K} \circ \mathbf{f}$ with the operators defined in Proposition \ref{proppourvoila}, as well as the differentials of  $\mathbf{K}$ and $\mathbf{f}$. More precisely, if $j\coloneqq j_0+\dots+j_n \geq 1$, one has that
$$ \widetilde{\partial}^{(j_n)}\circ\cdots\circ\widetilde{\partial}^{(j_1)}\circ\partial^{(j_0)}(\mathbf{K}\circ \mathbf{f}) $$
is a finite linear combination of terms of the form
$$ \cM\left( \left(\widetilde{\partial}^{(a_p)}\circ\cdots\circ\widetilde{\partial}^{(a_1)}\circ\partial^{(a_0)}\mathbf{K} \right) \circ \mathbf{f}, \left( \widetilde{\partial}^{(i_m)}\circ\cdots\circ\widetilde{\partial}^{(i_1)}\circ\partial^{(i_0)} \mathbf{f}_l\right)_{1\leq l\leq d,\ (i_0,\dots,i_m)\in\cF} \right), $$
where $\cM$ is a multilinear map defined with the operators of Proposition \ref{proppourvoila}, $a_0+\dots+a_p \in [1,j]$, and $\cF$ has at most $j$ elements. Note that since every maps defined in Proposition \ref{proppourvoila} is Lipschitz with respect to the projective tensor norm, we have that for some constant $C_{\cM}$,
\begin{align*}
    &\Bigg\|\cM\left( \left(\widetilde{\partial}^{(a_p)}\circ\cdots\circ\widetilde{\partial}^{(a_1)}\circ\partial^{(a_0)}\mathbf{K} \right) \circ \mathbf{f}, \left( \widetilde{\partial}^{(i_m)}\circ\cdots\circ\widetilde{\partial}^{(i_1)}\circ\partial^{(i_0)} \mathbf{f}_l\right)_{1\leq l\leq d,\ (i_0,\dots,i_m)\in\cF} \right) \\
    & - \cM\left( \left(\widetilde{\partial}^{(a_p)}\circ\cdots\circ\widetilde{\partial}^{(a_1)}\circ\partial^{(a_0)}\mathbf{K} \right) \circ \mathbf{g}, \left( \widetilde{\partial}^{(i_m)}\circ\cdots\circ\widetilde{\partial}^{(i_1)}\circ\partial^{(i_0)} \mathbf{g}_l\right)_{1\leq l\leq d,\ (i_0,\dots,i_m)\in\cF} \right) \Bigg\|_{0,c,R}\\
    &\leq C_{\cM} \norm{ \left(\widetilde{\partial}^{(a_p)}\circ\cdots\circ\widetilde{\partial}^{(a_1)}\circ\partial^{(a_0)}\mathbf{K} \right) \circ \mathbf{f} - \left(\widetilde{\partial}^{(a_p)}\circ\cdots\circ\widetilde{\partial}^{(a_1)}\circ\partial^{(a_0)}\mathbf{K} \right) \circ \mathbf{g}}_{0,c,R} (1+S)^k \\
    &\quad + dk\times C_{\cM} \norm{\mathbf{K}}_{1+a_0+\dots+a_p,S} (1+S)^{k-1} \norm{\mathbf{f}-\mathbf{g}}_{\cC^k,R}.
\end{align*}
Besides, with the same strategy as in Corollary \ref{cor:taylorineq}, one can show that
\begin{align*}
    &\left(\widetilde{\partial}^{(a_p)}\circ\cdots\circ\widetilde{\partial}^{(a_1)}\circ\partial^{(a_0)}\mathbf{K} \right) \circ \mathbf{f} - \left(\widetilde{\partial}^{(a_p)}\circ\cdots\circ\widetilde{\partial}^{(a_1)}\circ\partial^{(a_0)}\mathbf{K} \right) \circ \mathbf{g} \\
    &= \int_0^1 \frac{\d}{\d t} \left(\widetilde{\partial}^{(a_p)}\circ\cdots\circ\widetilde{\partial}^{(a_1)}\circ\partial^{(a_0)}\mathbf{K} \right) \circ \left( (t\mathbf{f} +(1-t) \mathbf{g} \right) dt,
\end{align*}
which, as in Corollary \ref{cor:taylorineq}, implies that for some constant $C$,
\begin{align*}
    &\norm{\left(\widetilde{\partial}^{(a_p)}\circ\cdots\circ\widetilde{\partial}^{(a_1)}\circ\partial^{(a_0)}\mathbf{K} \right) \circ \mathbf{f} - \left(\widetilde{\partial}^{(a_p)}\circ\cdots\circ\widetilde{\partial}^{(a_1)}\circ\partial^{(a_0)}\mathbf{K} \right) \circ \mathbf{g}}_{0,c,R} \\
    &\leq C \norm{\bK}_{1+a_0+\dots+a_p,S} \norm{\mathbf{f} - \mathbf{g}}_{0,c,R}.
\end{align*}
Consequently, this implies that for some constant $C_{\cM}$,
\begin{align} \label{slkvjsnvjn}
    &\Bigg\|\cM\left( \left(\widetilde{\partial}^{(a_p)}\circ\cdots\circ\widetilde{\partial}^{(a_1)}\circ\partial^{(a_0)}\mathbf{K} \right) \circ \mathbf{f}, \left( \widetilde{\partial}^{(i_m)}\circ\cdots\circ\widetilde{\partial}^{(i_1)}\circ\partial^{(i_0)} \mathbf{f}_l\right)_{1\leq l\leq d,\ (i_0,\dots,i_m)\in\cF} \right) \\
    & - \cM\left( \left(\widetilde{\partial}^{(a_p)}\circ\cdots\circ\widetilde{\partial}^{(a_1)}\circ\partial^{(a_0)}\mathbf{K} \right) \circ \mathbf{g}, \left( \widetilde{\partial}^{(i_m)}\circ\cdots\circ\widetilde{\partial}^{(i_1)}\circ\partial^{(i_0)} \mathbf{g}_l\right)_{1\leq l\leq d,\ (i_0,\dots,i_m)\in\cF} \right) \Bigg\|_{0,c,R} \nonumber \\
    &\leq C_{\cM} (S+1)^k\norm{\mathbf{K}}_{1+a_0+\dots+a_p,S}\norm{\mathbf{f} - \mathbf{g}}_{\cC^k,R}. \nonumber
\end{align}
And since as long as $j\leq k$, one can upper bound 
$$ \norm{\widetilde{\partial}^{(j_n)}\circ\cdots\circ\widetilde{\partial}^{(j_1)}\circ\partial^{(j_0)}(\mathbf{K}\circ \mathbf{f}) - \widetilde{\partial}^{(j_n)}\circ\cdots\circ\widetilde{\partial}^{(j_1)}\circ\partial^{(j_0)}(\mathbf{K}\circ \mathbf{g})}_{0,c,R} $$
by a finite number of terms of the form \ref{slkvjsnvjn}, we can conclude.
\end{proof}

\begin{corollary}\label{cor.Compbd}
    There exist constants $A_1,\ldots,A_k < \infty$ such that for all $\mathbf{K} \in \cC_1^{k+1}(X_1,\ldots,X_d)$, $n=1,\ldots,k$, $R> 0$, and $\mathbf{f} \in \cC^k(\cB)_{\sa}^d$,
\[
\norm{\mathbf{K} \circ \mathbf{f}}_{n,c,R} \leq A_n\left( \norm{\mathbf{K}}_{1,\infty} \norm{\mathbf{f}}_{n,c,R} + \Big( \norm{\mathbf{f}}_{\cC^{n-1},R}^n+\norm{\mathbf{f}}_{\cC^{n-1},R}^2\Big)\times \max_{j=2,\ldots,n}\norm{\mathbf{K}}_{j,\infty}  \right),
\]
where $\norm{\mathbf{f}}_{n,c,R} \coloneqq \sum_{i=1}^d \norm{\mathbf{f}_i}_{n,c,R}$ and an empty max is interpreted in this case as $0$.
\end{corollary}

\begin{proof}
    This is a direct corollary of Equations \eqref{chainestimate} and \eqref{chainestimate}, together with the fact that with the notations of Proposition \ref{chainrule},
    $$\norm{\cX \sh_1 \cY}_{0,c,R} \leq \norm{\cX}_{0,c,R} \norm{\cY}_{0,c,R}, $$
    $$\norm{\cX \sh_2 \cY}_{0,c,R} \leq \norm{\cX}_{0,c,R} \norm{\cY}_{0,c,R}, $$
\[
\norm{\cX \sh_3 \cY}_{0,c,R} \leq \norm{\cX}_{0,c,R} \norm{\cY}_{0,c,R}. \qedhere
\]
\end{proof}

\begin{corollary} \label{cor:multicompfin}
    Let $f\in\cC^k(X_1,\dots,X_d)$ then if $f^j$ is the function defined in Corollary \ref{cor:taylorineq}, $g_{a,b}\in\cC^k(X_1,\dots,X_d)$ for $a\in [0,j],\ b\in [1,d] $ , then
    $$ \cg \coloneqq f\circ\big( ( g_{a,b} )_{0\leq a\leq j, 1\leq b\leq d} \big)$$
    belongs to $\cC^k(X_1,\dots,X_d)$. Besides, for some constant $C_k$,
    \begin{equation}
        \label{eq:multicomp}
        \norm{\cg}_{\cC^{k-j},R} \leq C_k \norm{f}_{\cC^k,S} \left( \sup_{1\leq a\leq j}\ \sum_{b=1}^d\ \norm{g_{a,b}}_{\cC^k,R} \right)^j,
    \end{equation}
    where
    $$S\coloneqq \sup_{1\leq b\leq d} \sum_{a=1}^d \norm{g_{0,b}}_{0,R}.$$
\end{corollary}

\begin{proof}
    Since $f^j$ belongs to $\cC^{k-j}(X_1,\dots,X_{d(j+1)})$ by Corollary \ref{cor:taylorineq}, we simply apply Theorem \ref{chainrule} with $\bJ=f^j$. Equation \eqref{eq:multicomp} is then a consequence of Equation \eqref{dvjnslnvsdv}.
\end{proof}

\section{The asymptotic expansion} \label{sec.asymptoticexpansion}

From now on, we make the following assumption on our weights. This will ensure that the functions we consider are integrable.

\begin{assumption}
	\label{Assum:weight}
	We consider a set $\cB$ with associated weights $(c_e)_{e\in \cB}$. We assume that $c_e\geq 1$ for all $e\in \cB$, and that there exists a constant $\alpha$ and a partition of $\cB$ into $\cU$ and $\cV$ such that
	$$ \sum_{e\in \cU} e^{\alpha (1-c_e^2)} <\infty.$$
\end{assumption}

\begin{example}[Haar basis with standard weights]
With $\cU=\cB$ and $\cV=\emptyset$, the Haar basis $\cB$ and Brownian weights $(c_e)_{e \in \cB}$ from Definition \ref{def:HaarBasis} satisfy Assumption \ref{Assum:weight} for any $\alpha>1$.
Indeed,
\begin{align*}
    \sum_{e\in \cB} e^{\alpha(1-c_e^2)} = \sum_{i,n\geq 1} e^{-(\ln\left(i\right)+\ln\left(n\right))\alpha} = \left( \sum_{i\geq 1} \frac{1}{i^{\alpha}} \right)^2 <\infty,
\end{align*}
as desired.
\end{example}

\begin{definition}
	We say that $f\in \cL^k(\cB)$ if and only if it belongs to $\cC^{k}(\cB)$, and there exists a constant $K$ such that,
	\[
	\norm{f}_{\cL^k,K} \coloneqq \sup_{R>0}\ \norm{f}_{\cC^k,R}\ e^{-KR^2} <\infty,
	\]
    where $\norm{\cdot}_{\cC^k(H),R}$ is as in Definition \ref{def.NCsmooth}.
\end{definition}

We immediately state the following proposition which ensures that when evaluated in GUE random matrices, our functions are integrable.

\begin{proposition}
\label{sodncks}
	Given the following objects,
	\begin{itemize}
        \item $\cU$ and $\cV$ bases of Hilbert spaces $H_{\cU}$ and $H_{\cV}$,
		\item $Z^N = (Z_i^N)_{1\leq i\leq d}$ deterministic matrices of size $N$,
		\item $X^N = (X_e^N)_{e\in \cU}$ independent GUE random matrices of size $N$,
		\item $x = (x_e)_{e\in \cV}$ free semicircular variables,
		\item $f\in \cL^k(\cB)$ with $\cB$ the basis of $\C^d\oplus H_{\cU}\oplus H_{\cV}$.
	\end{itemize}
	Then if our weights satisfy Assumption \ref{Assum:weight}, with $\cA_N$ the free product of $\M_N(\C)$ and the von Neumann algebra generated by $x$,
    $$ \limsup_{N\to\infty} \E\left[\norm{(Z^N,X^N,x)}_{\cA_N,c}\right] <\infty,$$
    and almost surely $R_N=\norm{(Z^N,X^N,x)}_{\cA_N,c}$ is finite. Finally, for all $M\geq 0$ and $K>0$, there exists a constant $C_{M,K}$ such that for $N>2K$,
    \begin{equation}
        \label{upperboundLk}
        \E\left[\norm{f}_{\cC^k,R_N+M}\right] \leq C_{M,K} e^{K\sup_i \norm{Z_i^N}^2} \norm{f}_{\cL^k,K}.
    \end{equation}
\end{proposition}

\begin{proof}

    Thanks to \cite[Lemma 2.3.3]{AGZ2009}, coupled with the fact that given $Y_N$ a GUE random matrix $\sup_{N\geq 1}\E[\norm{Y_N}] $ is finite (see \cite[Lemma 3.3.2]{AGZ2009}), one can find a constant $L$ such that for all $e\in U$, $N\in\N$, and $\delta\geq 0$,
    $$\P\left( \norm{X_{e}^N} \geq L+\delta \right) \leq 2 e^{- \frac{\delta^2 N}{2}}.$$

    \noindent With $\alpha$ as in Assumption \ref{Assum:weight}, for $R\geq \max(L,\norm{Z_1^N},\dots,\norm{Z_d^N})$ such that $\left(R-L\right)^2\geq 2\alpha$,
    \begin{align}
    \label{sdkjvns}
        \P\left( R_N \geq R\right) &= \P\left(\exists e\in U,\ \norm{X_e^N} \geq c_e R\right) \\ \nonumber
		&\leq \sum_{e\in U} \P\left( \norm{X_{e}^N} \geq c_e R \right) \\ \nonumber
		&\leq \sum_{e\in U} \P\left( \norm{X_{e}^N} \geq L + c_e (R-L) \right) \\ \nonumber
		&\leq 2 \sum_{e\in U} e^{-Nc_e^2\times\left(R-L\right)^2/2} \\ \nonumber
		&\leq 2 \sum_{e\in U} e^{-N(c_e^2-1)\times \left(R-L\right)^2/2} e^{-N\times \left(R-L\right)^2/2}\\ \nonumber
		&\leq 2 \left(\sum_{e\in U} e^{-(c_e^2-1) \alpha}\right) e^{-N\times \left(R-L\right)^2/2}
    \end{align}
    In particular, 
    $$ \E[R_N] = \int_0^{\infty} \P\left( R_N \geq R\right) dR <\infty, $$
    thus almost surely $R_N$ is finite. Besides, for a given $N$, let us remark that for any $K>0$,
	$$ \norm{f}_{\cC^k,R_N+M} \leq \norm{f}_{\cL^k,K} \times e^{K(R_N+M)^2}.$$
	Consequently one has that
	$$ \E\left[\norm{f}_{\cC^k,R_N}\right] \leq \norm{f}_{\cL^k,K} \E\left[e^{K(R_N+M)^2}\right] = \norm{f}_{\cL^k,K} \int_{t\geq 0} \P\left( e^{K(R_N+M)^2} \geq t\right) dt.$$
    
	\noindent Thus with $\alpha$ as in Assumption \ref{Assum:weight}, for $t$ such that $\sqrt{K^{-1}\ln(t)}\geq M+ \max(L,\norm{Z_1^N},\dots,\norm{Z_d^N})$ and $\left(\sqrt{K^{-1}\ln(t)}-L-M\right)^2\geq 2 \alpha$, by using Equation \eqref{sdkjvns} with $R=\sqrt{(sK)^{-1}\ln(t)}$,
	\begin{align*}
		\P\left( e^{K(R_N+M)^2} \geq t\right) &= \P\left( R_N \geq \sqrt{K^{-1}\ln(t)}-M\right) \\
		&\leq 2 \left(\sum_{e\in U} e^{-(c_e^2-1) \alpha}\right) e^{-N\times \left(\sqrt{K^{-1}\ln(t)}-L-M\right)^2/2}
	\end{align*}
	Besides, 
    $$ \sup_{N>2K} \int_{0}^{\infty} e^{-N\times \left(\sqrt{K^{-1}\ln(t)}-L-M\right)^2/2} dt < \infty, $$
    Hence one can find a constant $C_{K,M}$ such that for $N>2K$, 
    $$ \int_{t\geq 0} \P\left( e^{KR_N^2} \geq t\right) dt \leq C_{K,M} e^{ K\sup_i \norm{Z_i^N}^2}$$
    Hence the conclusion.
\end{proof}

Let us now prove the Schwinger--Dyson equations for differentiable functions.

\begin{proposition}
	\label{soicdmlsmc}
	Given the following objects,
	\begin{itemize}
        \item $\cU$ and $\cV$ bases of Hilbert spaces $H_{\cU}$ and $H_{\cV}$,
		\item $Z^N = (Z_i^N)_{1\leq i\leq d}$ deterministic matrices of size $N$,
		\item $X^N = (X_e^N)_{e\in \cU}$ independent GUE random matrices of size $N$,
		\item $x = (x_e)_{e\in \cV}$ free semicircular variables,
		\item $f\in \cL^1(\cB)$ with $\cB$ the basis of $\C^d\oplus H_{\cU}\oplus H_{\cV}$.
	\end{itemize}
	Then if our weights satisfy Assumption \ref{Assum:weight}, with $\tau_N$ the trace on $\cA_N$ the free product of $\M_N(\C)$ and the von Neumann algebra generated by $x$, one has the following,
	\begin{align}
		\label{ocmslswc}
		\forall e\in \cU,\quad \E\left[\tau_N\left( X_e^N f(Z^N,X^N,x) \right)\right] &= \frac{1}{N} \sum_{1\leq i,j\leq N} \E\left[\tau_N\left(E_{i,j}\ \partial_e f(Z^N,X^N,x)\# E_{j,i}\right)\right] \\
        &\quad +\frac{1}{N} \E\left[ \sum_{1\leq i,j\leq N} \tau_N\otimes\tau_N\left( \widetilde{\partial}_e f(Z^N,X^N,x) \#\left(E_{j,i} \otimes E_{i,j} \right)\right) \right]. \nonumber \nonumber
	\end{align}
	\begin{equation}
		\label{mcslmc}
		\forall e\in \cV,\quad \tau_N\left(x_ef(Z^N,X^N,x)\right) = \tau_N\otimes\tau_N\left(\partial_e f(Z^N,X^N,x)\right).
	\end{equation}
\end{proposition}

\begin{proof}

    To begin with, let us remark that if $f\in \cL^k$, then for some constant $K>0$,
    $$ \sup_{R>0}\ \norm{f}_{\cC^k,R}\ e^{-KR^2} <\infty. $$
    Hence
    $$ \sup_{R>0}\ \norm{f \times X_e}_{\cC^k,R}\ e^{-2KR^2} <\infty. $$
    Thus $f \times X_e\in \cL^k$, and the left hand side of Equations \eqref{ocmslswc} and \eqref{mcslmc} are well-defined thanks to Proposition \ref{sodncks}. Let us now prove Equation \eqref{ocmslswc}. First, given $H\in\M_N(\C)$, we set $\overline{H}= (H \1_{f=e})_{f\in U}$. Then thanks to Proposition \ref{skjodncskncv},
	\begin{align*}
		\frac{d}{dt}\biggr|_{t = 0} f(Z^N,X^N+t\overline{H},x) &= Df(Z^N,X^N,x)[\overline{H}] \\
        &= \partial f(Z^N,X^N,x) \# \overline{H} + \tau_N\otimes\id_{\cA_N}\left( \widetilde{\partial}f(Z^N,X^N,x) \# (\overline{H}\otimes 1) \right) \\
        &= \partial_e f(Z^N,X^N,x) \# H + \tau_N\otimes\id_{\cA_N}\left( \widetilde{\partial}_e f(Z^N,X^N,x) \# (H\otimes 1) \right),
	\end{align*}
	
	\noindent Thus, using integration by parts, with $r$ the real part of $\sqrt{2N}(X_e^N)_{i,j}$, if $i\neq j$,
	\begin{align*}
		&\E\left[ r\ \tau_N\left(f(Z^N,X^N,x) \frac{E_{i,j}}{\sqrt{2N}}\right) \right] \\
        &= \E\left[ \frac{1}{\sqrt{2\pi}}\int_{\R} \tau_N\left(f(Z^N,X^N,x) \frac{E_{i,j}}{\sqrt{2N}}\right)\ ue^{-u^2/2}du \right] \\
		&= \E\left[ \frac{1}{\sqrt{2\pi}}\int_{\R} \tau_N\left(\frac{d}{dt}\biggr|_{t = 0} f\left(Z^N,X^N+t\overline{\frac{E_{i,j}+E_{j,i}}{\sqrt{2N}}},x\right) \frac{E_{i,j}}{\sqrt{2N}} \right)\ e^{-u^2/2}du \right] \\		
		&= \E\left[ \frac{1}{\sqrt{2\pi}}\int_{\R} \tau_N\left( \left( \partial_e f(Z^N,X^N,x) \#\left(\frac{E_{i,j}+E_{j,i}}{2N}\right) \right) E_{i,j} \right)\ e^{-u^2/2}du \right] \\
        &\quad +\E\left[ \frac{1}{\sqrt{2\pi}}\int_{\R} \tau_N\otimes\tau_N\left( \widetilde{\partial}_e f(Z^N,X^N,x) \#\left(\left(\frac{E_{i,j}+E_{j,i}}{2N}\right) \otimes E_{i,j} \right)\right)\ e^{-u^2/2}du \right] \\
		&= \E\left[ \tau_N\left( \left( \partial_e f(Z^N,X^N,x) \#\left(\frac{E_{i,j}+E_{j,i}}{2N}\right) \right) E_{i,j} \right) \right] \\
        &\quad +\E\left[ \tau_N\otimes\tau_N\left( \widetilde{\partial}_e f(Z^N,X^N,x) \#\left(\left(\frac{E_{i,j}+E_{j,i}}{2N}\right) \otimes E_{i,j} \right)\right) \right].
    \end{align*}
    The equation above remains true with $r =\sqrt{N}(X_e^N)_{i,i}$. Finally, for $i\neq j$, with $r$ the imaginary part of $\sqrt{2N}(X_e^N)_{i,j}$, one has that
	\begin{align*}
		\E\left[ r\ \tau_N\left(f(Z^N,X^N,x) \frac{\mathbf{i} E_{i,j}}{\sqrt{2N}}\right) \right] &= \E\left[ \tau_N\left( \left( \partial_e f(Z^N,X^N,x) \#\left(\frac{E_{j,i}-E_{i,j}}{2N}\right) \right) E_{i,j} \right) \right] \\
        &\quad +\E\left[ \tau_N\otimes\tau_N\left( \widetilde{\partial}_e f(Z^N,X^N,x) \#\left(\left(\frac{E_{j,i}-E_{i,j}}{2N}\right) \otimes E_{i,j} \right)\right) \right]. 
	\end{align*}
	Hence, for all $i,j\in [1,N]$,
	\begin{align*}
		\E\left[ (X^N_e)_{i,j} \tau_N\left(f(Z^N,X^N,x) E_{i,j} \right) \right] &= \frac{1}{N} \E\left[ \tau_N\left( E_{i,j} \ \partial_e f(Z^N,X^N,x) \# E_{j,i} \right) \right] \\
        &\quad +\frac{1}{N} \E\left[ \tau_N\otimes\tau_N\left( \widetilde{\partial}_e f(Z^N,X^N,x) \#\left(E_{j,i} \otimes E_{i,j} \right)\right) \right].
	\end{align*}
	Hence we deduce Equation \eqref{ocmslswc} by summing over $i,j$.

    To prove Equation \eqref{mcslmc}, let us first remark that when $f$ is a polynomial, this is simply a reformulation of the Schwinger--Dyson equations for a system of free semicircular variables (see \cite[Lemma 5.4.7]{AGZ2009}). Then thanks to Proposition \ref{sodncks}, one can assume that $R_N=\norm{(Z^N,X^N,x)}_{\cA_N,c}$ is finite. Following Definition \ref{def.NCsmooth}, for any $\varepsilon>0$, one can pick a polynomial $g_{\varepsilon}$ such that $\norm{f-g_{\varepsilon}}_{\cC^1,R_N} \leq \varepsilon$.  Therefore,
    $$ \left| \tau_N\left(f(Z^N,X^N,x) x_e\right) - \tau_N\left(g_{\varepsilon}(Z^N,X^N,x) x_e\right) \right| \leq 2 \norm{f-g_{\varepsilon}}_{0,c,R_N}, $$
	$$ \left| \tau_N\otimes\tau_N\left(\partial_e f(Z^N,X^N,x)\right) - \tau_N\otimes\tau_N\left( \partial_e f(Z^N,X^N,x)\right) \right| \leq \norm{f-g_{\varepsilon}}_{1,c,R_N}, $$
    which completes the proof.
\end{proof}

Let us now state the asymptotic expansion theorem.

\begin{theorem}
	\label{3lessopti}
	Let the following objects be given:
	\begin{itemize}
        \item $\cU$ and $\cV$ Hilbert bases of Hilbert spaces $H_{\cU}$ and $H_{\cV}$,
		\item $Z^N = (Z_i^N)_{1\leq i\leq d}$ deterministic matrices of size $N$,
		\item $X^N = (X_e^N)_{e\in U}$ independent GUE random matrices of size $N$,
		\item $x^i = (x_e^i)_{e\in \cU}$, $y^i = (y_e^i)_{e\in \cV}$ free semicircular variables,
		\item $f\in \cL^{k}\left(\cB\right)$ where $\cB$ is a basis of $\C^d\oplus H_{\cU}\oplus H_{\cV}$,
        \item $K$ such that $\norm{f}_{\cL^k,K}<\infty$.
	\end{itemize}
    
	\noindent Then if our weights satisfy Assumption \ref{Assum:weight}, there exist a function $g\in \cL^{k-4}\left(\widetilde{\cB}\right)$, where $\widetilde{\cB}$ is a basis of $\C^d\oplus H_{\cU}^{\oplus 6} \otimes H_{\cV}^{\oplus 7}$, such that for $N>2K$,
	\begin{align}
		\label{3mainresu0}
		\E\left[ \tau_N\Big(f\left(Z^N,X^N,y^0\right)\Big| \M_N(\C)\Big)\right] = &\ \tau_N\Big( f\left(Z^N,x^0,y^0\right)\Big| \M_N(\C)\Big) \\
        &+ \frac{1}{N^{2}} \E\left[ \tau_N\Big(g\left(Z^N,X^N,x^1,\dots,x^6,y^1,\dots,y^6\right)\Big| \M_N(\C)\Big)\right], \nonumber
	\end{align}
    where $\tau_N(\ \cdot\ | \M_N(\C))$ is the noncommutative conditional expectation on $\M_N(\C)$. Besides, there exists a constant $C_k$ such that for any $R>0$,
    \begin{equation}
        \label{eq:normupperexp}
        \norm{g}_{\cC^{k-4},R} \leq C_k \norm{f}_{\cC^k,R}.
    \end{equation}
    
\end{theorem}

\begin{proof}
	First we set $Z_{d+1}^N$ to be an arbitrary element of $\M_N(\C)$, then we define $h = f\times Z_{d+1}$. Finally we set,
	$$ \tau_N \coloneqq \frac{1}{N}\Tr \circ E_N,\quad z_{t} = \left(Z^N,(1-e^{-t} )^{1/2} x^0 + e^{-t/2}X^N,y^0\right).$$
    Note that since we work with the projective tensor norm, with $R_N$ defined as in Proposition \ref{sodncks},
    $$ \frac{e^{-t}}{2} \sum_{e\in \cU} \norm{ \partial_e h\left(z_t\right) \# \left( \frac{x_e^0}{(1-e^{-t})^{1/2}} - e^{t/2} X_e^N\right)} \leq \frac{e^{-t/2}}{(1-e^{-t})^{1/2}} \norm{(X^N,x)}_{\cA_N,c} \norm{h}_{\cC^1, R_N}.  $$
    \noindent Thanks to Proposition \ref{sodncks}, the right hand side is integrable. Thus, since
	\begin{align}
    \label{3vantlmbd2}
		\frac{d}{dt}\tau_N\Big(h\left(z_t\right) \Big) &= \frac{e^{-t}}{2} \sum_{e\in \cU} \tau_N\left(\partial_e h\left(z_t\right) \# \left( \frac{x_e}{(1-e^{-t})^{1/2}} - e^{t/2} X_e^N\right) \right) \\
        &\quad + \frac{e^{-t}}{2} \sum_{e\in \cU} \tau_N\otimes\tau_N\left(\widetilde{\partial}_e h\left(z_t\right) \# \left( \frac{x_e}{(1-e^{-t})^{1/2}} - e^{t/2} X_e^N\right)\otimes 1 \right), \nonumber
	\end{align}
    by dominated convergence theorem,  we have
    \begin{equation}
    \label{3vantlmbd1}
        \E\left[\tr_N\Big(h\left(Z^N,X^N,y^0\right)\Big)\right] - \tau_N\Big(h\left(Z^N,x^0,y^0\right)\Big) = -\int_{0}^{\infty} \E\left[ \frac{d}{dt} \tau_N\Big( h\left(z_{t}\right)\Big) \right] dt.
    \end{equation}
	
	\noindent And thanks to the Schwinger--Dyson equations, i.e., Proposition \ref{soicdmlsmc}, we get that
	\begin{align}
		\label{3vantlmbd}
		& \E\left[ \sum_{e\in \cU} \tau_N\left(\partial_e h\left(z_t\right) \# \left( \frac{x_e}{(1-e^{-t})^{1/2}} - e^{t/2} X_e^N\right) \right) \right] \\
        &= \E\Bigg[ \sum_{e\in \cU} \Bigg( \tau_N\otimes\tau_N \Big( \partial_e \cD_e h\left(z_t\right)\Big) - \frac{1}{N}\sum_{1\leq u,v\leq N} \tau_N\Big( E_{u,v}\ \partial_e \cD_e h\left(z_t \right)\# E_{v,u} \Big)\Bigg) \Bigg]. \nonumber \\
        &\quad -\frac{1}{N} \E\left[ \sum_{e\in \cU} \sum_{1\leq u,v\leq N} \tau_N\otimes\tau_N\left( \widetilde{\partial}_e\cD_e h(z_t) \#\left(E_{v,u} \otimes E_{u,v} \right)\right) \right] \nonumber \\
		&= \E\Bigg[ \tau_N\otimes\tau_N \Big(\eta_{\cU}\circ \partial \cD h\left(z_t\right)\Big) - \frac{1}{N}\sum_{1\leq u,v\leq N} \tau_N\Big( E_{u,v}\  \eta_{\cU}\circ\partial\cD h\left(z_t \right)\# E_{v,u} \Big) \Bigg] \nonumber \\
        &\quad -\frac{1}{N} \E\left[ \sum_{1\leq u,v\leq N} \tau_N\otimes\tau_N\left( \eta_{\cU}\circ \widetilde{\partial}\cD h(z_t) \#\left(E_{v,u} \otimes E_{u,v} \right)\right) \right]. \nonumber
	\end{align}
	
	\noindent For $A,B\in \C\langle \cB\rangle$, let
	\begin{align*}
		\Gamma_{N,t} &:= \frac{1}{N}\sum_{1\leq u,v\leq N} \tau_N\Big( A\left(z_t\right) E_{v,u}\Big) \tau_N\Big( E_{u,v} B\left(z_t\right) \Big), \\
        \Lambda_{N,t} &:= \tau_N\Big(A\left(z_t\right)\Big) \tau_N\Big(B\left(z_t\right)\Big)  - \frac{1}{N}\sum_{1\leq u,v\leq N} \tau_N\Big( E_{u,v}A\left(z_t\right) E_{v,u} B\left(z_t\right) \Big).
	\end{align*}
	
	\noindent We now want to compute those quantities. To do so we set,
	\begin{itemize}
        \item $z^1_t= \Big(Z^N, (1-e^{-t})^{1/2} x^1 + e^{-t/2} X^N ,y^1\Big)$,
		\item $z_t^2$ defined similarly but with $x^2,y^2$ instead of $x^1,y^1$,
		\item $ z_{s,t}^{1} = \Big(Z^N, (1-e^{-t})^{1/2}\left((1-e^{-s})^{1/2} x^2 +  e^{-s/2}x^1\right) + e^{-t/2} X^N ,(1-e^{-s})^{1/2} y^2 +  e^{-s/2}y^1\Big) ,$
		\item $z_{s,t}^{2}$ defined similarly but with $x^3,y^3$ instead of $x^2,y^2$,
		\item $\widetilde{z}_{s,t}^{1},\widetilde{z}_{s,t}^{2}$ defined similarly but where we replaced $x^1,x^2,x^3,y^1,y^2,y^3$ by $x^4,x^5,x^6,y^4,y^5,y^6$.
	\end{itemize}
    
	\noindent Next, note that $z_t$ has the same distribution as $z^1_{s,t}$ or $z^2_{s,t}$. Besides, $z^1_{0,t} = z^2_{0,t}$. Thus, thanks to Equation \eqref{sgwlefnnf},
	\begin{align*}
		\Lambda_{N,t} &= \lim_{s\to\infty}\tau_N\Big(A\left(z^1_{s,t}\right)\Big) \tau_N\Big(B\left(z^2_{s,t}\right)\Big)  - \frac{1}{N}\sum_{u,v} \tau_N\Big( E_{u,v}A\left(z^1_{0,t}\right) E_{v,u} B\left(z^2_{0,t}\right) \Big) \\
		&= \lim_{s\to\infty} \sum_{u,v} \tau_N\Big(A\left(z^1_{s,t}\right)E_{v,v}\Big) \tau_N\Big(B\left(z^2_{s,t}\right)E_{u,u}\Big) - \frac{1}{N}\sum_{u,v} \tau_N\Big( E_{u,v}A\left(z^1_{0,t}\right) E_{v,u} B\left(z^2_{0,t}\right) \Big) \\
        &= \lim_{s\to\infty} \frac{1}{N}\sum_{u,v} \tau_N\Big(E_{u,v}A\left(z^1_{s,t}\right)E_{v,u}B\left(z^2_{s,t}\right)\Big) - \tau_N\Big( E_{u,v}A\left(z^1_{0,t}\right) E_{v,u} B\left(z^2_{0,t}\right) \Big) \\
		&= \frac{1}{N}\sum_{u,v} \int_0^{\infty} \frac{d}{ds}\tau_N\Big(E_{u,v}A\left(z^1_{s,t}\right)E_{v,u}B\left(z^2_{s,t}\right)\Big)\ ds.
	\end{align*}
	Besides,
	\begin{align*}
		&\frac{d}{ds}\tau_N\Big(E_{u,v}A\left(z^1_{s,t}\right)E_{v,u}B\left(z^2_{s,t}\right)\Big) \\
		&= \sum_{e\in \cU} \frac{(1-e^{-t})^{1/2} e^{-s}}{2}\ \Bigg(\tau_N\left(E_{u,v}\ \partial_e A\left(z^1_{s,t}\right)\#\left( \frac{x^2_e}{(1-e^{-s})^{1/2}} - e^{s/2}x^1_e \right)\ E_{v,u}B\left(z^2_{s,t}\right)\right) \\
		&\quad\quad\quad\quad\quad\quad\quad\quad\quad\quad\quad+  \tau_N\left(E_{u,v} A\left(z^1_{s,t}\right) E_{v,u}\ \partial_e B\left(z^2_{s,t}\right)\#\left( \frac{x^3_e}{(1-e^{-s})^{1/2}} - e^{s/2}x^1_e \right)\right) \Bigg) \\
		&\quad + \sum_{e\in \cV} \frac{e^{-s}}{2}\ \Bigg(\tau_N\left(E_{u,v}\ \partial_e A\left(z^1_{s,t}\right)\#\left( \frac{y^2_e}{(1-e^{-s})^{1/2}} - e^{s/2}y^1_e \right)\ E_{v,u}B\left(z^2_{s,t}\right)\right) \\
		&\quad\quad\quad\quad\quad\quad\quad+  \tau_N\left(E_{u,v} A\left(z^1_{s,t}\right) E_{v,u}\ \partial_e B\left(z^2_{s,t}\right)\#\left( \frac{y^3_e}{(1-e^{-s})^{1/2}} - e^{s/2}y^1_e \right)\right) \Bigg)
	\end{align*}
	Thus if for $e\in \cU$ we write $\partial_e A = \sum_i A_i^1\otimes A_i^2$ and $\partial_e B = \sum_{j} B_j^1\otimes B_j^2$, then thanks to the Schwinger--Dyson equations, as well as Equation \eqref{sgwlefnnf},
	\begin{align*}
		&\tau_N\left(E_{u,v}\ \partial_e A\left(z^1_{s,t}\right)\#\left( \frac{x^2_e}{(1-e^{-s})^{1/2}} - e^{s/2}x^1_e \right)\ E_{v,u}B\left(z^2_{s,t}\right)\right) \\
		&= \sum_{i} \tau_N\left(\left( \frac{x^2_e}{(1-e^{-s})^{1/2}} - e^{s/2}x^1_e \right) A_i^2\left(z^1_{s,t}\right) E_{v,u}B\left(z^2_{s,t}\right) E_{u,v} A_i^1\left(z^1_{s,t}\right) \right) \\	
		&= -(1-e^{-t})^{1/2}\sum_{i,j} \tau_N\left( A_i^2\left(z^1_{s,t}\right) E_{v,u}B_j^1\left(z^2_{s,t}\right)\right)\ \tau_N\left( B_j^2\left(z^2_{s,t}\right) E_{u,v} A_i^1\left(z^1_{s,t}\right) \right) \\	
		&= -\frac{(1-e^{-t})^{1/2}}{N} \sum_{i,j} \tau_N\left( A_i^2\left(z^1_{s,t}\right) E_{v,v} A_i^1\left(\widetilde{z}^1_{s,t}\right) B_j^2\left(\widetilde{z}^2_{s,t}\right) E_{u,u} B_j^1\left(z^2_{s,t}\right) \right).
	\end{align*}
    Similarly, we also have
	\begin{align*}
		&\tau_N\left(E_{u,v} A\left(z^1_{s,t}\right) E_{v,u}\ \partial_e B\left(z^2_{s,t}\right)\#\left( \frac{x^3_e}{(1-e^{-s})^{1/2}} - e^{s/2}x^1_e \right)\right) \\
		&= -\frac{(1-e^{-t})^{1/2}}{N} \sum_{i,j} \tau_N\left( A_i^2\left(z^1_{s,t}\right) E_{v,v} A_i^1\left(\widetilde{z}^1_{s,t}\right) B_j^2\left(\widetilde{z}^2_{s,t}\right) E_{u,u} B_j^1\left(z^2_{s,t}\right) \right).
	\end{align*}
	Hence we get that,
	\begin{align*}
		&\frac{(1-e^{-t})^{1/2}}{2}\sum_{1\leq u,v\leq N} \tau_N\left(E_{u,v}\ \partial_e A\left(z^1_{s,t}\right)\#\left( \frac{x^2_e}{(1-e^{-s})^{1/2}} - e^{s/2}x^1_e \right)\ E_{v,u}B\left(z^2_{s,t}\right)\right) \\
		&\quad\quad\quad\quad\quad\quad\quad\quad\quad+  \tau_N\left(E_{u,v} A\left(z^1_{s,t}\right) E_{v,u}\ \partial_e B\left(z^2_{s,t}\right)\#\left( \frac{x^3_e}{(1-e^{-s})^{1/2}} - e^{s/2}x^1_e \right)\right) \\
		&= -\frac{1-e^{-t}}{N} \sum_{i,j} \tau_N\left( A_i^2\left(z^1_{s,t}\right) A_i^1\left(\widetilde{z}^1_{s,t}\right) B_j^2\left(\widetilde{z}^2_{s,t}\right) B_j^1\left(z^2_{s,t}\right) \right) \\
		&= -\frac{1-e^{-t}}{N} \tau_N\left( \ev_{z^1_{s,t}, \widetilde{z}^1_{s,t}, \widetilde{z}^2_{s,t}, z^2_{s,t}}\circ T \circ (\partial_e\otimes \partial_e) (A\otimes B) \right),
	\end{align*}
    where the permuted multiplication map $T(A_1\otimes A_2\otimes A_3\otimes A_4) = A_2A_1A_4A_3$ is a combination of operators from Proposition \ref{proppourvoila}, and $\ev_{z^1_{s,t}, \widetilde{z}^1_{s,t}, \widetilde{z}^2_{s,t}, z^2_{s,t}}$ is an operator evaluating a given function of $\cC^k_4(\cB)$ in different families of variables.
	Consequently, by doing the same process for $e\in V$, we get that
	\begin{align*}
		&\sum_{1\leq u,v\leq N} \frac{d}{ds}\tau_N\Big(E_{u,v}A\left(z^1_{s,t}\right)E_{v,u}B\left(z^2_{s,t}\right)\Big) \\
		&= -\frac{(1-e^{-t}) e^{-s}}{N} \tau_N\left( \ev_{z^1_{s,t}, \widetilde{z}^1_{s,t}, \widetilde{z}^2_{s,t}, z^2_{s,t}}\circ \eta_{\cU}\circ T \circ (\partial\otimes\partial) (A\otimes B) \right) \\
		&\quad -\frac{e^{-s}}{N} \tau_N\left( \ev_{z^1_{s,t}, \widetilde{z}^1_{s,t}, \widetilde{z}^2_{s,t}, z^2_{s,t}}\circ \eta_{\cV}\circ T \circ (\partial\otimes\partial) (A\otimes B) \right) \\
		&= -\frac{e^{-s}}{N} \tau_N\left( \ev_{z^1_{s,t}, \widetilde{z}^1_{s,t}, \widetilde{z}^2_{s,t}, z^2_{s,t}}\circ (\eta_{\cV}+(1-e^{-t})\eta_{\cU})\circ T \circ (\partial\otimes\partial) (A\otimes B) \right),
	\end{align*}
    
	And finally,
	\begin{align*}
		\Lambda_{N,t} &= - \frac{1}{N^2} \int_0^{\infty} e^{-s}\tau_N\left( \ev_{z^1_{s,t}, \widetilde{z}^1_{s,t}, \widetilde{z}^2_{s,t}, z^2_{s,t}}\circ (\eta_{\cV}+(1-e^{-t})\eta_{\cU})\circ T \circ (\partial\otimes\partial) (A\otimes B) \right)ds.
	\end{align*}
    Since we assumed that $f\in \cL^{k}\left(\cB\right)$, we have $h \in \cL^{k}\left(\cB\right)$ as well. Besides, $z^1_{s,t}, \widetilde{z}^1_{s,t}, \widetilde{z}^2_{s,t}$ and $ z^2_{s,t}$ all have the same noncommutative distribution, therefore thanks to Proposition \ref{proppourvoila} once again, there exists a constant $C$ independent of $s,t$ such that for any trace polynomials,
    $$ H_1(P) \coloneqq \frac{1}{N^2} \int_0^{\infty} e^{-s} \tau_N\left( \ev_{z^1_{s,t}, \widetilde{z}^1_{s,t}, \widetilde{z}^2_{s,t}, z^2_{s,t}}\circ (\eta_{\cV}+(1-e^{-t})\eta_{\cU})\circ T \circ (\partial\otimes\partial) (\eta_{\cU}\circ\partial\cD P) \right)ds, $$
    $$ H_2(P) \coloneqq \tau_N\otimes\tau_N \Big(\eta_{\cU}\circ \partial \cD h\left(z_t\right)\Big) - \frac{1}{N}\sum_{1\leq u,v\leq N} \tau_N\Big( E_{u,v}\  \eta_{\cU}\circ\partial\cD P\left(z_t \right)\# E_{v,u} \Big), $$
    Since $z^1_{s,t}, \widetilde{z}^1_{s,t}, \widetilde{z}^2_{s,t}$ and $ z^2_{s,t}$ all have the same noncommutative distribution, thanks to Proposition \ref{proppourvoila}, $H_1$ and $H_2$ are continuous linear maps on $\cL^4\left(\cB\right)$. Besides, we just showed that they are equal on trace polynomials; therefore, since we assumed that $f\in \cL^{k}\left(\cB\right)$ and thus so does $h$, we have that $H_1(h) = H_2(h)$. 

    Similarly, thanks to Equation \eqref{sgwlefnnf}, we have that
    \begin{align*}
        \Gamma_{N,t} &= \frac{1}{N^2}\sum_{1\leq u,v\leq N} \tau_N\Big( E_{v,v} A\left(z_t^1\right) E_{u,u} B\left(z_t^2\right) \Big) \\
        &= \frac{1}{N^2}\tau_N\Big( A\left(z_t^1\right) B\left(z_t^2\right) \Big).
    \end{align*}
    And therefore, with the same reasoning as above,
    $$ \frac{1}{N} \sum_{1\leq u,v\leq N} \tau_N\otimes\tau_N\left( \eta_{\cU}\circ \widetilde{\partial}\cD h(z_t) \#\left(E_{v,u} \otimes E_{u,v} \right)\right) = \frac{1}{N^2} \tau_N\left( \ev_{z^1_t,z^2_t}\circ m(\eta_{\cU}\circ \widetilde{\partial}\cD h) \right),$$
    where $m$ is as in Proposition \ref{proppourvoila}. Consequently, by combining Equations \eqref{3vantlmbd},\eqref{3vantlmbd1} and \eqref{3vantlmbd2} with previous computations we get that
    \begin{align}
    \label{skdjvnsnv}
		&\E\left[\tau_N\Big(h\left(Z^N,X^N,y^0\right)\Big)\right] - \tau_N\Big(h\left(Z^N,x^0,y^0\right)\Big) \\
		&= \frac{1}{2N^2}\int_{0}^{\infty}\int_0^{\infty} e^{-t-s} \E\left[ \tau_N\left( \ev_{z^1_{s,t}, \widetilde{z}^1_{s,t}, \widetilde{z}^2_{s,t}, z^2_{s,t}}\circ (\eta_{\cV}+(1-e^{-t})\eta_{\cU})\circ T \circ (\partial\otimes\partial) (\eta_{\cU}\circ\partial\cD h) \right) \right] ds dt \nonumber \\
        &\quad + \frac{1}{2N^2}\int_{0}^{\infty} e^{-t} \E\left[ \tau_N\left( \ev_{z^1_t,z^2_t}\circ m(\eta_{\cU}\circ \widetilde{\partial}\cD h) \right) \right] dt \nonumber \\
        &\quad + \frac{1}{2}\int_{0}^{\infty} e^{-t}\ \E\left[ \sum_{e\in \cU} \tau_N\otimes\tau_N\left(\widetilde{\partial}_e h\left(z_t\right) \# \left( \frac{x_e}{(1-e^{-t})^{1/2}} - e^{t/2} X_e^N\right)\otimes 1 \right) \right] dt. \nonumber
	\end{align}
	Note that one can write 
    $$ \ev_{z^1_t,z^2_t}\circ m(\eta_{\cU}\circ \widetilde{\partial}\cD h) = \Big[ \ev_{z^1_t,z^2_t} \circ \Gamma(\widetilde{\partial}\circ\partial f) \Big] \times Z_{d+1}^N, $$
    where $\Gamma(B\otimes A_1\otimes A_2) = A_1 B A_2$. One can then set 
    $$ g(Z,X,x^1,y^1,x^2,y^2) = \frac{1}{2} \int_0^{\infty} e^{-t} \ev_{z^1_t,z^2_t} \circ \Gamma(\widetilde{\partial}\circ\partial f), $$
    and 
    $$ \frac{1}{2N^2}\int_{0}^{\infty} e^{-t} \E\left[ \tau_N\left( \ev_{z^1_t,z^2_t}\circ m(\eta_{\cU}\circ \widetilde{\partial}\cD h) \right) \right] dt = \frac{1}{N^2} \E\left[ \tau_N\right(g(Z^N,X^N,x^1,y^1,x^2,y^2) Z_{d+1}^N\left) \right]. $$
    Moreover, one can easily find a constant $C_k$ such that $\norm{g}_{\cC^{k-2},R} \leq C_k \norm{f}_{\cC^{k},R}$. One can then do a similar process for the second line of Equation \eqref{skdjvnsnv}. Finally, we can repeat the entire process for the last line of Equation \eqref{skdjvnsnv}. This yields a function $g$ which thanks to Proposition \ref{proppourvoila} belongs to $\cL^{k-4}(\cB)$ and such that
    \begin{align*}
        \E\left[\tau_N\Big(f\left(Z^N,X^N,y^0\right) Z_{d+1}^N\Big)\right] &= \tau_N\Big(f\left(Z^N,x^0,y^0\right)Z_{d+1}^N\Big) \\
        &\quad + \frac{1}{N^2} \E\left[ \tau_N\Big(g(Z^N,X^N,x^1,\dots,x^6,y^1,\dots,y^6) Z_{d+1}^N\Big) \right].
    \end{align*}	
    Since $Z_{d+1}^N$ can be taken arbitrarily in $\M_N(\C)$ this yields the conclusion.
\end{proof}

\section{Analysis of a key stochastic differential equation}\label{sec.SDE}

\subsection{An integral equation}\label{sec.IE}

Key to our development is an integral equation of the form $y = c + \int_a^{\boldsymbol{\cdot}} f(t,y(t))\,\d t$.  This abstract form of integral equation in fact includes the free Langevin SDE, where we take $c$ to be the free Brownian motion.  We will thus analyze the behavior of the solution to the SDE as a smooth function of the process $(S_t)_{t \geq 0}$ expressed in terms of an orthonormal basis, by means of existence/uniqueness theory for general integral equations which is essentially identical to that of a standard ODE.

If $\cX$ is a Banach space, $x_0 \in \cX$, and $r > 0$, write
\[
C_r(x_0) \coloneqq \{x \in \cX : \norm{x-x_0}_{\cX} \leq r\}
\]
for the closed ball in $\cX$ of radius $r$ centered at $x_0$.

\begin{theorem}\label{thm.IE}
Let $\cX$ be a Banach space and $a \in \R$.
Suppose $f \colon [a,\infty) \times \cX \to \cX$ is a continuous function such that for all $T > a$ and $R > 0$, there exists a constant $c_{T,R} < \infty$ such that
\[
\sup_{a \leq t \leq T}\norm{f(t,x) - f(t,y)}_{\cX} \leq c_{T,R}\norm{x-y}_{\cX} \qquad (x,y \in C_R(x_0)).
\]
If $c \colon [a,\infty) \to \cX$ is continuous, then there exists a $T_{\infty} \in (a,\infty]$ and a continuous function $y_{\infty} \colon [a,T_{\infty}) \to \cX$ such that:
\begin{enumerate}[label=(\roman*),font=\normalfont]
    \item $\displaystyle y_{\infty}(t) = c(t) + \int_a^t f(s,y_{\infty}(s))\,\d s$ for all $t \in [a,T_{\infty})$, i.e., $y_{\infty} = c + \int_a^{\boldsymbol{\cdot}} f(t,y_{\infty}(t))\,\d t$ on $[a,T_{\infty})$;\label{item.solveIE}
    \item If $T > a$ and $y \colon [a,T) \to \cX$ is a continuous function such that $y = c + \int_a^{\boldsymbol{\cdot}} f(t,y(t))\,\d t$ on $[a,T)$, then $T \leq T_{\infty}$, and $y_{\infty}|_{[a,T)} = y$.\label{item.maxsolIE}
\end{enumerate}
Moreover, if $T_{\infty} < \infty$, then $\norm{y_{\infty}(t)}_{\cX} \to \infty$ as $t \nearrow T_{\infty}$.
The function $y_{\infty}$ is the \textbf{maximal solution} to the integral equation $y = c + \int_a^{\boldsymbol{\cdot}} f(t,y(t))\,\d t$, and $T_{\infty}$ is its \textbf{lifetime} (or \textbf{blow-up time}).
\end{theorem}

For the sake of completeness, we prove this result below.
First, though, note that the``blow-up'' phenomenon stated at the end of Theorem \ref{thm.IE} implies that if $f$ also grows sub-linearly, then $T_{\infty} = \infty$.
Indeed, suppose tnhat for every $T > a$, there exists an $M_T < \infty$ such that
\[
\sup_{a \leq t \leq T}\norm{f(t,x)}_{\cX} \leq M_T(1+\norm{x}_{\cX}) \qquad \forall x \in \cX.
\]
If $a \leq t < T < \infty$ and $y$ solves $y = c+\int_a^{\boldsymbol{\cdot}} f(t,y(t))\,\d t$ on $[a,T)$, then
\begin{align*}
    \norm{y(t)}_{\cX} & \leq \norm{c(t)}_{\cX} + \int_a^t \norm{f(s,y(s))}_{\cX} \,\d s \\
    & \leq \sup_{a \leq s \leq t} \norm{c(s)}_{\cX} + M_T\int_a^t (1+\norm{y(s)}_{\cX}) \,\d s \\
    & = \sup_{a \leq s \leq t} \norm{c(s)}_{\cX} + M_T(t-a) + M_T\int_a^t \norm{y(s)}_{\cX} \,\d s.
\end{align*}
Therefore, by Gr\"onwall's inequality,
\[
 \norm{y(t)}_{\cX} \leq \left( \sup_{0 \leq s \leq t} \norm{c(s)}_{\cX} + M_T(t-a)\right) e^{M_T(t-a)} \qquad (a \leq t < T).
\]
In particular, $\sup\left\{\norm{y(s)}_{\cX} : a \leq t < T\right\} < \infty$.
It follows that $T_{\infty} = \infty$.

The proof of Theorem \ref{thm.IE} begins with a Picard--Lindel\"of-type local existence theorem.

\begin{theorem}\label{thm.PL}
Let $\cX$ be a Banach space, and fix $x_0 \in \cX$, $r > 0$, $a \in \R$, and $T_0 > a$.
Suppose $f \colon [a,\infty) \times C_r(x_0) \to \cX$ is a continuous function such that for all $T > a$, there exists a $c_T < \infty$ such that
\[
\sup_{a \leq t \leq T} \norm{f(t,x) - f(t,y)}_{\cX} \leq c_T\norm{x-y}_{\cX} \qquad (x,y \in C_r(x_0)).
\]
Define
\[
M \coloneqq \sup_{(t,x) \in [a,T_0] \times C_r(x_0)} \norm{f(t,x)}_{\cX} < \infty.
\]
If $a < T < \min\left\{T_0,a + r/(2M)\right\}$ and $c \colon [a,T] \to C_{r/2}(x_0)$ is continuous, then there exists a continuous function $y \colon [a,T] \to C_r(x_0)$ such that
\begin{equation}
    y(t) = c(t) + \int_a^t f(s,y(s))\,\d s \qquad (a \leq t \leq T).\label{eq.IE}
\end{equation}
\end{theorem}

\begin{proof}
For any continuous $y \colon [a,T] \to C_r(x_0)$, define
\[
\Phi_c(y)(t) \coloneqq c(t) + \int_a^t f(s,y(s))\,\d s \qquad (a \leq t \leq T).
\]
If $a \leq t \leq T$, then
\[
\norm{\Phi_c(y)(t) - x_0}_{\cX} \leq \norm{c(t) - x_0}_{\cX} + \int_a^t \norm{f(s,y(s))}_{\cX} \,\d s \leq \frac{r}{2} + M(T-a) \leq r,
\]
i.e., $\Phi_c(y)$ is a continuous function from $[a,T]$ to $C_r(x_0)$.
Consequently,
\[
\Phi_c^n(y) \coloneqq \underbrace{\left(\Phi_c \circ \cdots \circ \Phi_c\right)}_{n \text{ times}}(y) \qquad (n \in \N)
\]
is well defined.

Next, observe that if $y,z \colon [a,T] \to C_r(x_0)$ are continuous and $a \leq t \leq T$, then
\begin{align*}
    \norm{\Phi_c(y)(t) - \Phi_c(z)(t)}_{\cX} & = \norm{\int_a^t \left(f(y(s)) - f(z(s))\right)\d s}_{\cX} \\
    & \leq \int_a^t \norm{f(s,y(s)) - f(s,z(s))}_{\cX}\,\d s \\
    & \leq c_T\int_a^t \norm{y(s) - z(s)}_{\cX}\,\d s.
\end{align*}
Consequently, if $y_0 \coloneqq c$ and $y_n \coloneqq \Phi_c^n(y_{n-1})$ for all $n \in \N$, then
\begin{align*}
    \norm{y_{n+1}(t) - y_n(t)}_{\cX} & = \norm{\Phi_c\left(\Phi_c^n(c)\right)(t) - \Phi_c\left(\Phi_c^{n-1}(c)\right)(t)}_{\cX} \\
    & \leq  c_T\int_a^t \norm{\Phi_c^n(c)(s) - \Phi_c^{n-1}(c)(s)}_{\cX}\,\d s \\
    & = c_T\int_a^t \norm{y_n(s) - y_{n-1}(s)}_{\cX}\,\d s
\end{align*}
for all $n \in \N$.
Applying this estimate inductively yields
\begin{align*}
    \norm{y_{n+1}(t) - y_n(t)}_{\cX} & \leq c_T^n\int_a^t \int_a^{t_1}\cdots \int_a^{t_{n-1}} \norm{y_1(t_n) - y_0(t_n)}_{\cX}\,\d t_n \cdots \d t_2\, \d t_1  \\
    & \leq \sup_{a \leq s \leq t}\norm{y_1(s) - y_0(s)}_{\cX} c_T^n\frac{(t-a)^n}{n!}
\end{align*}
for all $n \in \N_0$.
Writing $M_0 \coloneqq \sup\left\{\norm{y_1(s) - y_0(s)}_{\cX} : a \leq s \leq T\right\} < \infty$, it follows that
\[
\sum_{n=0}^{\infty} \sup_{a \leq t \leq T} \norm{y_{n+1}(t) - y_n(t)}_{\cX} \leq M_0\sum_{n=0}^{\infty} \frac{(c_T(T-a))^n}{n!} = M_0e^{c_T(T-a)} < \infty.
\]
Since $\cY \coloneqq C([a,T];C_r(x_0))$ is a complete metric space with the metric
\[
d(z_1,z_2) \coloneqq \sup_{a \leq t \leq T}\norm{z_1(t) - z_2(t)}_{\cX} \qquad (z_1,z_2 \in \cY),
\]
it follows that $(y_n)_{n \in \N}$ converges in $\cY$.
Writing $y \in \cY = C([a,T];C_r(x_0))$ for the limit of $(y_n)_{n \in \N}$, it is easy to check that $y$ satisfies \eqref{eq.IE}.
\end{proof}

Let $\cX$ be a Banach space, $\cU \subseteq \cX$ be an open set, and $I \subseteq \R$ be an interval.
A function $f \colon I \times \cU \to \cX$ is \textbf{locally Lipschitz in space}, \textbf{locally uniformly in time} if for all compact $J \subseteq I$ and $x_0 \in \cU$, there exist $r > 0$ and $c_{J,x_0,r} < \infty$ such that $C_r(x_0) \subseteq \cU$ and
\[
\sup_{t \in J}\norm{f(t,x) - f(t,y)}_{\cX} \leq c_{J,x_0,r}\norm{x-y}_{\cX} \qquad (x,y \in C_r(x_0)).
\]
By a standard compactness argument, if $f \colon I \times \cU \to \cX$ is locally Lipschitz in space, locally uniformly in time and $J \subseteq I$ and $K \subseteq \cU$ are both compact, then there exists a $c_{J,K} < \infty$ such that
\begin{equation}
    \sup_{t \in J}\norm{f(t,x) - f(t,y)}_{\cX} \leq c_{J,K}\norm{x-y}_{\cX} \qquad (x,y \in K).\label{eq.locLip}
\end{equation}
With this in mind, here is a uniqueness result for \eqref{eq.IE} when the coefficient $f$ is locally Lipschitz in space, locally uniformly in time.

\begin{proposition}\label{prop.uniqueness}
Suppose $\cX$ is a Banach space, $\cU \subseteq \cX$ is an open set, and $-\infty < a < T < \infty$.
If $f \colon [a,\infty) \times \cU \to \cX$ is locally Lipschitz in space, locally uniformly in time; $c_1,c_2 \in C([a,T];\cX)$; and $y_1,y_2 \in C([a,T];\cU)$ satisfy
\[
y_i(t) = c_i(t) + \int_a^t f(s,y_i(s))\,\d s \qquad (a \leq t \leq T, \; i=1,2),
\]
then
\[
\norm{y_1(t) - y_2(t)}_{\cX} \leq  e^{C(t-a)}\sup_{a \leq s \leq t}\norm{c_1(s) - c_2(s)}_{\cX} \qquad (a \leq t \leq T),
\]
where $C = c_{J,K}$ is as in \eqref{eq.locLip} with $J \coloneqq [a,T]$ and $K \coloneqq y_1([a,T]) \cup y_2([a,T])$.
In particular, if $c_1 \equiv c_2$, then $y_1 \equiv y_2$.
\end{proposition}

\begin{proof}
If $a \leq t \leq T$, then
\begin{align*}
    \norm{y_1(t) - y_2(t)}_{\cX} & \leq \norm{c_1(t) - c_2(t)}_{\cX} + \int_a^t \norm{f(s,y_1(s)) - f(s,y_2(s))}_{\cX} \,\d s \\
    & \leq \sup_{a \leq s \leq t}\norm{c_1(s) - c_2(s)}_{\cX} + C\int_a^t \norm{y_1(s) - y_2(s)}_{\cX} \,\d s.
\end{align*}
Consequently, the desired estimate follows immediately from Gr\"onwall's inequality.
\end{proof}

Theorem \ref{thm.PL} and Proposition \ref{prop.uniqueness} together yield the existence and uniqueness of a maximal solution to \eqref{eq.IE} when $f$ is locally Lipschitz in space, locally uniformly in time.

\begin{theorem}\label{thm.existuniq}
Let $\cX$ be a Banach space, $\cU \subseteq \cX$ be an open set, and $a \in \R$.
If $f \colon [a,\infty) \times \cU \to \cX$ is locally Lipschitz in space, locally uniformly in time and $c \in C([a,\infty);\cX)$ is such that $c(a) \in \cU$, then there exists a time $T_{\infty} \in (a,\infty]$ and a continuous function $y_{\infty} \colon [a,T_{\infty}) \to \cU$ such that:
\begin{enumerate}[label=(\roman*),font=\normalfont]
    \item $y_{\infty} = c + \int_a^{\boldsymbol{\cdot}} f(t,y_{\infty}(t))\,\d t$ on $[a,T_{\infty})$;\label{item.soln}
    \item If $T \in (a,\infty]$ and $y \colon [a,T) \to \cU$ is a continuous function such that $y = c + \int_a^{\boldsymbol{\cdot}} f(t,y(t))\,\d t$ on $[a,T)$, then $T \leq T_{\infty}$, and $y_{\infty}|_{[a,T)} = y$.\label{item.max}
\end{enumerate}
The function $y_{\infty}$ is the \textbf{maximal solution} to $y = c + \int_a^{\boldsymbol{\cdot}} f(t,y(t))\,\d t$, and $T_{\infty}$ is its \textbf{lifetime}.
\end{theorem}

\begin{proof}
Let $T_0 > a$ be arbitrary.
Since $x_0 \coloneqq c(a) \in \cU$ and $f$ is locally Lipschitz in space, locally uniformly in time, there exist $r > 0$ and $c_{[a,T_0],r} < \infty$ such that $C_r(x_0) \subseteq \cU$ and
\[
\sup_{a \leq t \leq T_0} \norm{f(t,x) - f(t,y)}_{\cX} \leq c_{[a,T_0],r}\norm{x-y}_{\cX} \qquad (x,y \in C_r(x_0)).
\]
Let $M \coloneqq \sup\left\{ \norm{f(t,x)}_{\cX} : (t,x) \in [a,T_0] \times C_r(x_0)\right\} < \infty$.
Since $c(a) = x_0$ and $c$ is continuous, there exists a $T > a$ such that $c(t) \in C_{r/2}(x_0)$ for all $t \in [a,T]$.
Assuming, in addition, that $T < \min\left\{T_0,a+ r/(2M)\right\}$, we obtain from Theorem \ref{thm.PL} that there exists a continuous function $y \colon [a,T] \to C_r(x_0) \subseteq \cU$ such that \eqref{eq.IE} holds.

Next, define $\cS$ to be the set of all continuous functions $y \colon [a,T) \to \cU$ such that $T \in (a,\infty]$ and $y = c + \int_a^{\boldsymbol{\cdot}} f(t,y(t))\,\d t$ on $[a,T)$.
By the previous paragraph, $\cS \neq \emptyset$.
Let $I \coloneqq \bigcup_{y \in \cS} \operatorname{dom} \left(y\right) \subseteq [a,\infty)$, where $\operatorname{dom} \left(y\right)$ denotes the domain of $y$.
Since $\operatorname{dom}\left(y\right)$ is a left-closed, right-open interval with left endpoint $a$ for every $y \in \cS$, the set $I$ is also a left-closed, right-open interval with left endpoint $a$.
Write $T_{\infty} \in (a,\infty]$ for the right endpoint of $I$ so that $I = [a,T_{\infty})$.
By Proposition \ref{prop.uniqueness}, if $y_1,y_2 \in \cS$, then $y_1 \equiv y_2$ on $\operatorname{dom}\left(y_1\right) \cap \operatorname{dom}\left(y_2\right)$.
Consequently, there exists a well-defined continuous function $y_{\infty} \colon I = [a,T_{\infty}) \to \cU$ such that for every $y \in \cS$, $y_{\infty}|_{\operatorname{dom}\left(y\right)} = y$.
It is easy to check from these definitions of $T_{\infty}$ and $y_{\infty}$ that \ref{item.soln} holds, and \ref{item.max} holds by construction.
\end{proof}


If $f$ has the stronger Lipschitz property in Theorem \ref{thm.IE}, then the lifetime of the solution can be characterized in terms of blow-up.

\begin{theorem}\label{thm.blowup}
Retain the setup of Theorem \ref{thm.IE}.
Suppose $c \colon [a,\infty) \to \cX$ is continuous, and let $y_{\infty} \colon [a,T_{\infty}) \to \cX$ be the maximal solution to $y = c + \int_a^{\boldsymbol{\cdot}} f(t,y(t))\,\d t$ provided by Theorem \ref{thm.existuniq}.
If $T_{\infty} < \infty$, then $\norm{y_{\infty}(t)}_{\cX} \to \infty$ as $t \nearrow T_{\infty}$.
Therefore, in this case, $T_{\infty}$ is also called the \textbf{blow-up time} of $y = c + \int_a^{\boldsymbol{\cdot}} f(t,y(t))\,\d t$.
\end{theorem}

\begin{proof}
Suppose $a < T_0 < \infty$ and $y_0 \colon [a,T_0) \to \cX$ satisfies $y_0 = c + \int_a^{\boldsymbol{\cdot}} f(t,y_0(t))\,\d t$ on $[a,T_0)$.
We claim that if $\norm{y_0(t)}_{\cX} \not\to \infty$ as $t \nearrow T_0$, then there exists a $T > T_0$ and a continuous function $y \colon [a,T] \to \cX$ extending $y_0$ and satisfying $y = c + \int_a^{\boldsymbol{\cdot}} f(t,y(t))\,\d t$ on $[a,T]$.
If we prove this claim, then we are done.

To prove the claim, note that if $\norm{y_0(t)}_{\cX} \not\to \infty$ as $t \nearrow T_0$, then there exists a sequence $(t_n)_{n \in \N}$ in $[a,T_0)$ such that $t_n \nearrow T_0$ as $n \to \infty$ and $R \coloneqq \sup\left\{\norm{y(t_n)}_{\cX} : n \in \N\right\} < \infty$.
Now, set
\[
L \coloneqq \sup_{t_1 \leq t \leq T_0+1}\norm{c(t)}_{\cX} < \infty, \; r \coloneqq 2(2L+R), \; \text{ and } \; M \coloneqq \sup_{(t,x) \in [a,T_0+1] \times C_r}\norm{f(x)}_{\cX} < \infty,
\]
where $C_r = C_r(0)$.
Also, choose an $n \in \N$ such that $t_n + r/(2M) > T_0$ and a $T$ such that $T_0 < T < \min\left\{T_0+1,t_n+r/(2M)\right\}$.
Finally, if
\[
c_n(t) \coloneqq c(t) + \int_a^{t_n} f(t,y_0(t))\,\d t = c(t) - c(t_n) + y(t_n) \qquad (a \leq t < \infty),
\]
then
\[
\sup_{t_n \leq t \leq T_0 + 1} \norm{c_n(t)}_{\cX}  \leq 2 \sup_{t_1 \leq t \leq T_0+1} \norm{c(t)}_{\cX} + \sup_{m \in \N} \norm{y(t_m)}_{\cX} = 2L+R = \frac{r}{2}.
\]
Consequently, Theorem \ref{thm.PL} produces a continuous function $y_n \colon [t_n,T] \to C_r \subseteq \cX$ such that
\[
y_n = c_n + \int_{t_n}^{\boldsymbol{\cdot}} f(t, y_n(t))\,\d t \; \text{ on } \; [t_n,T].
\]
We now paste $y_n$ and $y_0$ together to obtain the desired $y$.
First, note that if $t_n \leq t < T_0$, then
\begin{align*}
    y_0(t) & = c(t) + \int_a^t f(s,y_0(s))\,\d s \\
    & = c(t) + \int_a^{t_n} f(s,y_0(s))\,\d s + \int_{t_n}^t f(s,y_0(s))\,\d s \\
    & = c_n(t) + \int_{t_n}^t f(s,y_0(s))\,\d s.
\end{align*}
Therefore, by Proposition \ref{prop.uniqueness}, $y_0 \equiv y_n$ on $[t_n,T_0)$.
Consequently,
\[
y(t) \coloneqq \begin{cases}
    y_0(t) & \text{if } a \leq t < T_0, \\
    y_n(t) & \text{if } T_0 \leq t \leq T,
\end{cases}
\]
unambiguously defines a continuous function from $[a,T]$ to $\cX$.
By construction, $y$ satisfies the equation $y = c + \int_a^{\boldsymbol{\cdot}} f(t,y(t))\,\d t$ on $[a,T_0)$.
If $T_0 \leq t \leq T$, then
\begin{align*}
    y(t) & = y_n(t) = c_n(t) + \int_{t_n}^t f(s,y_n(s))\,\d s \\
    & = c(t) + \int_a^{t_n} f(s,y_0(s))\,\d s + \int_{t_n}^t f(s,y_n(s))\,\d s \\
    & = c(t) + \int_a^{t_n} f(s,y(s))\,\d s + \int_{t_n}^t f(s,y(s))\,\d s \\
    & = c(t) + \int_a^t f(s,y(s))\,\d s.
\end{align*}
Thus, $y$ satisfies $y = c + \int_a^{\boldsymbol{\cdot}} f(t,y(t))\,\d t$ on all of $[a,T]$, as desired.
\end{proof}

\begin{proof}[Proof of Theorem \ref{thm.IE}]
Combine Theorems \ref{thm.existuniq} and \ref{thm.blowup}.
\end{proof}

\subsection{Existence and uniqueness for the stochastic differential equation}\label{sec.subSDE}

Let $H$, $\cB$, $(c_e)_{e \in \cB}$, and $\W$ be as they were at the beginning of Section \ref{sec.NCfunc}.
The purpose of this section is to study an integral equation, interpreted as a kind of abstract stochastic differential equation (SDE), in our space of noncommutative $C^k$ functions.
To prepare for the main result, we introduce some extra terminology and notation.

\begin{definition}\label{def.CkLip}
Let $k \in \N_0$ and $\mathbf{f} \in \cC_1^k(\cB)$.
We say that $\mathbf{f}$ is \textbf{$\boldsymbol{\cC^k}$-Lipschitz on bounded sets} if for every $R > 0$, there exists a $C < \infty$ such that for all $k_0,\ldots,k_n \in \N_0$ with $j \coloneqq k_0+\cdots+k_n \leq k$, $(\cA,\tau) \in \W$, and $\mathbf{x},\mathbf{y} \in \cA_{c,\sa}^{\cB,j+1}$ with $\max\big\{ \norm{\mathbf{x}}_{\cA,c,j+1},\norm{\mathbf{y}}_{\cA,c,j+1}\big\} \leq R$,
Write
\[
\norm{\widetilde{\partial}^{(j_n)}\circ\cdots\circ \widetilde{\partial}^{(j_1)}\circ \partial^{(j_0)}f^{\cA}(\mathbf{x}) - \widetilde{\partial}^{(j_n)}\circ\cdots\circ \widetilde{\partial}^{(j_1)}\circ \partial^{(j_0)}f^{\cA}(\mathbf{y})}_{\cA^{j,j+1}} \leq C\norm{\mathbf{x}-\mathbf{y}}_{\cA,c,j+1}.
\]
In this case, the smallest such constant $C$ is denoted by $[\mathbf{f}]_{\operatorname{Lip},k,R}$.
If $C$ may be chosen independent of $R$, then $\mathbf{f}$ is said to be (\textbf{globally}) \textbf{$\boldsymbol{\cC^k}$-Lipschitz}, and the smallest constant $C$ (that works for all $R>0$) is denoted by $[\mathbf{f}]_{\operatorname{Lip},k}$.
We take $[\mathbf{f}]_{\operatorname{Lip},k} = \infty$ (or $[\mathbf{f}]_{\operatorname{Lip},k,R} = \infty$ for some $R > 0$) to mean $\mathbf{f}$ is not $\cC^k$-Lipschitz (on bounded sets).
\end{definition}

If $P \in \Tr^1(\cB)$, $j_0,\ldots,j_n$ are such that $j_0+\cdots+j_n = k$, $(\cA,\tau) \in \W$, and $\mathbf{x},\mathbf{y} \in \cA_{\sa}^{\cB,k+1}$, then, similarly to the proof of Proposition \ref{prop.nablaXiinterp}, one can compute
\[
\frac{\d}{\d t}\Big|_{t=0} \big(\widetilde{\partial}^{(j_n)}\circ\cdots\circ \widetilde{\partial}^{(j_1)}\circ \partial^{(j_0)}P\big)^{\cA}(\mathbf{x} + t\mathbf{y})
\]
in terms of the functions
\[
\big(\widetilde{\partial}^{(i_m)}\circ\cdots\circ \widetilde{\partial}^{(i_1)}\circ \partial^{(i_0)}P\big)^{\cA} \qquad (i_0+\cdots+i_m = k+1).
\]
A consequence of this computation is that there exists a universal constant $C_k$ such that
\[
\norm{\frac{\d}{\d t}\Big|_{t=0} \big(\widetilde{\partial}^{(j_n)}\circ\cdots\circ \widetilde{\partial}^{(j_1)}\circ \partial^{(j_0)}P\big)^{\cA}(\mathbf{x} + t\mathbf{y})}_{\cA^{k,k+1}} \leq C_k\norm{P}_{k+1,c,R}\norm{\mathbf{y}}_{\cA,c,k+1},
\]
where $R \coloneqq \norm{\mathbf{x}}_{\cA,c,k+1}$
By a limiting argument, the same statements are true with an arbitrary $\mathbf{f} \in \cC_1^k(\cB)$ in place of $P$.
As a result, if $\mathbf{f} \in \cC_1^{k+1}(\cB)$, $j_0,\ldots,j_n$ are such that $j \coloneqq j_0+\cdots+j_n \leq k$, $R > 0$, and $\mathbf{x},\mathbf{y} \in \cA_{c,\sa}^{\cB,j+1}$ are such that $\max\big\{\norm{\mathbf{x}}_{\cA,c,j+1},\norm{\mathbf{y}}_{\cA,c,j+1}\big\} \leq R$, then
\begin{align*}  
        \Big\|\widetilde{\partial}^{(j_n)}\circ\cdots\circ \widetilde{\partial}^{(j_1)} &\circ\partial^{(j_0)}f^{\cA}(\mathbf{x}) - \widetilde{\partial}^{(j_n)}\circ\cdots\circ \widetilde{\partial}^{(j_1)}\circ \partial^{(j_0)}f^{\cA}(\mathbf{y})\Big\|_{\cA^{j,j+1}} \\
        &= \norm{\int_0^1 \frac{\d}{\d t} \widetilde{\partial}^{(j_n)}\circ\cdots\circ \widetilde{\partial}^{(j_1)}\circ \partial^{(j_0)}f^{\cA}(t\mathbf{x}+(1-t)\mathbf{y})\, \d t}_{\cA^{j,j+1}} \\
        & \leq \int_0^1 \norm{\frac{\d}{\d t} \widetilde{\partial}^{(j_n)}\circ\cdots\circ \widetilde{\partial}^{(j_1)}\circ \partial^{(j_0)}f^{\cA}(t\mathbf{x}+(1-t)\mathbf{y})}_{\cA^{j,j+1}}\, \d t \\
        & \leq C_j\norm{\mathbf{f}}_{j+1,c,R}\norm{\mathbf{x} - \mathbf{y}}_{\cA,c,j+1}.
\end{align*}
In particular, if $\mathbf{f} \in \cC_1^{k+1}(\cB)$, then $\mathbf{f}$ is $\cC^k$-Lipschitz on bounded sets with
\begin{equation}
    [\mathbf{f}]_{\operatorname{Lip},k,R} \lesssim \max_{n=1,\ldots,k+1}\norm{\mathbf{f}}_{n,c,R} \qquad (R > 0).\label{eq.Lipestim}
\end{equation}
Consequently, if, in addition, $\norm{\mathbf{f}}_{\cC^n,\infty} \coloneqq \sup\big\{ \norm{\mathbf{f}}_{\cC^n,R} : R > 0 \big\} < \infty$ for all $n=1,\ldots,k+1$, then $\mathbf{f}$ is globally $\cC^k$-Lipschitz with $[\mathbf{f}]_{\operatorname{Lip},k} \lesssim \max\big\{\norm{\mathbf{f}}_{n,c,\infty} : n=1,\ldots,k+1\big\}$.

We are now prepared for the main result of this section.

\begin{theorem}\label{thm.keySDE}
Let $k \in \N_0$, and suppose $\mathbf{J} = (\mathbf{J}_1,\ldots,\mathbf{J}_d) \colon \R_+ \to \cC^{k+1}(X_1,\ldots,X_d)_{\sa}^d$ is a continuous function satisfying
\begin{equation}
    \sum_{n=1}^{k+1}\sup_{0 \leq t \leq T}\norm{\mathbf{J}_i(t)}_{n,\infty} < \infty \qquad (R,T > 0, \; i=1,\ldots,d). \label{eq.SDEhyp}
\end{equation}
If $\mathbf{c} = (\mathbf{c}_1,\ldots,\mathbf{c}_d) \colon \R_+ \to \cC^k(\cB)_{\sa}^d$ is a continuous function, then there exists a unique continuous function $\mathbf{X} \colon \R_+ \to \cC^k(\cB)_{\sa}^d$ such that
\begin{equation}
    \mathbf{X}(t) = \mathbf{c}(t) + \int_0^t \mathbf{J}(s,\mathbf{X}(s))\,\d s \qquad (t \geq 0),\label{eq.theSDE}
\end{equation}
i.e., $\mathbf{X}_i(t) = \mathbf{c}_i(t) + \int_0^t \mathbf{J}_i(s)\circ \mathbf{X}(s)\,\d s$ for all $i=1,\ldots,d$ and $t \geq 0$.
\end{theorem}

\begin{proof}
We begin with the uniqueness part.
Observe from Corollary \ref{cor:complip} that there exists a constant $C_k < \infty$ such that for all $\mathbf{K} \in \cC_1^{k+1}(X_1,\ldots,X_d)$, $R > 0$, and $\mathbf{f},\mathbf{g} \in \cC^k(\cB)_{\sa}^d$,
\begin{equation}
    \norm{\mathbf{K} \circ \mathbf{f} - \mathbf{K} \circ \mathbf{g}}_{\cC^k,R} \leq C_k(S+1)^k\norm{\mathbf{K}}_{\cC^{k+1},S}\norm{\mathbf{f} - \mathbf{g}}_{\cC^k,R},\label{eq.compLip}
\end{equation}
where $S \coloneqq \max\big\{ \norm{\mathbf{f}}_{\cC^k,R},\norm{\mathbf{g}}_{\cC^k,R} \big\}$.
Now, suppose $\mathbf{Y},\mathbf{X} \colon \R_+ \to \cC^k(\cB)_{\sa}^d$ both solve \eqref{eq.theSDE}, let $T,R > 0$, and define
\[
S_T \coloneqq \sup_{0 \leq t \leq T} \max\big\{\norm{\mathbf{X}(t)}_{\cC^k,R},\norm{\mathbf{Y}(t)}_{\cC^k,R}\big\} < \infty.
\]
By \eqref{eq.compLip}, if $0 \leq t \leq T$, then
\begin{align*}
    \|\mathbf{X}(t ) - \mathbf{Y}(t)\|_{\cC^k,R}& = \norm{\int_0^t \left(\mathbf{J}(s,\mathbf{X}(s)) - \mathbf{J}(s,\mathbf{Y}(s))\right)\d s}_{\cC^k,R}\\
     & \leq \int_0^t \norm{\mathbf{J}(s,\mathbf{X}(s)) - \mathbf{J}(s,\mathbf{Y}(s))}_{\cC^k,R} \,\d s \\
     & = \int_0^t \norm{\mathbf{J}(s) \circ \mathbf{X}(s) - \mathbf{J}(s)\circ \mathbf{Y}(s)}_{\cC^k,R} \,\d s \\
    & \leq C_k(S_T+1)^k\sup_{0 \leq s \leq T}\norm{\mathbf{J}(s)}_{\cC^{k+1},S_T}\int_0^t\norm{\mathbf{X}(s) - \mathbf{Y}(s)}_{\cC^k,R}\,\d s.
\end{align*}
It therefore follows from Gr\"onwall's inequality that $\norm{\mathbf{X}(t) - \mathbf{Y}(t)}_{\cC^k,R} = 0$ for all $t \in [0,T]$.
Since $R,T>0$ were arbitrary, we conclude that $\mathbf{X} = \mathbf{Y}$, as desired.

Next, we tackle the existence of a solution.
For each $R > 0$, let $\cC^k(\cB)_R$ be the Banach-space completion of $\Tr^1(\cB)$ with respect to $\norm{\cdot}_{\cC^k,R}$.
Also, let $\cC^k(\cB)_{\sa,R}$ be the real-linear subspace of $\cC^k(\cB)_R$ such that $f^{\cA}(x) \in \cA_{\sa}$ for all $x \in \cA_{c,\sa}^{\cB}$ with $\norm{x}_{\cA,c} \leq R$.
(The function $f^{\cA}$ should be defined analogously to the ``$R=\infty$'' case.)
We note for later use that if $R \geq S > 0$, then there are natural continuous ``restriction'' maps $\rho_{R,S} \colon \cC^k(\cB)_{\sa,R}^d \to \cC^k(\cB)_{\sa,S}^d$ and $\rho_{\infty,R} \colon \cC^k(\cB)_{\sa}^d \to \cC^k(\cB)_{\sa,R}^d$.

Let $R > 0$.
By \eqref{eq.Lipestim}, if $\mathbf{K} \in \cC^{k+1}(X_1,\ldots,X_d)$, then we may sensibly define $\mathbf{K} \circ \mathbf{f} \in \cC^k(\cB)_R$ for all $\mathbf{f} \in \cC^k(\cB)_{\sa,R}^d$,\footnote{More precisely, if $(P_n(X_1,\ldots,X_d))_{n \in \N}$ is a sequence of trace polynomials converging in $\cC^k(X_1,\ldots,X_d)$ to $\mathbf{K}$ and $(Q_n(X)) = (Q_{1,n}(X),\ldots,Q_{d,n}(X))_{n \in \N}$ is a sequence of $d$-tuples of self-adjoint trace polynomials converging in $\cC^k(\cB)_{\sa,R}^d$ to $\mathbf{f}$, then $(P_n(Q_n(X)))_{n \in \N}$ converges in $\cC^k(\cB)_R$ to a limit, $\mathbf{K} \circ \mathbf{f}$, independent of the choice of approximating sequences.}
and $\mathbf{K} \circ \mathbf{f} \in \cC^k(\cB)_{\sa,R}$ whenever $\mathbf{K} \in \cC^{k+1}(X_1,\ldots,X_d)_{\sa}$.
Consequently, it is easy to see the function
\[
\R_+ \times \cC^k(\cB)_{\sa,R}^d \ni (t,\mathbf{f}) \mapsto \mathbf{J}_R(t,\mathbf{f}) \coloneqq \mathbf{J}(t) \circ \mathbf{f} \in \cC^k(\cB)_{\sa,R}^d
\]
is well defined and continuous.
Furthermore, if $S > 0$ and $\mathbf{f},\mathbf{g} \in \cC^k(\cB)_{\sa,R}^d$ are such that
\[
\max\big\{\norm{\mathbf{f}}_{\cC^k,R},\norm{\mathbf{g}}_{\cC^k,R}\big\} \leq S,
\]
then \eqref{eq.Lipestim} yields
\[
\sup_{0 \leq t \leq T}\norm{\mathbf{J}_R(t, \mathbf{f}) - \mathbf{J}_R(t,\mathbf{g})}_{\cC^k,R} \leq C_k(S+1)^k\sup_{0 \leq t \leq T}\norm{\mathbf{J}(t)}_{\cC^{k+1},S}\norm{\mathbf{f} - \mathbf{g}}_{\cC^k,R}.
\]
Finally, if $\mathbf{c}_R \coloneqq \rho_{\infty,R} \circ \mathbf{c}$, then $\mathbf{c}_R \colon \R_+ \to \cC^k(\cB)_{\sa,R}^d$ is continuous.
Thus, the triple
\[
(\cX,f,c) \coloneqq \big(\cC^k(\cB)_{\sa,R}^d, \mathbf{J}_R,\mathbf{c}_R\big)
\]
satisfies the hypotheses of Theorem \ref{thm.IE}.

By Theorem \ref{thm.IE}, there exists a unique maximal solution $\mathbf{X}_R \colon [0,T_{\ell}) \to \cC^k(\cB)_{\sa,R}^d$ to the integral equation $\mathbf{X}_R = \mathbf{c}_R + \into  \mathbf{J}_R(t,\mathbf{X}_R(t))\,\d t$.
We claim that $T_{\ell} = \infty$.
We shall prove this using the blow-up criterion in Theorem \ref{thm.IE}.
Specifically, we shall prove that if $0 < T \leq T_{\ell}$ and $T < \infty$, then $\sup \big\{ \norm{\mathbf{X}_R(t)}_{\cC^k,R} : 0 \leq t < T  \big\} < \infty$, which will establish the claim.
To this end, observe from Corollary \ref{cor.Compbd} that there exist constants $A_1,\ldots,A_k < \infty$ such that for all $\mathbf{K} \in \cC_1^{k+1}(X_1,\ldots,X_d)$, $n=1,\ldots,k$, $R> 0$, and $\mathbf{f} \in \cC^k(\cB)_{\sa,R}^d$,
\begin{equation}
    \norm{\mathbf{K} \circ \mathbf{f}}_{n,c,R} \leq A_n\left( \norm{\mathbf{K}}_{1,\infty} \norm{\mathbf{f}}_{n,c,R} + \Big( \norm{\mathbf{f}}_{\cC^{n-1},R}^n+\norm{\mathbf{f}}_{\cC^{n-1},R}^2\Big)\times \max_{j=2,\ldots,n}\norm{\mathbf{K}}_{j,\infty}  \right).\label{eq.compboundn}
\end{equation}
In addition, observe from \eqref{eq.Lipestim} that if $\mathbf{K} \in \cC_1^1(X_1,\ldots,X_d)$, $\mathbf{f} \in \cC^0(\cB)_{\sa}$, and $R > 0$, then
\begin{equation}
    \norm{\mathbf{K} \circ \mathbf{f}}_{0,c,R} \leq \norm{\mathbf{K}}_{1,\infty}\norm{\mathbf{f}}_{0,c,R} + \norm{\mathbf{K} \circ \mathbf{0}}_{0,c,\infty} < \infty.\label{eq.compbound0}
\end{equation}
Now, by \eqref{eq.compbound0}, if $0 \leq t < T$, then
\begin{align*}
    \norm{\mathbf{X}_R(t)}_{0,c,R} & = \norm{\mathbf{c}_R(t) + \int_0^t \mathbf{J}_R(s,\mathbf{X}_R(s)) \,\d s}_{0,c,R} \\
    & \leq \sup_{0 \leq t \leq T} \norm{\mathbf{c}(t)}_{0,c,R} + \int_0^t \norm{\mathbf{J}_R(s,\mathbf{X}_R(s))}_{0,c,R}\,\d s \\
    & \leq \sup_{0 \leq t \leq T} \norm{\mathbf{c}(t)}_{0,c,R} + \int_0^t \left( \norm{\mathbf{J}_R(s,\mathbf{0})}_{0,c,R}+ \norm{\mathbf{J}(s)}_{1,\infty} \norm{\mathbf{X}_R(s)}_{0,c,R}\right)\d s \\
    & \leq  \sup_{0 \leq t \leq T} \norm{\mathbf{c}(t)}_{0,c,R} + \sup_{0 \leq s \leq T}\norm{\mathbf{J}(s) \circ \mathbf{0}}_{0,c,R}T + \sup_{0 \leq s \leq T}\norm{\mathbf{J}(s)}_{1,\infty} \int_0^t \norm{\mathbf{X}_R(s)}_{0,c,R}\,\d s.
\end{align*}
Consequently, by Gr\"onwall's inequality,
\[
\norm{\mathbf{X}_R(t)}_{0,c,R} \leq\left(\sup_{0 \leq t \leq T} \norm{\mathbf{c}(t)}_{0,c,R} + \sup_{0 \leq s \leq T}\norm{\mathbf{J}(s) \circ \mathbf{0}}_{0,c,R}T\right)\exp\left(\sup_{0 \leq s \leq T}\norm{\mathbf{J}(s)}_{1,\infty}t\right)
\]
for all $t \in [0,T)$.
In particular, $\sup\big\{ \norm{\mathbf{X}_R(t)}_{0,c,R} : 0 \leq t < T\big\} < \infty$.
Now, suppose $n=1,\ldots,k$, and assume $\sup\big\{ \norm{\mathbf{X}_R(t)}_{j,c,R} : 0 \leq t < T, \; j=0,\ldots,n-1\big\} < \infty$.
In particular,
\[
M_0 \coloneqq \sup_{0 \leq t < T}\norm{\mathbf{X}_R(t)}_{\cC^{n-1},R} < \infty.
\]
Also, write $M_1(t) \coloneqq \max\big\{ \norm{\mathbf{J}(t)}_{j,\infty} : j=2,\ldots,n\big\}$.
By \eqref{eq.compboundn}, if $0 \leq t < T$, then
\begin{align*}
   \|\mathbf{X}_R (t)\|_{n,c,R} & = \norm{\mathbf{c}_R(t) + \int_0^t \mathbf{J}_R(s,\mathbf{X}_R(s)) \,\d s}_{n,c,R} \\
   & \leq \sup_{0 \leq s \leq T}\norm{\mathbf{c}_R(s)}_{n,c,R} + \int_0^t \norm{\mathbf{J}_R(s,\mathbf{X}_R(s))}_{n,c,R} \,\d s\\
    & \leq \sup_{0 \leq s \leq T}\norm{\mathbf{c}_R(s)}_{n,c,R} + A_n\int_0^t\left(\norm{\mathbf{J}(s)}_{1,\infty}\norm{\mathbf{X}_R(s)}_{n,c,R} \right. \\
    & \qquad\qquad\qquad\qquad\left. + \Big( \norm{\mathbf{X}_R(s)}_{\cC^{n-1},R}^n+\norm{\mathbf{X}_R(s)}_{\cC^{n-1},R}^2\Big)M_1(s) \right)\d s \\
    & \leq \sup_{0 \leq s \leq T}\norm{\mathbf{c}_R(s)}_{n,c,R} +A_n(M_0^n+M_0^2)\sup_{0 \leq s \leq T}M_1(s)\,T \\
    & \qquad\qquad \qquad\qquad+ A_n\sup_{0 \leq s \leq T}\norm{\mathbf{J}(s)}_{1,\infty}\int_0^t \norm{\mathbf{X}_R(s)}_{n,c,R}\,\d s.
\end{align*}
Consequently, by Gr\"onwall's inequality,
\begin{align*}
    \norm{\mathbf{X}_R(t)}_{n,c,R} & \leq \left( \sup_{0 \leq s \leq T}\norm{\mathbf{c}_R(s)}_{n,c,R} +A_n(M_0^n+M_0^2)\sup_{0 \leq s \leq T}M_1(s)\,T\right) \exp\left(A_n\sup_{0 \leq s \leq T}\norm{\mathbf{J}(s)}_{1,\infty}t\right)
\end{align*}
for all $t \in [0,T)$.
In particular, $\sup \big\{ \norm{\mathbf{X}_R(t)}_{n,c,R} : 0 \leq t < T\big\} < \infty$.
This establishes that
\[
\sup_{0 \leq t < T}\norm{\mathbf{X}_R(t)}_{\cC^n,R} < \infty,
\]
so we are done proving that $T_{\ell} = \infty$.

To complete the proof, we must glue the solutions $\left\{ \mathbf{X}_R : R > 0\right\}$ together.
Though a slight headache, the argument is routine, so we only sketch it briefly.
First, show that if $\mathbf{f}_R \in \cC^k(\cB)_{\sa,R}^d$ for all $R > 0$ and $\rho_{R,S}(\mathbf{f}_R) = \mathbf{f}_S$ for all $R \geq S > 0$, then there exists a unique $\mathbf{f} \in \cC^k(\cB)_{\sa}^d$ such that $\rho_{\infty,R}(\mathbf{f}) = \mathbf{f}_R$ for all $R > 0$;
furthermore, $\norm{\mathbf{f}}_{\cC^k,R} = \norm{\mathbf{f}_R}_{\cC^k,R}$ for all $R > 0$.
With this in mind, by applying $\rho_{R,S}$ to both sides of $\mathbf{X}_R = \mathbf{c}_R + \into \mathbf{J}_R(s,\mathbf{X}_R(s))\,\d s$, we obtain that $t \mapsto y(t) \coloneqq \rho_{R,S}(\mathbf{X}_R(t))$ solves $y = \mathbf{c}_S + \into \mathbf{J}_S(s,y(s))\,\d s$.
It then follows from the uniqueness of solutions in Theorem \ref{thm.IE} that $\rho_{R,S}(\mathbf{X}_R(t)) = \mathbf{X}_S(t)$ for all $t \geq 0$.
Consequently, for each $t \geq 0$, there exists a unique $\mathbf{X}(t) \in \cC^k(\cB)_{\sa}^d$ such that $\rho_{\infty,R}(\mathbf{X}(t)) = \mathbf{X}_R(t)$ for all $R > 0$.
Finally, check that $\mathbf{X} \colon \R_+ \to \cC^k(\cB)_{\sa}^d$ is continuous and satisfies $\mathbf{X} = \mathbf{c} + \into \mathbf{J}(t,\mathbf{X}(t))\,\d t$ on $\R_+$, which completes the proof.
\end{proof}

For our purposes, the most important example of a $\mathbf{c} \colon \R_+ \to \cC^k(\cB)_{\sa}^d$ in Theorem \ref{thm.keySDE} is one of the form $\mathbf{c} = \mathbf{f}_h = (\mathbf{f}_{h_1},\ldots,\mathbf{f}_{h_d})$, where $h = (h_1,\ldots,h_d) \colon \R_+ \to H_c^d$ a continuous function with the property that $\ip{e,h_i(t)} \in \R$ for all $i=1,\ldots,d$, $t \geq 0$, and $e \in \cB$.
(Please see Example \ref{ex.Sasfunc} for the $\mathbf{f}_{h_i}$ notation.)
A particular instance of this example also explains our use of the word ``stochastic'' in the title of this subsection.

The reader should note well that there is nothing stochastic about Theorem \ref{thm.keySDE}.
Indeed, it is an integral equation that we used the ODE-type result Theorem \ref{thm.IE} to prove.
However, in our applications of interest (Section \ref{sec.transport} and the next subsection), Theorem \ref{thm.keySDE} will give rise to (free) SDEs when evaluated on certain elements of certain tracial von Neumann algebras.
Specifically, suppose $H = L^2(\R_+)^d \oplus \C^d$ and $\cB = \cB_1 \sqcup \cdots \sqcup \cB_d$ is the orthonormal basis of $H$ defined as follows:
For each $i=1,\ldots,d$, let $\cB_i \coloneqq \cC_i \cup \{e_i\}$, where $\cC_i$ is the Haar basis from Definition \ref{def:HaarBasis} of the $i$th copy of $L^2(\R_+)$ in $H$ and $e_i$ is the $i$th standard basis element.
Also, let $c = (c_e)_{e \in \cB}$ be the weights those from Definition \ref{def:HaarBasis} on $(c_e)_{e \in \cC_i}$ and $c_{e_i} \coloneqq 1$ ($i=1,\ldots,d$).
Now, for each $i=1,\ldots,d$, let $h_i \colon \R_+ \to H_c$ be the path defined by
\[
h_i(t) \coloneqq (0,\ldots,0,\underbrace{1_{[0,t]}}_{\text{spot } i},0,\ldots,0,e_i) \in H_c \qquad (t \geq 0),
\]
and define $h \coloneqq (h_1,\ldots,h_d) \colon \R_+ \to H_c^d$.
Finally, let $\mathbf{X}$ be the solution to $\mathbf{X} = \mathbf{f}_h + \into \mathbf{J}(t,\mathbf{X}(t))\,\d t$ provided by Theorem \ref{thm.keySDE}.
If $(x_e)_{e \in \cC_1\cup \cdots \cup \cC_d}$ are i.i.d.\ $N \times N$ GUEs on the probability space $(\Om,\sF,P)$, $(x_{1_1},\ldots,x_{1_d}) \coloneqq (x_{0,1},\ldots,x_{0,d}) = x_0$ is any $\MnC_{\sa}^d$-valued random variable independent of $(x_e)_{e \in \cC_1\cup \cdots \cup \cC_d}$, and $(\cA,\tau) \coloneqq (\MnC,\tr_N)$, then for almost all $\om \in \Om$, $(x_e)_{e \in \cB} \in \MnC_{\sa,c}^{\cB}$, and the (a.s.-defined) stochastic process
\[
\R_+ \times \Om \ni (t,\om) \mapsto f_{h(t)}^{\MnC}\big((x_e(\om))_{e \in \cB}\big) = \underbrace{\begin{bmatrix}
    f_{1_{[0,t]}}^{\MnC}\big((x_e(\om))_{e \in \cC_1}\big) \\
    \vdots & \\
    f_{1_{[0,t]}}^{\MnC}\big((x_e(\om))_{e \in \cC_d}\big)
\end{bmatrix}}_{S^N(t)(\om)} + x_0 \in \MnC_{\sa}^d
\]
is an $\MnC_{\sa}^d$-valued Brownian motion starting at $x_0$.
(Note that on the right-hand side above, we are changing Hilbert spaces to just $L^2(\R_+)$ and thereby slightly abusing notation.)
Thus, the (a.s.-defined) stochastic process
\[
\R_+ \times \Om \ni (t,\om) \mapsto X^N(t)(\om) \coloneqq X^{\MnC}(t)((x_e(\om))_{e \in \cB}) = \begin{bmatrix}
    X_1^{\MnC}(t)((x_e(\om))_{e \in \cB}) \\
    \vdots & \\
    X_d^{\MnC}(t)((x_e(\om))_{e \in \cB})
\end{bmatrix} \in \MnC_{\sa}^d
\]
solves the $\MnC_{\sa}^d$-valued SDE
\[
\begin{cases}
    \d X^N(t) = \d S^N(t) + J^{\MnC}(t,X_n(t))\,\d t& \\
   \;\, X^N(0) = x_0, & 
\end{cases}
\]
i.e., $X^N = (X_1^N,\ldots,X_d^N)$ solves the coupled $\MnC_{\sa}$-valued SDEs
\[
\begin{cases}
    \d X_i^N(t) = \d S_i^N(t) + J_i^{\MnC}(t,X_n(t))\,\d t& \\
   \;\, X_i^N(0) = x_{0,i} & 
\end{cases} \qquad (i=1,\ldots,d).
\]
Similarly, if $(x_e)_{e \in \cC_1 \cup \cdots \cup \cC_d}$ are, instead, f.i.i.d.\ semicirculars in some tracial von Neumann algebra $(\cA,\tau)$, and $(x_{1_1},\ldots,x_{1_d}) \coloneqq (x_{0,1},\ldots,x_{0,d}) = x_0 \in \cA_{\sa}^d$ is an arbitrary $d$-tuple freely independent of $(x_e)_{e \in \cC_1 \cup \cdots \cup \cC_d}$, then the free stochastic process
\[
\R_+ \ni t \mapsto f_{h(t)}^{\cA}\big((x_e)_{e \in \cB}\big) = \underbrace{\begin{bmatrix}
    f_{1_{[0,t]}}^{\cA}\big((x_e)_{e \in \cC_1}\big) \\
    \vdots \\
    f_{1_{[0,t]}}^{\cA}\big((x_e)_{e \in \cC_d}\big)
\end{bmatrix}}_{S(t)} + x_0 \in \cA_{\sa}^d
\]
is a semicircular Brownian motion starting at $x_0$.
Thus, the free stochastic process
\[
\R_+ \ni t \mapsto X(t) \coloneqq X^{\cA}(t)((x_e)_{e \in \cB}) \in \cA_{\sa}^d
\]
solves the free SDE
\[
\begin{cases}
    \d X(t) = \d S(t) + J^{\cA}(t,X(t))\,\d t& \\
   \;\, X(0) = x_0, & 
\end{cases}
\]
i.e., $X \coloneqq (X_1,\ldots,X_d)$ solves the coupled free SDEs
\[
\begin{cases}
    \d X_i(t) = \d S_i(t) + J_i^{\cA}(t,X(t))\,\d t& \\
   \;\, X_i(0) = x_{0,i} & 
\end{cases} \qquad (i=1,\ldots,d).
\]
This is the reason we interpret $\mathbf{X} = \mathbf{f}_h + \into \mathbf{J}(t,\mathbf{X}(t))\,\d t$ as a stochastic differential equation.
In the next subsection, we establish important bounds on certain seminorms of the solution in a special case of interest.

\subsection{Estimate on the large-time asymptotics}

In this section we study the equation, 
\begin{equation}
    \label{skvjdnand}
    \bX_{t,p}(x,X) = x_p -\frac{1}{2} \int_0^t \bJ_p\circ\bX_s(x,X) ds + S_{t,p}(X),\quad 1\leq p\leq d,
\end{equation}
 where $(x,X)\in \cA^{\cB}_{c,\sa} $. In particular, $\cB = \cU \cup \cV$ with $\cU$ the orthonormal basis of $L^2(\R_+)^{\oplus d}$ defined with the Haar basis of $L^2(\R_+)$ from Definition \ref{def:HaarBasis} and $\cV=\{e_1,\dots,e_d\}$ the standard basis of $\C^d$. Besides, as in Example \ref{ex.Sasfunc}, if $\1_{[0,t]}^p$ is viewed as an element of the $p$-th space in the direct sum $L^2(\R_+)^{\oplus d}$,
\begin{equation}
    \label{skgjnslvms}
    S_{t,p} = \sum_{f\in\cU}\ \la \1_{[0,t]}^p | f\ra\ X_f.
\end{equation}
In particular, if $\cW$ is the Haar basis of $L^2(\R_+)$ then $S_{t,p} = \sum_{e\in\cW}\ \la \1_{[0,t]} | e\ra\ X_{e^p}$. In particular, if we evaluate $S_{t,p}$ in a family of independent Gaussian random variables, or GUE random matrix, or free semicircular, we get respectively the classical Brownian motion, the Hermitian Brownian motion, and the free Brownian motion.
 
 Theorem \ref{thm.keySDE} prove that, assuming $\bJ\in (\cC^{k+1}(\cB))^d$, there exists a continuous solution $(\bX_{t,p})_{t\geq 0, 1\leq p\leq d}$ such that for any $t,p$, $\bX_{t,p}\in\cC^k(\cB)$. Besides, it also shows that 
\begin{equation}
    \label{skvsjdvb}
     \partial \bX_{t,p} = e_p\otimes 1\otimes 1 -\frac{1}{2} \int_0^t \partial(\bJ_p\circ\bX_s) ds + \1_{[0,t]}^p\otimes 1 \otimes 1,\quad 1\leq p\leq d, 
\end{equation}
anf if $j_0>1$ or $j_1>0$,
\begin{equation}
    \label{soivsvd}
    \widetilde{\partial}^{(j_n)}\circ\cdots\circ \widetilde{\partial}^{(j_1)}\circ \partial^{(j_0)} \bX_{t,p} = -\frac{1}{2} \int_0^t \widetilde{\partial}^{(j_n)}\circ\cdots\circ \widetilde{\partial}^{(j_1)}\circ \partial^{(j_0)} (\bJ_p\circ\bX_s) ds ,\quad 1\leq p\leq d.
\end{equation}

\begin{theorem}
\label{thm:fastdecay}
    With $\bX$ the solution to Equation \eqref{skvjdnand}, assuming $\bJ\in (\cC^{k+1}(\cB))^d$, and that for all $p\in [1,d]$, $\norm{\bJ_p - X_p }_{\cC^k,\infty} <\infty$, then we set
    $$\kappa \coloneqq \sup_{R>0}\ \left\{ \sum_{1\leq i\leq d} \norm{\partial_i\bJ_p - \1_{i=p} 1\otimes 1}_{0,R} + \norm{\widetilde{\partial}_i \bJ_p }_{0,R} \right\},$$
    where $\norm{\cdot}_{0,R}$ is in Definition \ref{def.norms-of-derivatives}.  And there exists a constant $C_{\bJ}$ such that for any $j_0,\dots,j_n$ with $j\coloneqq j_0+\dots+j_n \in [1,k]$, $R>0$, and $i,p\in [1,d]$,
    \begin{align}
        \label{eq:fastdecay}
        \norm{ \la e_i| \widetilde{\partial}^{(j_n)}\circ\cdots\circ \widetilde{\partial}^{(j_1)}\circ \partial^{(j_0)} \bX_{t,p}}_{{0,c,R}} &\leq C_{\bJ} (1+t)^{2j} \times e^{-\frac{1-j\kappa}{2}t}, \\
        \norm{\widetilde{\partial}^{(j_n)}\circ\cdots\circ\widetilde{\partial}^{(j_1)}\circ\partial^{(j_0)} \bX_{t,p}}_{0,c,R} &\leq C_{\bJ} (1+t)^{2j} e^{\frac{\kappa j}{2}t}.          \label{eq:fastdecay2}
    \end{align}
\end{theorem}

\begin{proof}

\textbf{Step 1:}  First note that thanks to the chain rule, i.e., Theorem \ref{chainrule}, we have
    $$ \partial (\bJ_p\circ \bX_t) = \sum_{j=1}^d  \partial_j \bJ_p\sh_1 \partial \bX_{t,j}. $$
    Therefore
    \begin{equation}
        \label{eq:sldvjn}
        \norm{ \partial \left( \bJ_p \circ \bX_t - \bX_{t,p} \right) }_{0,c,R} \leq \kappa \times  \sup_p\ \norm{ \partial \bX_{t,p}}_{0,c,R}.
    \end{equation}
    Besides,
    $$  \partial \bX_{t,p} = (e_p + \1_{[0,t]}^p)\otimes 1\otimes 1 -\frac{1}{2} \int_0^t \partial(\bJ_p\circ\bX_s) ds.$$
    Hence, since $\la e_i|\1^p_{[0,t]} \ra =0$, $\la e_i| \partial \bX_{t,p}$ is actually differentiable with respect to $t$, and
    $$ \frac{d}{dt} \left(e^{t/2}\la e_i| \partial \bX_{t,p}\right) = \frac{1}{2} e^{t/2}\la e_i| \partial \bX_{t,p} - \frac{1}{2} e^{t/2}\la e_i| \partial(\bJ_p\circ\bX_t) = \frac{1}{2} e^{t/2}\la e_i| \partial( \bX_{t,p} - \bJ_p\circ\bX_t). $$
    
    \noindent Thus if we set $\alpha(t)\coloneqq e^{t/2} \sup_p \norm{ \la e_i| \partial \bX_{t,p}}_{0,c,R} $, thanks to Equation \eqref{eq:sldvjn}, we get that
    $$ \alpha(t) \leq 1 + \frac{\kappa}{2} \int_0^t \alpha(s) ds. $$
    So by Gr\"onwall's inequality, we get that $\alpha(t)\leq e^{\frac{\kappa}{2}t}$, which implies that for all $p,i\in [1,d]$,
    \begin{equation}
        \label{slvdknsldkvms}
        \norm{ \la e_i| \partial \bX_{t,p}}_{0,c,R} \leq e^{-\frac{1-\kappa}{2}t}.
    \end{equation}

\textbf{Step 2:} Let us upper bound $\norm{\partial \bX_{t,p}}_{0,c,R}$ which we will need for further computations. First, we have 
\begin{align*}
    &e^{t/2}\ \partial \bX_{t,p} \\
    &= \frac{1}{2}\int_0^t e^{u/2}du \left( (e_p+\1_{[0,t]}^p)\otimes 1\otimes 1 - \frac{1}{2} \int_0^t \partial(\bJ_p\circ\bX_s) ds \right) \\
    &= e^{t/2} (e_p+\1_{[0,t]}^p)\otimes 1\otimes 1 -\frac{1}{4} \int_0^t\int_0^s e^{u/2} du\ \partial(\bJ_p\circ\bX_s) ds - \frac{1}{4} \int_0^t\int_0^u \partial(\bJ_p\circ\bX_s) ds\  e^{u/2} du \\
    &= e^{t/2} (e_p+\1_{[0,t]}^p)\otimes 1\otimes 1  -\frac{1}{2} \int_0^t (e^{s/2}-1) \partial(\bJ_p\circ\bX_s) ds + \frac{1}{2} \int_0^t \left( \partial \bX_{u,p} - (e_p+\1_{[0,u]}^p)\otimes 1\otimes 1 \right)  e^{u/2} du \\
    &= e^{t/2} (e_p+\1_{[0,t]}^p)\otimes 1\otimes 1 - \frac{1}{2} \int_0^t (e_p+\1_{[0,u]}^p)\otimes 1\otimes 1  e^{u/2} du \\
    &\quad + \frac{1}{2} \int_0^t e^{s/2}\ \partial(\bX_{s,p} - \bJ_p\circ\bX_s) ds + \frac{1}{2} \int_0^t \partial(\bJ_p\circ\bX_s) ds.
\end{align*}
Note that thanks to Proposition \ref{prop.indicatorHaarweightnorm}, for some constant $C$,
$$\norm{e^{t/2} (e_p+\1_{[0,t]}^p)\otimes 1\otimes 1 - \frac{1}{2} \int_0^t (e_p+\1_{[0,u]}^p)\otimes 1\otimes 1  e^{u/2} du}_c \leq C (t+1)^2 e^{t/2}.$$
Therefore, thanks to Equation \ref{skvjnlsdvkn}, if $\alpha(t) = e^{t/2}\ \norm{\partial \bX_{t,p}}_{0,c,R}$, we have that
\begin{align*}
    \alpha(t) \leq C(1+t)^2 e^{t/2} + \int_0^t \left( \frac{\kappa}{2} + \frac{e^{-s/2}}{2} \norm{\bJ_p}_{1,c,\infty} \right) \alpha(s) ds.
\end{align*}
Thus, by Gr\"onwall inequality, 
$$ \alpha(t) \leq C(1+t)^2 e^{t/2} \exp\left( \int_0^t \frac{\kappa}{2} + \frac{e^{-s/2}}{2} \norm{\bJ_p}_{1,c,\infty} ds\right) \leq C(1+t)^2 e^{t/2} e^{\frac{\kappa}{2}t+\norm{\bJ_p}_{1,c,\infty}}.$$
Hence, for some constant $C_{\bJ}$,
\begin{equation}
    \label{sdlvjnsovdm}
    \norm{\partial \bX_{t,p}}_{0,c,R} \leq C_{\bJ}(1+t)^2 e^{\frac{\kappa}{2}t}.
\end{equation}

\textbf{Step 3:} Let us prove Equation \eqref{eq:fastdecay2} by induction. We set $j\coloneqq j_0+\dots+j_n$. Thanks to Equation \eqref{sdlvjnsovdm} this is true if $j=j_0=1$. Else, if $j\geq 2$ or $j_1\geq 1$, then thanks to Equation \eqref{soivsvd}, one can differentiate $t\mapsto e^{t/2} \widetilde{\partial}^{(j_n)}\circ\cdots\circ\widetilde{\partial}^{(j_1)}\circ\partial^{(j_0)} \bX_{t,p} $, and
\begin{align}
\label{sjchbskjcnks}
        \frac{d}{dt} \left(e^{t/2}\ \widetilde{\partial}^{(j_n)}\circ\cdots\circ\widetilde{\partial}^{(j_1)}\circ\partial^{(j_0)} \bX_{t,p}\right) = \frac{1}{2} e^{t/2}\ \widetilde{\partial}^{(j_n)}\circ\cdots\circ\widetilde{\partial}^{(j_1)}\circ\partial^{(j_0)}( \bX_{t,p} - \bJ_p\circ\bX_t).
\end{align}
Thus if $j_0>0$, then by Equation \eqref{chainestimate}, if $\alpha(t)\coloneqq e^{t/2} \sup_p \norm{\widetilde{\partial}^{(j_n)}\circ\cdots\circ\widetilde{\partial}^{(j_1)}\circ\partial^{(j_0)} \bX_{t,p}}_{0,c,R}$, there exists a constant $C_{\bJ}$ such that,
$$ \alpha(t) \leq \kappa \int_0^t \alpha(s) + C_{\bJ}\times e^{s/2} \Lambda^j_R(\bX_s) \ ds. $$

\noindent Besides, by induction, since in the definition of $\Lambda^j_R$ we have $m\geq 2$, thanks to Equation \eqref{sdlvjnsovdm} and our induction hypothesis, there exist a constant $C_{\bJ}$ such that,
$$\Lambda^j_R(\bX_s) \leq C_{\bJ} (1+t)^{2j} e^{\frac{\kappa j}{2}t}.$$
Thus for some constant $C_{\bJ}$,
$$ \alpha(t) \leq \kappa \int_0^t \alpha(s) ds + C_{\bJ}\times e^{t/2} (1+t)^{2j} e^{\frac{\kappa j}{2}t}. $$

\noindent Therefore, by Gr\"onwall lemma, there exist a constant $C_{\bJ}$ such that,
$$ \alpha(t) \leq C_{\bJ}\times (1+t)^{2j} \left( e^{t/2}e^{\frac{\kappa j}{2}t} + \kappa \int_0^t e^{s/2} e^{\frac{\kappa j}{2}s} e^{\frac{\kappa}{2}(t-s)} ds  \right) \leq  C_{\bJ} (1+2\kappa)\times (1+t)^{2j} e^{t/2}e^{\frac{\kappa j}{2}t}, $$
which implies Equation \eqref{eq:fastdecay2}.

If we now assume that $j_0=0$, then Equation \eqref{sjchbskjcnks} remains true, but instead of using Equation \eqref{chainestimate} we need to use Equation \eqref{chainestimate2}, this yields that for some constant $C_{\bJ}$,
$$ \alpha(t) \leq \kappa \int_0^t \alpha(s) + C_{\bJ}\times e^{s/2} \left( \Lambda^j_R(\bX_s)  + \sup_p\ \norm{\widetilde{\partial}^{(j_n)}\circ\cdots\circ\widetilde{\partial}^{(j_2)} \circ \partial^{(j_1)}\mathbf{X}_{s,p} }_{0,c,R} \right) \ ds. $$
And so by using the case $j_0>0$ which we just proved, we still get that
$$ \alpha(t) \leq \kappa \int_0^t \alpha(s) ds + C_{\bJ}\times e^{t/2} (1+t)^{2j} e^{\frac{\kappa j}{2}t}. $$
And from there the proof follows as in the previous case.

\textbf{Step 4:} Let us now prove Equation \eqref{eq:fastdecay} by induction on $j\coloneqq j_0+\dots+j_n$. If $j=j_0=1$ then we know that \eqref{eq:fastdecay} is true thanks to our first step. Else, thanks to Equation \eqref{soivsvd}, we have 
\begin{equation}
    \label{soivsvd23}
    \la e_i| \widetilde{\partial}^{(j_n)}\circ\cdots\circ \widetilde{\partial}^{(j_1)}\circ \partial^{(j_0)} \bX_{t,p} = -\frac{1}{2} \int_0^t \la e_i| \widetilde{\partial}^{(j_n)}\circ\cdots\circ \widetilde{\partial}^{(j_1)}\circ \partial^{(j_0)} (\bJ_p\circ\bX_s) ds ,\quad 1\leq p\leq d.
\end{equation}
In particular, one can differentiate with respect to $t$ and we get that
    \begin{align*}
        &\frac{d}{dt} \left(e^{t/2}\la e_i| \widetilde{\partial}^{(j_n)}\circ\cdots\circ\widetilde{\partial}^{(j_1)}\circ\partial^{(j_0)} \bX_{t,p}\right) = \frac{1}{2} e^{t/2}\la e_i| \widetilde{\partial}^{(j_n)}\circ\cdots\circ\widetilde{\partial}^{(j_1)}\circ\partial^{(j_0)}( \bX_{t,p} - \bJ_p\circ\bX_t).
    \end{align*}
    So if we set $\alpha(t)\coloneqq e^{t/2} \sup_p \norm{ \la e_i| \widetilde{\partial}^{(j_n)}\circ\cdots\circ\widetilde{\partial}^{(j_1)}\circ\partial^{(j_0)} \bX_{t,p}}_{0,c,R}$, and we assume that $j_0>1$, then thanks to Equation \eqref{chainestimate}, we get that
    $$ \alpha(t) \leq \kappa \int_0^t \alpha(s) + C_{\bJ}\times e^{s/2} \widetilde{\Lambda}^j_R(\bX_s) \ ds, $$
    where 
    $$ \widetilde{\Lambda}^s_R(\cf) = \sup_{\substack{1\leq i_1,\dots,i_m\leq d \\ s_1+\dots+s_m = s \\ s_l\geq 1,\ m\geq 2 }}\ \left( \sup_{a_0+\dots+a_q=s_1}\ \norm{\la e_i| \widetilde{\partial}^{(a_q)}\circ\cdots\circ\widetilde{\partial}^{(a_1)}\circ\partial^{(j_0)} \cf_{i_l}}_{0,c,R} \right) \prod_{l=2}^m\ \norm{\cf_{i_l}}_{s_l,c,R}, $$
    And thanks to our induction hypothesis as well as Equation \eqref{eq:fastdecay2} we get that
    $$ \alpha(t) \leq \kappa \int_0^t \alpha(s) ds + C_{\bJ}\times (1+t)^{2j} e^{\frac{\kappa j}{2}t}. $$

    \noindent Therefore, by Gr\"onwall lemma, there exist a constant $C_{\bJ}$ such that,
    $$ \alpha(t) \leq C_{\bJ}\times (1+t)^{2j} \left(e^{\frac{\kappa j}{2}t} + \kappa \int_0^t e^{\frac{\kappa j}{2}s} e^{\frac{\kappa}{2}(t-s)} ds  \right) \leq 3 C_{\bJ} \times (1+t)^{2j} e^{\frac{\kappa j}{2}t}, $$
    which implies Equation \eqref{eq:fastdecay}.
    
    Let us now assume that $j_0=0$, then thanks to the chain rule, and more precisely Equation \eqref{chainestimate2},
    $$ \alpha(t) \leq \kappa \int_0^t \alpha(s) + C_{\bJ}\times e^{s/2} \left( \widetilde{\Lambda}^j_R(\bX_s) + \sup_{1\leq i,p\leq d}\ \norm{ \la e_i| \widetilde{\partial}^{(j_n)}\circ\cdots\circ\widetilde{\partial}^{(j_2)} \circ \partial^{(j_1)}\bX_{s,p} }_{0,c,R} \right) \ ds. $$
    And so by using the case $j_0>0$ which we just proved, we still get that
    $$ \alpha(t) \leq \kappa \int_0^t \alpha(s) ds + C_{\bJ}\times (1+t)^{2j} e^{\frac{\kappa j}{2}t}. $$
    Therefore, by Gr\"onwall lemma, there exist a constant $C_{\bJ}$ such that, if $j>1$,
    $$ \alpha(t) \leq C_{\bJ}\times (1+t)^{2j} \left(e^{\frac{\kappa j}{2}t} + \kappa \int_0^t e^{\frac{\kappa j}{2}s} e^{\frac{\kappa}{2}(t-s)} ds  \right) \leq 3 C_{\bJ} \times (1+t)^{2j} e^{\frac{\kappa j}{2}t}, $$
    which implies Equation \eqref{eq:fastdecay}.

    Finally, for the case $j=j_1=1$, we have
    $$ \alpha(t) \leq \kappa \int_0^t \alpha(s) + C_{\bJ}\times e^{s/2} \sup_{1\leq i,p\leq d}\ \norm{ \la e_i| \partial\bX_{s,p} }_{0,c,R} \ ds. $$
    And the proof follows as in the previous case by using Equation \eqref{slvdknsldkvms}.
\end{proof}

\section{Construction of transport maps}\label{sec.transport}

\subsection{General construction of a transport map}

This subsection aims to provide a short, self-contained explanation of the method we will use to construct transport maps in the classical case. We begin with the following definitions.

\begin{definition}
    Given $V:\R^n\to\R$ we define
    $$ d\mu_V(x) := \frac{e^{-V(x)}}{Z_V}\ dx,$$
    where $Z_V$ is a normalizing constant. Besides, if $V_0:\R^n\to\R$ and $V_1:\R^n\to\R$, we say that $T$ is a tranport map between $\mu_{V_0}$ and $\mu_{V_1}$ if $T_*\mu_{V_0} = \mu_{V_1} $, that is if for all measurable function $f$,
    $$ \int_{\R^n} f( T (x)) \frac{e^{-V_0(x)}}{Z_{V_0}}\ dx = \int_{\R^n} f (x) \frac{e^{-V_1(x)}}{Z_{V_1}}\ dx$$
\end{definition}

Various methods of constructing transport maps have a long history in PDE and optimal transport theory. See for example \cite[Chapters 3 and 4]{villani2003topics} or the survey paper \cite{de2014monge} for an account on optimal transport theory and its link to the Monge-Ampère equation. The approach detailed in this subsection and notably the next proposition will follow the heuristic explained in \cite[Section 2]{bekerman2015transport}, although instead of taking $\mu_V$ to be the law of the eigenvalues of our random matrices, we will eventually work directly with the law of our random matrices.  This is the same approach taken to construct free transport maps \cite[\S 8.4]{JLS2022}, and the motivation in terms of the Wasserstein manifold is explained in \cite[\S 2.2]{JLS2022}.

\begin{proposition}
\label{pot1}
        Given $V_0:\R^n\to\R$ and $V_1:\R^n\to\R$, we define $V_t := tV_1+(1-t)V_0$. We also assume that $V_0,V_1\in\cC^1(\R^n)$ and that there exists a function $\psi:\R^n\times [0,1]\to\R$ such that 
        \begin{equation}
            \label{eqduifpsi}
            \Delta \psi_t(x) - \langle\nabla V_t(x),\nabla \psi_t(x)\rangle_{\R^n} = \dot{V_t}(x) - \mu_{V_t}\big[\dot{V_t}\big],
        \end{equation}
        where $\Delta$ is the Laplacian, $\nabla$ the gradient, $\langle\cdot,\cdot\rangle_{\R^n}$ the usual scalar product in $\R^n$, $\dot{V_t} = V_1-V_0$, and $\mu_{V_t}\big[f\big]$ is the expectation of $f$ with respect to $\mu_{V_t}$. If we set $T_t^N$ the solution to the following differential equation
        \begin{equation}
            \label{floweq}
             \dot{T_t}(x) = \nabla \psi_t\left( T_t(x)\right),\quad T_0 = \id_{\R^n},
        \end{equation}
        then $T_t$ is a transport map between $\mu_{V_0}$ and $\mu_{V_t}$.
\end{proposition}

\begin{proof} 
Since $T_t$ is the flow of an ordinary differential equation, it is a diffeomorphism of $\R^n$ for any $t$. Besides since $T_0 = \id_{\R^n}$, by continuity for all $t\in [0,1]$ and $x\in\R^n$, $\det( dT_s(x))>0$. Hence by a change of variable, we get that
\begin{align*}
    &\int_{\R^n} f( T_t (x)) \frac{e^{-V_0(x)}}{Z_{V_0}}\ dx - \int_{\R^n} f(x) \frac{e^{-V_t(x)}}{Z_{V_t}}\ dx \\
    &= \int_{\R^n} f( T_t (x)) \left(\frac{e^{-V_0(x)}}{Z_{V_0}}\ dx - \frac{e^{-V_t(T_t(x))}}{Z_{V_t}} \det( dT_t(x))\right)\ dx \\
    &= - \int_{\R^n} f( T_t (x)) \int_0^t \frac{d}{ds}\left(\frac{e^{-V_s(T_s(x))}}{Z_{V_s}} \det( dT_s(x))\right) ds\ dx.
\end{align*}
Besides,
$$ \frac{d}{ds}\left(\frac{1}{Z_{V_s}} \right) = \frac{1}{Z_{V_s}} \int_{\R^n} \dot{V_s}(x) \frac{e^{-V_s(x)}}{Z_{V_s}}\ dx = \frac{1}{Z_{V_s}} \times \mu_{V_t}\big[\dot{V_t}\big], $$
$$ \frac{d}{ds}e^{-V_s(T_s(x))} = - e^{-V_s(T_s(x))}\times \left( \dot{V_s}(T_s(x)) + \langle\nabla V_t(T_s(x)), \dot{T_s}(x) \rangle_{\R^n} \right), $$
$$ \frac{d}{ds}\det( dT_s(x)) = \det( dT_s(x)) \times \Tr_n\left( \left( dT_s(x)\right)^{-1} \dot{dT_s}(x) \right). $$
Besides, thanks to Equation \eqref{floweq}, with $\hess(\psi_t)$ the Hessian of $\psi_t$, we have that
$$ \dot{dT_s}(x) = \left[\hess(\psi_s)\left( T_s(x)\right)\right] dT_s(x),$$
Thus by using Equation \eqref{floweq} again as well as the fact that the trace of the Hessian is the Laplacian,
\begin{align*}
    &\int_{\R^n} f( T_t (x)) \frac{e^{-V_0(x)}}{Z_{V_0}}\ dx - \int_{\R^n} f(x) \frac{e^{-V_t(x)}}{Z_{V_t}}\ dx \\
    &= - \int_{\R^n} f( T_t (x)) \int_0^t \Big( \mu_{V_t}\big[\dot{V_t}\big] - \dot{V_s}(T_s(x)) - \langle\nabla V_t(T_s(x)), \nabla \psi_s\left( T_s(x)\right)\rangle_{\R^n} \\
    &\quad\quad\quad\quad\quad\quad\quad\quad\quad\quad\quad\quad\quad
    + \Delta\psi_s(T_s(x))\Big)\frac{e^{-V_s(T_s(x))}}{Z_{V_s}} \det( dT_s(x)) ds\ dx\\
    &=0,
\end{align*}
where we used Equation \eqref{eqduifpsi} in the last line.
\end{proof}

Thus, we need to find a solution to Equation \eqref{eqduifpsi}. To do so one can use stochastic calculus.

\begin{proposition}
\label{pot2}
    We assume that $V_0,V_1\in\cC^4(\R^n)$ and that its first two differentials grow at most polynomially at infinity. We also assume that there exists a constant $c>0$ such that for all $t\in [0,1]$, $x\in\R^n$, $\hess(V_t)(x)\geq cI_n$. Then there exists a solution $(\cX_s^{x,t})_{s\geq 0}$ to the stochastic differential equation
    \begin{equation}
        \label{kldsjvns}
        \cX_s^{x,t} = x - \frac{1}{2} \int_0^s \nabla V_t\left(\cX_u^{x,t}\right)\ du + B_s,
    \end{equation}
    where $B$ is a vector of $n$ independent Brownian motions. Moreover, if we define for all $x\in\R^n$, 
    $$\Xi_s(x) := \E\left[ \dot{V_t}\left(\cX_s^{x,t}\right) - \mu_{V_t}\big[\dot{V_t}\big]\right],$$
    then $\Xi_s\in\cC^2(R^n)$ for all $s\geq 0$, and for all $x\in\R^n$, the functions $s\mapsto\Xi_s(x)$, $s\mapsto\nabla\Xi_s(x)$ and $s\mapsto \hess(\Xi_s)(x)$ are integrable on $\R_+$. Furthermore,
    $$ \psi_t: x\mapsto -\frac{1}{2}\int_0^{\infty} \E\left[ \dot{V_t}\left(\cX_s^{x,t}\right) - \mu_{V_t}\big[\dot{V_t}\big] \right] ds$$
    is a solution to Equation \eqref{eqduifpsi}, and $\nabla\psi_t = -\frac{1}{2}\int_0^{\infty} \nabla\Xi_s ds $.
\end{proposition}

\begin{proof}
    Thanks to our assumption on the Hessian, the Equation \eqref{kldsjvns} has a solution thanks to \cite[Lemma 4.4.20]{AGZ2009}. Besides, with $\norm{x}_2^2 = \langle x,x \rangle_{\R^n}$, thanks to It\^o formula, we have that
    \begin{align*}
        &\norm{\cX_s^{x,t}-\cX_s^{y,t}}_2^2 \\
        &= \norm{ x-y}_2^2 - \int_0^s \langle \cX_u^{x,t}-\cX_u^{y,t}, \nabla V_t(\cX_u^{x,t})- \nabla V_t(\cX_u^{y,t}) \rangle_{\R^n} du \\
        &= \norm{ x-y}_2^2 - \int_0^s \int_0^1 \langle \cX_u^{x,t}-\cX_u^{y,t}, \left[\hess(V_t)(h\cX_u^{x,t} +(1-h)\cX_u^{y,t})\right] (\cX_u^{x,t}- \cX_u^{y,t}) \rangle_{\R^n} dh du \\
        &\leq \norm{ x-y}_2^2 -c \int_0^s \norm{\cX_u^{x,t}-\cX_u^{y,t}}_2^2 du \\
    \end{align*}
    Consequently,
    \begin{equation}
        \label{normboundscs}
        \norm{\cX_s^{x,t}-\cX_s^{y,t}}_2^2 \leq \norm{ x-y}_2^2 e^{-cs}.
    \end{equation}
    Note that, since $s\mapsto \cX_s^{x,t}$ is continuous, the supremum of $\cX_s^{x,t}$ for $t$ and $x$ in bounded sets is finite. Let us now show that $\Xi_s$ can be differentiated. Since the expectation of any polynomials in $\cX_s^{x,t}$ is finite (see \cite[Lemma 4.4.20]{AGZ2009}), and that we assumed that $V_0,V_1$ and its first two differential grows at most polynomially, to do so it is sufficient to show that $x\mapsto \cX_s^{x,t}$ is differentiable. However since we assumed that $V_t\in\cC^4(\R^n)$,
    \begin{align*}
        \cX_s^{x,t}-\cX_s^{y,t} &= x-y - \frac{1}{2} \int_0^s \nabla V_t(\cX_u^{x,t})- \nabla V_t(\cX_u^{y,t})\ du \\
        &= x-y - \frac{1}{2} \int_0^s \left[\hess(V_t)(\cX_u^{x,t})\right](\cX_u^{x,t})-\cX_u^{y,t}) du + \mathcal{O}\left( \norm{\cX_u^{x,t})-\cX_u^{y,t}}_2^2\right)  \\
        &= x-y - \frac{1}{2} \int_0^s \left[\hess(V_t)(\cX_u^{x,t})\right](\cX_u^{x,t})-\cX_u^{y,t}) du + \mathcal{O}\left( \norm{x-y}_2^2\right).
    \end{align*}
    Thus, by a Picard iteration argument, we have that
    \begin{align*}
        \cX_s^{x,t}-\cX_s^{y,t} &= \sum_{l\geq 0} \frac{1}{(-2)^n} \int_{0\leq s_1\leq \dots\leq s_n\leq s} \left[\hess(V_t)(\cX_{s_n}^{x,t})\right]\dots \left[\hess(V_t)(\cX_{s_1}^{x,t})\right] (x-y) + \mathcal{O}\left( \norm{x-y}_2^2\right).
    \end{align*}
    Hence $x\mapsto \cX_s^{x,t}$ is differentiable and
    $$ d\cX_s^{x,t} = \sum_{l\geq 0} \frac{1}{(-2)^n} \int_{0\leq s_1\leq \dots\leq s_n\leq s} \left[\hess(V_t)(\cX_{s_n}^{x,t})\right]\dots \left[\hess(V_t)(\cX_{s_1}^{x,t})\right]. $$
    In particular,
    \begin{align*}
        d\cX_s^{x,t} &= \id_{\R^n} - \frac{1}{2} \int_0^s \left[\hess(V_t)(\cX_u^{x,t})\right]d\cX_u^{x,t} du.
    \end{align*}
    Hence, thanks to our assumption on the Hessian, 
    $$\frac{d}{ds} (d\cX_s^{x,t})^* d\cX_s^{x,t} = - (d\cX_s^{x,t})^* \left[\hess(V_t)(\cX_s^{x,t})\right] d\cX_s^{x,t} \leq -c (d\cX_s^{x,t})^* d\cX_s^{x,t}. $$
    Hence, with $\norm{\cdot}$ the operator norm, $\norm{d\cX_s^{x,t}} \leq e^{-cs/2}$. Hence, not only is $\Xi_s$ differentiable, but we also have that $s\mapsto\nabla\Xi_s(x)$ is integrable. Similarly, we get that $\Xi_s\in\cC^2(\R^n)$ and that $s\mapsto \hess(\Xi_s)(x)$ are integrable on $\R_+$. 
    
    Finally, for any functions $f\in\cC^2(\R^n)$, thanks to It\^o formula,
    \begin{equation}
        \label{sdvjnsok}
    \end{equation}
    $$ \frac{d}{dt}\E\left[ f\left(\cX_s^{x,t}\right) \right] = \frac{1}{2} \E\left[ \Delta f\left(\cX_s^{x,t}\right) - \langle\nabla V_t\left(\cX_s^{x,t}\right),\nabla f\left(\cX_s^{x,t}\right)\rangle_{\R^n} \right]. $$
    And since by \cite[Lemma 4.4.20]{AGZ2009}, $(\cX^{t,x})_{s\geq 0}$ is a Markov process, with $P_tf = \E[f(\cX_s^{x,t})]$, we have that
    $$ \frac{d}{dt}\E\left[ f\left(\cX_s^{x,t}\right) \right] = \frac{1}{2} \left( \Delta P_tf - \langle\nabla V_t\left(x\right),\nabla P_tf\rangle_{\R^n}\right). $$
    Consequently,
    $$ \frac{d}{dt} \mu_{V_t}\left( P_s(f)\right) = \frac{1}{2} \mu_{V_t}\left( \Delta P_sf - \langle\nabla V_t\left(x\right),\nabla P_sf\rangle_{\R^n}\right) = 0,$$
    where we used integration by part in the last equality. Thus if the law of $\mathbf{x}$ is $\mu_{V_t}$, then $\E[f(\cX_s^{t,\mathbf{x}})] = \E[f(\cX_0^{t,\mathbf{x}})] = \E[f(\mathbf{x})]= \mu_{V_t}(f)$. Hence thanks to Equation \eqref{normboundscs} we have that
    \begin{align*}
        |\Xi_s(x)| = |\E[\dot{V_t}(\cX_s^{t,x})-\dot{V_t}(\cX_s^{t,\mathbf{x}})]| = \left| \int_0^1 \left\langle\nabla V_t\left(\cX_s^{t,ux+(1-u)\mathbf{x}}\right),\cX_s^{t,x}-\cX_s^{t,\mathbf{x}}\right\rangle_{\R^n} du \right|.
    \end{align*}
    Thus thanks to Equation \eqref{normboundscs}, Cauchy-Schwarz, and the fact that $\nabla V_t$ grows at most polynomially at infinity, we have that $\int_0^{\infty}|\Xi_s(x)| ds$ can be upper bounded by the expectation of moments of $\cX_s^{t,ux+(1-u)\mathbf{x}}$. Besides, thanks to \cite[Lemma 4.4.22]{AGZ2009}, if $f$ is a polynomial, then so is $P_tf$, and since $\hess(V_t)(x)\geq cI_n$, the expectation of every polynomial in $\mathbf{x}$ is finite.
     
    Consequently, the only thing that remains to prove is that $\psi_t$ is a solution of Equation \eqref{eqduifpsi}. Since we have shown that the differentials of $x\mapsto\Xi_s(x)$ are integrable, by dominated convergence we have that
    \begin{align*}
        \Delta \psi_t(x) - \langle\nabla V_t(x),\nabla \psi_t(x)\rangle_{\R^n} &= -\frac{1}{2} \int_0^{\infty} 2\frac{d}{ds}\E\left[ \dot{V_t}\left(\cX_s^{x,t}\right) - \mu_{V_t}\big[\dot{V_t}\big] \right] ds \\
        &= \dot{V_t}(x) - \mu_{V_t},
    \end{align*}
    where we used Equation \eqref{sdvjnsok} and the fact that $X_0^{x,t}=x$ and $\lim_{s\to\infty} \E\left[\dot{V_t}\left(\cX_s^{x,t}\right)\right] = \mu_{V_t}\big[\dot{V_t}\big]$ for all $x$ since $s\mapsto\Xi_s(x)$ must converge to $0$ at infinity.  
\end{proof}

\subsection{The case of random matrices}

By viewing $(\M_N(\C)_{sa})^d$ as isomorphic to $\R^{dN^2}$, the foregoing construction of transport yields the following proposition.

\begin{proposition}
\label{prop:transportmap}
    Let $V_0, V_1 \in \cC^4(X_1,\dots,X_d)$ be self-adjoint, i.e., such that for all $(\cA,\tau) \in\W$, $x_1,\dots,x_d\in\cA$ self-adjoint, $V_0(x)$ and $V_1(x)$ are self-adjoint. Let $(\bX_t^{s,H})_{t\geq 0}$ be the solution of
    \begin{equation}
    \label{eqmatdiff}
        \bX_{t,p}^s(H) = H - \frac{1}{2} \int_0^t \bD_p V_s\left( \bX_u^s(H) \right) du + S^N_p(t),\quad 1\leq p\leq d,
    \end{equation}
    where $V_s\coloneqq s V_0 + (1-s) V_1$, $\bD_p \coloneqq \cD_p + \id\otimes\tr \circ \widetilde{\partial}_p$, and $S^N_p(t) \coloneqq S_{t,p}(X^N)$ where $X^N$ is a family of independent GUE random matrices and $S_{t,p}$ is as in Equation \eqref{skgjnslvms}. We set 
    \begin{equation}
        \label{masterop}
        \Psi_s:H\in\M_N(\C)_{sa}^d \mapsto \left( -\frac{1}{2} \int_0^{\infty}\E\left[ \bD_p\left( \dot{V}_s\circ \bX_t^s\right)(H)  \right]  dt \right)_{1\leq p\leq d},
    \end{equation}
    where $\dot{V}_s = V_1-V_0$.
    Then, if $(T_s^N)_{0\leq s\leq 1}$ is solution to the linear differential equation
    $$ \dot{T}_s^N = \Psi_s \circ T_s^N,\quad T_0^N=\id_{\M_N(\C)_{sa}^d},$$
    then it is a transport map between $V_0$ and $V_1$. Concretely, for all measurable function $f$,
    $$ \int_{\M_N(\C)_{sa}^d} f(H) \frac{e^{-N\Tr_N(V_s(H))}}{Z_{V_s}}\ dH = \int_{\M_N(\C)_{sa}^d} f\circ T_s^N(H) \frac{e^{-N\Tr_N(V_0(H))}}{Z_{V_0}}\ dH. $$
\end{proposition}

\begin{proof}
    
    For scaling reasons, we want to use Propositions \ref{pot1} and \ref{pot2} with the potential 
    $$W_s(x) = N\Tr_N(V_s(X(x)))$$
    where $x$ is the collection of $dN^2$ variables $(x_{i,j}^k)_{1\leq i\leq j\leq N, 1\leq k\leq d}$ and $(y_{i,j}^k)_{1\leq i< j\leq N, 1\leq k\leq d}$, and $X(x)=(X_1(x),\dots,X_d(x))$ is the $d$-tuple of matrices such that 
    $$(X_k(x))_{i,i} = \frac{x_{i,i}^k}{\sqrt{N}}, \quad (X_k(x))_{i,j} = \frac{x_{i,j}^k+\ii y_{i,j}}{\sqrt{2N}} \text{ if }i<j,\quad  \text{and } (X_k(x))_{i,j} = \frac{x_{i,j}^k-\ii y_{i,j}}{\sqrt{2N}} \text{ else}.$$
    We also define $X^{-1}(H)$ the collection of $dN^2$ variables such that if $H$ is a $d$-tuple of self-adjoint matrices then $X(X^{-1}(H)) = H$. Note that thanks to Proposition \ref{skjodncskncv}, we do have that $V_0,V_1\in \cC^4((\M_N(\C)_{\sa})^d)$, thus one can use Proposition \ref{pot2}. Besides, thanks to Proposition \ref{skjodncskncv}, we have that for $i<j$,
    $$ \partial_{x_{i,i}^p} W_s(x)= \sqrt{N}(\bD_p V_s(X(x)))_{i,i}, $$
    $$ \partial_{x_{i,j}^p} W_s(x)= \sqrt{\frac{N}{2}} \left( (\bD_p V_s(X(x)))_{j,i} + (\bD_p V_s(X(x)))_{i,j} \right) = \sqrt{2N} \cRe (\bD_p V_s(X(x)))_{i,j}, $$
    $$ \partial_{y_{i,j}^p} W_s(x)= \ii \sqrt{\frac{N}{2}} \left( (\bD_p V_s(X(x)))_{j,i} - (\bD_p V_s(X(x)))_{i,j} \right) =\sqrt{2N} \cIm (\bD_p V_s(X(x)))_{i,j}, $$
    
    \noindent Consequently, if $(X_t^{s,H})_{t\geq 0}$ is the solution of Equation \eqref{eqmatdiff}, then by taking $\cX_t^{x,s} = X^{-1}\left( X_t^s(X(x)) \right)$ we obtain a solution of Equation \eqref{kldsjvns}. Thus if we consider the map $T_s$ of Proposition \ref{pot2}, we know that it is a transport map if one views our random matrix law as law on $\R^{dN^2}$. Thus for any measurable functions,
    $$ \int_{\R^{dN^2}} f( X(T_s (x))) \frac{e^{-N\Tr_N(V_0(X(x)))}}{Z_{V_0}}\ dx = \int_{\R^{dN^2}} f(X(x)) \frac{e^{-N\Tr_N(V_s(X(x)))}}{Z_{V_s}}\ dx. $$
    Thus by a change of variable, since $X(X^{-1}(H)) = H$,
    $$ \int_{\M_N(\C)_{sa}^d} f( X(T_s (X^{-1}(H)))) \frac{e^{-N\Tr_N(V_0(H))}}{Z_{V_0}}\ dH = \int_{\M_N(\C)_{sa}^d} f(H) \frac{e^{-N\Tr_N(V_s(H))}}{Z_{V_s}}\ dH. $$

    \noindent Thus the transport map that we are looking for is simply $T_s^N: H\mapsto X(T_s(X^{-1}(H)))$ which satisfies the differential equation
    $$ \dot{T}_s^N = X\circ\nabla\psi_s\circ X^{-1}\circ T_s^N,\quad f_0^N = \id_{\M_N(\C)_{sa}^d}.$$
    Besides, we have that $\dot{W_s}\left( \cX_t^{x,s} \right) = N\Tr_N\left( \dot{V_s}\left( \bX_t^{s,X(x)}\right) \right)$, thus, thanks to Proposition \ref{skjodncskncv}, as well as Theorem \ref{chainrule} and \ref{thm.keySDE} which guarantees that $\dot{V}_s\circ \bX_t^s\in\cC^3(\cB)$,  for $i<j$,
    $$ \partial_{x_{i,i}^p} \dot{W_t}\left( \cX_s^{x,t} \right) = \sqrt{N}\left( \bD_p\left( \dot{V}_s\circ \bX_t^s\right)(X(x)) \right)_{i,i}, $$
    $$ \partial_{x_{i,j}^p} \dot{W_t}\left( \cX_s^{x,t} \right) = \sqrt{2N} \cRe \left( \bD_p\left( \dot{V}_s\circ \bX_t^s\right)(X(x)) \right)_{i,j}, $$
    $$ \partial_{y_{i,j}^p} \dot{W_t}\left( \cX_s^{x,t} \right) =\sqrt{2N} \cIm \left( \bD_p\left( \dot{V}_s\circ \bX_t^s\right)(X(x)) \right)_{i,j}. $$
    Consequently, 
    $$X\circ\nabla \left( x\mapsto \dot{W_t}\left( \cX_s^{x,t} \right)\right) \circ X^{-1}(H) = \left( \bD_p\left( \dot{V}_s\circ \bX_t^s\right)(H) \right)_{1\leq p\leq d}.$$
    This implies that
    $$ X\circ\nabla\psi_t\circ X^{-1}(H) = \left( -\frac{1}{2} \int_0^{\infty} \E\left[ \bD_p\left( \dot{V}_s\circ \bX_t^s\right)(H)  \right] ds \right)_{1\leq p\leq d}. $$
    Hence the conclusion.
\end{proof}

The following theorem is a self-contained version of Theorem \ref{mainthm} and in particular, immediately implies it.

\begin{theorem} \label{thm:main}
    Let $T_s^N$ be the transport map defined in Proposition \ref{prop:transportmap}, with $V_0=\frac{1}{2}\sum_{i=1}^d X_i^2$, and $V_1 = \frac{1}{2}\sum_{i=1}^d X_i^2 + W $ where $W\in\cC^{4k+7}(X_1,\dots,X_d)$. We assume that
    $$\kappa \coloneqq \sup_{1\leq p\leq d}\ \left\{ \sum_{1\leq i\leq d} \norm{\partial_i\bD_pW}_{0,\infty} + \norm{\widetilde{\partial}_i \bD_p W }_{0,\infty} \right\} < \frac{1}{4k+5},$$
    $$ \sup_{1\leq j\leq 4k+7} \norm{W}_{j,\infty} <\infty. $$    
    
    \noindent Then there exists $T_s^0,\in (\cC^k(X_1,\dots,X_d))^d$, as well as $T_s^m\in (\cC^{k+1-m}(X_1,\dots,X_d))^d$ for $m\in [1,k]$, continuous with repect to $s$, such that for all $s\in [0,1]$, $N\geq 1$,
    \begin{equation}
        \label{eq:main}
        \sup_{H\in\M_N(\C)_{sa}^d}\ \norm{T_s^N(H) - \sum_{i=0}^k \frac{T_s^i(H)}{N^{2i}} } \leq \frac{C_{W,k}}{N^{2k+2}}.
    \end{equation}
    Besides, with  $\norm{\cf}_{\cC^k,\infty} = \sum_{p=1}^d \norm{\cf_i}_{\cC^k,\infty} $ if $\cf\in (\cC^k(X_1,\dots,X_d))^d$, there exists a constant $C_{W,k}$ such that
    \begin{align}
        \label{eqkvjnso}
        \norm{T_s^0 - \id }_{\cC^k,\infty}\leq C_{W,k}, \quad\quad \norm{T_s^m}_{\cC^{k+1-m},\infty}\leq C_{W,k},\quad 1\leq m\leq k.
    \end{align}
\end{theorem}

\begin{proof}

In order to shorten notations, if $\cf\in (\cC^k(X_1,\dots,X_d))^d$, then we will simply denote $\cf\in \cC^k(X_1,\dots,X_d)$.
To begin with, with the notations of Proposition \ref{prop:transportmap}, we have $\dot{V}_s=W$ for all $s\in [0,1]$, and $V_s = \frac{1}{2}\sum_{i=1}^d X_i^2 + s W$. Therefore, thanks to Theorem \ref{thm.keySDE}, we can find a solution $\bX_{t,p}^s\in\cC^{4k+5}(\cB)$ to
\begin{equation}
    \bX_{t,p}^s(x,X) = x_p - \frac{1}{2} \int_0^t (\bD_p V_s\circ \bX_u^s)(x,X) du + S_{t,p}(X),\quad 1\leq k\leq d,
\end{equation}
and since
$$ \bX_{t,p}^{s_1} - \bX_{t,p}^{s_2} = - \frac{s_1-s_2}{2} \int_0^t \bD_p W\circ \bX_u^{s_1} - \bD_p W\circ \bX_u^{s_2} du, $$
by combining Theorem \ref{thm:fastdecay} and Corollary \ref{cor:complip}, thanks to our assumptions there exists a constant $C_{W,k}$ such that 
\begin{equation}
    \label{lipsovdns}
    \norm{\bX_{t,p}^{s_1} - \bX_{t,p}^{s_2}}_{\cC^{4k+5},\infty} \leq C_{W,k} |s_1-s_2|
\end{equation}
In particular, the process $s\in\R\mapsto X_{t,p}^{s} \in\cC^{4k+5}(\cB)$ is continuous. 

Besides, thanks to Theorem \ref{chainrule} and our assumptions on $W$, if $j \coloneqq j_0+\dots+j_n\leq 4k+4$, one can find a constant $C_{W,k}$ such that,
\begin{align*}
    &\norm{ \widetilde{\partial}^{(j_n)}\circ\cdots\circ\widetilde{\partial}^{(j_1)}\circ\partial^{(j_0)} \left(\bD_p\left( \dot{V}_s\circ \bX_t^s\right) \right)}_{0,c,R} \leq C_{W,k} \times \widetilde{\Lambda}^j_R(\bX_t^s),
\end{align*}
where
$$ \widetilde{\Lambda}^j_R(\cf) = \sup_{\substack{1\leq i_1,\dots,i_m\leq d \\ s_1+\dots+s_m = j+1 \\ s_l\geq 1,\ m\geq 1 }}\ \left( \sup_{a_0+\dots+a_q=s_1}\ \norm{\la e_i| \widetilde{\partial}^{(a_q)}\circ\cdots\circ\widetilde{\partial}^{(a_1)}\circ\partial^{(j_0)} \cf_{i_l}}_{0,c,R} \right) \prod_{l=2}^m\ \norm{\cf_{i_l}}_{s_l,c,R}, $$

\noindent Thus, thanks to Theorem \ref{thm:fastdecay} and the assumptions we made on the potential $W$, there exists a constant $C_{W,k}$ such that for all $s\in[0,1]$,
\begin{align*}
    &\norm{ \widetilde{\partial}^{(j_n)}\circ\cdots\circ\widetilde{\partial}^{(j_1)}\circ\partial^{(j_0)} \left(\bD_p\left( \dot{V}_s\circ \bX_t^s\right) \right)}_{0,c,R} \leq C_{W,k} \times (1+t)^{8k+10} e^{-\frac{1-\kappa(4k+5)}{2}t}.
\end{align*}

\noindent Thus, one can find a constant $C_{W,k}$ such that
\begin{align*}
    &\norm{ \bD_p\left( \dot{V}_s\circ X_t^s\right)}_{\cC^{4k+4},R} \leq C_{W,k} \times (1+t)^{8k+10} e^{-\frac{1-\kappa(4k+5)}{2}t}.
\end{align*}
This implies that the function
$$ F_s := -\frac{1}{2} \int_0^{\infty} \bD_p\left( \dot{V}_s\circ \bX_t^s\right) dt$$
also belongs to $\cC^{4k+4}(\cB)$ and $\norm{F_s}_{\cC^{4k+4},\infty} \leq C_{W,k}$ for some constant $C_{W,k}$. Besides, $ s\in[0,1]\mapsto F_s\in\cC^{4k+4}(\cB)$ is continuous since, thanks to Equation \eqref{lipsovdns},
\begin{align*}
    \norm{F_{s_1}-F_{s_2}}_{\cC^{4k+4},\infty} &\leq \frac{1}{2} \int_0^{(s_1-s_2)^{-1/2}} \norm{\bD_p\left( \dot{V}_s\circ \bX_t^s\right) - \bD_p\left( \dot{V}_s\circ \bX_t^s\right)}_{\cC^{4k+4},\infty} dt \\
    &\quad + \int_{(s_1-s_2)^{-1/2}}^{\infty} (1+t)^{8k+10} e^{-\frac{1-\kappa(4k+5)}{2}t} dt.
\end{align*}
Then, by combining Corollary \ref{cor:complip} with Equation \eqref{lipsovdns}, we get that for some constant $C_{W,k}$,
\begin{equation}
    \norm{F_{s_1}-F_{s_2}}_{\cC^{4k+4},\infty} \leq C_{W,k} (s_1-s_2)^{1/2}.
\end{equation}
Besides, we can use Theorem \ref{3lessopti} to deduce that 
$$ \E\left[ F_s(H,X^N) \right] = \sum_{i=0}^k \frac{\tau_N\left( g_{s,i}(H,x^i) \middle| \M_N(\C) \right)}{N^{2i}} + \frac{1}{N^{2k+2}} \E\left[ \tau_N\left( g_{s,k+1}(H,X^N,x^{k+1}) \middle| \M_N(\C) \right) \right], $$
where $x^i$ are families of free semicircular variables and $X^N$ a family of independent GUE random matrices. Furthermore, thanks to Equation \eqref{eq:normupperexp},
\begin{equation}
    \label{ldifnvsdn}
    \norm{g_{s,i}}_{\cC^{4(k-i+1)},\infty} \leq C_i \times \norm{F_s}_{\cC^{4k+4},\infty} \leq C_iC_{W,k}.
\end{equation}
By applying Theorem \ref{3lessopti} to $F_{s_1}-F_{s_2}$, and making use of Equation \eqref{eq:normupperexp} one more time, we also get that $s\in[0,1]\mapsto g_{s,i}\in\cC^{4(k-i+1)}(\cB)$ is continuous. Hence, thanks to Proposition \ref{prop:condexp}, we can find functions $\phi^i_s\in\cC^{k-i+1}(X_1,\dots,X_d)$, still continuous with respect to $s$, such that for some constant $C_{W,k}$,
\begin{equation}
    \label{ldifnvsdn2}
    \norm{\phi_s^i}_{\cC^{k-i+1},\infty} \leq C_{W,k}.
\end{equation}
and
\begin{equation}
\label{sjklvdnslvm}
    \sup_{N\geq 1}\ \sup_{H\in\M_N(\C)_{sa}^d}\ \norm{\Psi_s(H) -  \sum_{i=0}^k \frac{\phi^i_s(H)}{N^{2i}} } \leq  \frac{C_{W,k}}{N^{2k+2}}.
\end{equation}
Note that since $\phi^i_s\in\cC^{k-i+1}(X_1,\dots,X_{d})$, thanks to Corollary \ref{cor:taylorineq}, one can find 
$$\phi^{i,j}_s\in\cC^{k-i-j+1}(X_1,\dots,X_{d(j+1)}), \quad 0\leq j\leq k-i,$$ 
linear in the last $d\times j$ variables, still continuous with respect to $s$, such that for any $N\geq 1$,
\begin{equation}
\label{kdjvnsln}
    \sup_{A\in\M_N(\C)_{sa}^d}\ \norm{\phi^i_s(A+H) - \sum_{j=0}^{k-i} \phi^{i,j}_{s}(A,H,\dots,H)} \leq C_{W,k} \norm{H}^{j+1}.
\end{equation}

\noindent Thus we define $(T_s^m)_{0\leq s\leq 1}$ as the solution of 
\begin{equation}
    \label{firstorder}
    T_s^0 = \id + \int_0^s \phi_r^0\circ T_r^0\ dr, 
\end{equation}
\begin{equation}
    \label{higherorder}
    \forall m\geq 1,\quad T_s^m = \int_0^s \phi^{m,0}_r\circ T_r^0 + \sum_{n=0}^{m-1} \sum_{j=1}^{m-n} \sum_{\substack{a_1,\dots,a_j\geq 1 \\ a_1+\dots+a_j = m-n}} \phi_r^{n,j}\circ (T^0_r,T^{a_1}_r,\dots,T^{a_j}_r ) dr.
\end{equation}
Note that, thanks to Equation \eqref{ldifnvsdn2}, we can immediately use Theorem \ref{thm.keySDE} to deduce that Equation \eqref{firstorder} has a solution $(T^0_s)_{0\leq s\leq 1}$ continuous in $\cC^k(X_1,\dots,X_d)$. Moreover, Gr\"onwall's inequality coupled with Corollary \ref{cor:complip} immediately yields by induction that for some constant $C_{W,k}$,
\begin{equation}
\label{dlfbnlsmvspvm}
    \forall s\in [0,1],\quad \norm{T^0_s-\id}_{\cC^k,\infty} \leq C_{W,k}.
\end{equation}
Let us now prove by induction that for $m\geq 1$, Equation \eqref{higherorder} has a solution $(T_s^m)_{0\leq s\leq 1}$ continuous in $\cC^{k+1-m}(X_1,\dots,X_d)$. First, in Equation \eqref{higherorder}, we separate out the term $\varphi_r^{0,1}(T_r^0,T_r^m)$ with $n = 0$, $j = 1$ and $j_1 = m$, and let $G_s$ denote the sum of the remaining terms, so that
\[
T_s^m = \int_0^s \phi^{0,m}_r\circ(T_r^0, T_r^m) dr + G_s.
\]
Then by Corollary \ref{cor:multicompfin}, $G_s$ belongs to $\cC^{k+1-m}(X_1,\dots,X_d)$. Moreover, if
$$ \cK\coloneqq \left\{ f\in\cC^{k+1-m}(X_1,\dots,X_d)\ \middle|\ \norm{f}_{\cC^{k+1-m},\infty}<\infty \right\}, $$
then $\cK$ is a Banach space for the norm $\norm{\cdot}_{\cC^{k+1-m},\infty}$, and $G\in\cC^0([0,1],\cK)$. Since by iterating sufficiently many times the linear map $L:\cK\to\cK$,
$$ L: h\mapsto \left( \int_0^s \phi^{0,m}_r\circ(T_r^0, h_r)\ dr + G_s \right)_{0\leq s\leq 1} $$
we get a contraction, the Banach fixed point theorem yields a unique solution to Equation \eqref{higherorder}. Besides, once again by Gr\"onwall inequality coupled with Theorem \ref{chainrule}, we get by induction that for some constant $C_{W,k}$,
\begin{equation}
    \label{dovjnsdn}
    \sup_{0\leq s\leq 1}\ \norm{T_s^m}_{\cC^{k+1-m},\infty} \leq C_{W,k}.
\end{equation}

Therefore, with the notations of Proposition \ref{prop:transportmap}, we set
$$ h_s \coloneqq T_s^N - \sum_{m=0}^k \frac{T_s^m}{N^{2m}}, $$
and we have that
\begin{align*}
    h_s &= \int_0^s \Psi_r\circ f_r^N - \sum_{0\leq n\leq m\leq k } \sum_{j=0}^{m-n} \sum_{\substack{a_1,\dots,a_j\geq 1 \\ a_1+\dots+a_j = m-n}} \left(\frac{\phi_r^{n,j}}{N^{2n}}\circ f^0_r\right)\sh \left( \frac{f^{a_1}_r}{N^{2a_1}},\dots, \frac{f^{a_j}_r}{N^{2a_j}} \right) dr \\
    &= \int_0^s \Psi_r\circ f_r^N - \Psi_r\circ \left(\sum_{m=0}^k \frac{f_s^m}{N^{2m}}\right) + \left(\Psi_r - \sum_{i=0}^k \frac{\phi^i_s}{N^{2i}} \right)\circ \left(\sum_{m=0}^k \frac{f_s^m}{N^{2m}}\right) \\
    &\quad + \sum_{n=0}^k \frac{1}{N^{2n}} \left(\phi^n_s\circ \left(\sum_{m=0}^k \frac{f_s^m}{N^{2m}}\right) - \sum_{m=n}^k \sum_{j=0}^{m-n} \sum_{\substack{a_1,\dots,a_j\geq 1 \\ a_1+\dots+a_j = m-n}} \left(\phi_r^{n,j}\circ f^0_r\right)\sh \left( \frac{f^{a_1}_r}{N^{2a_1}},\dots, \frac{f^{a_j}_r}{N^{2a_j}} \right)\right) dr.
\end{align*}
Note that
\begin{align*}
    &\sum_{m=n}^k \sum_{j=0}^{m-n} \sum_{\substack{a_1,\dots,a_j\geq 1 \\ a_1+\dots+a_j = m-n}} \left(\phi_r^{n,j}\circ f^0_r\right)\sh \left( \frac{f^{a_1}_r}{N^{2a_1}},\dots, \frac{f^{a_j}_r}{N^{2a_j}} \right) \\
    &= \sum_{j=0}^{k-n} \sum_{m=j+n}^k \sum_{\substack{a_1,\dots,a_j\geq 1 \\ a_1+\dots+a_j = m-n}} \left(\phi_r^{n,j}\circ f^0_r\right)\sh \left( \frac{f^{a_1}_r}{N^{2a_1}},\dots, \frac{f^{a_j}_r}{N^{2a_j}} \right) \\
    &= \sum_{j=0}^{k-n} \sum_{j \leq a_1+\dots+a_j \leq k-n} \left(\phi_r^{n,j}\circ f^0_r\right)\sh \left( \frac{f^{a_1}_r}{N^{2a_1}},\dots, \frac{f^{a_j}_r}{N^{2a_j}} \right).
\end{align*}
Therefore, using Equations \eqref{kdjvnsln} and \eqref{dlfbnlsmvspvm}, for all $N\geq 1$,
\begin{align*}
    &\sup_{H\in\M_N(\C)_{sa}^d}\ \norm{\phi^n_s\circ \left(\sum_{m=0}^k \frac{f_s^m(H)}{N^{2m}}\right) - \sum_{j=0}^{k-n} \sum_{j \leq a_1+\dots+a_j \leq k-n} \phi_r^{n,j}\left(f^0_r(H), \frac{f^{a_1}_r(H)}{N^{2a_1}},\dots, \frac{f^{a_j}_r(H)}{N^{2a_j}} \right)} \\
    &\leq \sum_{j=0}^{k-n}\  \sup_{H\in\M_N(\C)_{sa}^d}\  \Bigg\| \phi^{n,j}_{s}\circ \left(f_s^0(H),\sum_{m=1}^k \frac{f_s^m(H)}{N^{2m}},\dots,\sum_{m=1}^k \frac{f_s^m(H)}{N^{2m}}\right) \\
    &\quad\quad\quad\quad\quad\quad\quad\quad\quad - \sum_{j \leq a_1+\dots+a_j \leq k-n} \phi_r^{n,j}\left( f^0_r(H),\frac{f^{a_1}_r(H)}{N^{2a_1}},\dots, \frac{f^{a_j}_r(H)}{N^{2a_j}} \right) \Bigg\| + \frac{C_{W,k}}{N^{2k+2}}\\
    &= \sum_{j=0}^{k-n}\ \sup_{H\in\M_N(\C)_{sa}^d}\  \norm{\sum_{ \substack{a_1+\dots+a_j > k-n \\ a_1,\dots a_j \leq k}} \left(\phi_r^{n,j}\circ f^0_r(H)\right)\sh \left( \frac{f^{a_1}_r(H)}{N^{2a_1}},\dots, \frac{f^{a_j}_r(H)}{N^{2a_j}} \right)} + \frac{C_{W,k}}{N^{2k+2}} \\
    &\leq \frac{2C_{W,k}}{N^{2k+2}} \\
\end{align*}

\noindent Then using Equation \eqref{sjklvdnslvm}, we have, for all $N\geq 1$,
$$ \sup_{H\in\M_N(\C)_{sa}^d}\ \norm{\left(\Psi_r - \sum_{i=0}^k \frac{\phi^i_s}{N^{2i}} \right)\circ \left(\sum_{m=0}^k \frac{f_s^m(H)}{N^{2m}}\right)} \leq \frac{C_{W,k}}{N^{2k+2}}. $$

\noindent Therefore, for all $N\geq 1$,
$$ \sup_{H\in\M_N(\C)_{sa}^d}\ \norm{g_s(H)} \leq \frac{C_{W,k}}{N^{2k+2}} + \int_0^s  \sup_{H\in\M_N(\C)_{sa}^d}\ \norm{\Psi_r\circ f_r^N(H) - \Psi_r\circ \left(\sum_{m=0}^k \frac{f_s^m(H)}{N^{2m}}\right)}\ dr. $$

\noindent However, we have
\begin{align*}
    (\Psi_s(H)-\Psi_s(K))_p &= \frac{1}{2} \int_0^{\infty}\E\left[ \bD_p\left( W\circ \bX_t^s\right)(K,X^N) - \bD_p\left( W\circ \bX_t^s\right)(H,X^N) \right]  dt.
\end{align*}
Besides, with $\norm{\cdot}$ the operator norm in $\M_N(\C)$, for some constant $C$,
\begin{align*}
    &\norm{\bD_p\left( W\circ \bX_t^s\right)(K,X^N) - \bD_p\left( W\circ \bX_t^s\right)(H,X^N)} \\
    &\leq C \sup_i \Bigg( \norm{ \partial \bX_{t,i}^s(H,X^N)-\partial \bX_{t,i}^s(K,X^N)} + \norm{ \widetilde{\partial} \bX_{t,i}^s(H,X^N)- \widetilde{\partial} \bX_{t,i}^s(K,X^N)} \\
    &\quad\quad\quad\quad + \norm{\bX_{t,i}^s(H,X^N)-\bX_{t,i}^s(K,X^N)} \times \sup_l \left( \norm{\la e_p | \partial \bX_{t,l}^s }_{0,\infty} + \norm{\la e_p | \widetilde{\partial} \bX_{t,l}^s }_{0,\infty}\right) \Bigg) \\
    &\quad\quad \times \left( \norm{W}_{1,\infty} + \norm{W}_{2,\infty} \right).
\end{align*}
Furthermore,
$$  \bX_t^s(H,X^N)- \bX_t^s(K,X^N) = H-K -\frac{1}{2} \int_0^t \bD_p V_s(\bX_u^s(H,X^N)) du - \bD_p V_s(\bX_u^s(H,X^N)) \ du, $$
$$  \partial \bX_t^s(H,X^N) - \partial \bX_t^s(K,X^N) = -\frac{1}{2} \int_0^t \partial(\bD_p V_s \circ \bX_u^s)(H,X^N) du - \partial(\bD_p V_s \circ \bX_u^s)(K,X^N) \ du, $$
Thus by following the proof of Theorem \ref{thm:fastdecay}, we get that
$$  \sup_p\ \norm{\bX_{t,p}^s(H)- \bX_{t,p}^s(K)} \leq C_W \norm{ H-K } \times e^{-\frac{1-\kappa}{2}t}, $$
$$  \sup_p\ \norm{\partial \bX_{t,p}^s(H) - \partial \bX_{t,p}^s(K)} \leq C_W \norm{H-K} \times (1+t)^2 e^{-\frac{1-2\kappa}{2}t}. $$
Thus one can find a constant $C_W$ such that 
$$ \norm{(\Psi_s(H)-\Psi_s(K))_p} \leq C_W \norm{H-K}.$$
This implies that, for all $N\geq 1$,
$$ \sup_{H\in\M_N(\C)_{sa}^d}\ \norm{G_s(H)} \leq \frac{C_{W,k}}{N^{2k+2}} + C_W \int_0^s \sup_{H\in\M_N(\C)_{sa}^d}\ \norm{G_r(H)}\ dr. $$
And so by Gr\"onwall's inequality, we have that for some constant $C_{W,k}$, for all $N\geq 1$,
$$ \sup_{H\in\M_N(\C)_{sa}^d}\ \norm{G_s(H)} \leq \frac{C_{W,k}}{N^{2k+2}}. $$
Hence the conclusion.
\end{proof}

\section{Proofs of corollaries} \label{sec.otherproofs}

\subsection{Proof of the asymptotic expansion of expectations (Theorem \ref{cor1})} 

We begin by applying Theorem \ref{thm:main}, which states that 
$$ \E\left[\frac{1}{N}\Tr\left( f(Y^N) \right) \right] = \E\left[\frac{1}{N}\Tr\left( f\circ T^N(X^N) \right) \right], $$
where $X^N$ is a $d$-tuple of GUE random matrices, and $T^N\coloneqq T_1^N$ the transport map defined in Theorem \ref{thm:main}. Then, thanks to Equation \eqref{eq:main} and Corollary \ref{cor:taylorineq}, one can find constants $C_{W,k}$ and $R_{W,k}$ such that
$$ \left| \frac{1}{N}\Tr\left( f\circ T^N(X^N) \right) - \frac{1}{N}\Tr\left( f\circ \left(\sum_{i=0}^k \frac{T^i(H)}{N^{2i}}\right)(X^N) \right) \right| \leq \frac{C_{W,k}}{N^{2k+2}} \times \norm{f}_{1,c,R_N+R_{W,k}}, $$
where $X^N$ is a $d$-tuple of independent GUE random matrices, $R_N\coloneqq \max_i \norm{X_i^N}$, and $T^i\coloneqq T_1^i$ are the maps defined in Theorem \ref{thm:main}. In particular, $T^0\in\cC^{4k+4}(X_1,\dots,X_d)$ and  $T^m\in\cC^{4k+5-m}(X_1,\dots,X_d)$ for $m\geq [1,k]$.

Consequently, by using Corollary \ref{cor:taylorineq} once again, in combination with Corollary \ref{cor:multicompfin} and Equation \eqref{eqkvjnso}, there exists functions $g^i\in\cC^{4(k-i)+4}(X_1,\dots,X_d)$ for $i\in [1,k]$ and constants $C_{W,k}$ and $R_{W,k}$ such that
$$\norm{f\circ \left(\sum_{i=0}^k \frac{T^i}{N^{2i}}\right) - f\circ T^0 - \frac{g^1}{N^2} \dots - \frac{g^{k}}{N^{2k}}}_{0,R} \leq  \frac{C_{W,k}}{N^{2k+2}} \times \norm{f}_{\cC^k,R+R_{W,k}}, $$

$$\norm{g^i}_{\cC^{4(k-i)+4},R} \leq C_{W,k} \norm{f}_{\cC^{4k+4},R+R_{W,k}},\quad 1\leq i\leq k. $$

\noindent Consequently, thanks to Proposition \ref{sodncks}, we have that for some constant $C_{W,k,K}$
\begin{align*}
    \Bigg| \E\Bigg[ \frac{1}{N}\Tr\left( f\circ T^N(X^N) \right) - \frac{1}{N}\Tr\left( f\circ T^0(X^N) \right) - &\frac{1}{N^3}\Tr\left( g^1(X^N) \right) - \dots \\
    \dots& - \frac{1}{N^{2k+1}}\Tr\left( g^k(X^N) \right)  \Bigg] \Bigg| \leq \frac{C_{W,k,K}}{N^{2k+2}} \times \norm{f}_{\cL^k,K}. 
\end{align*}
Thus we conclude by applying Theorem \ref{3lessopti} and Proposition \ref{prop:condexp} repeatedly.\qed

\subsection{Proof of the strong convergence of multimatrix models (Theorem \ref{cor2})}

The statement of Theorem \ref{cor2} refers to Theorem \ref{mainthm} for the map $T^0$.
Note that $T^0 = T_1^0$ is the map defined in Theorem \ref{thm:main}.

Given $X^N$ a $d$-tuple of independent GUE random matrices, we have thanks to Proposition \ref{prop:transportmap},
\begin{align*}
    \P\left( \lim_{N\to\infty} \norm{f(Y^N)} = \norm{f\circ T^0(x)} \right) &= \P\left( \lim_{N\to\infty} \norm{f\circ T^N_1(X^N)} = \norm{f\circ T^0(x)} \right) \\
    &= \P\left( \lim_{N\to\infty} \norm{f\circ T^0(X^N)} = \norm{f\circ T^0(x)} \right),
\end{align*}
where in the last line we used Equation \eqref{eq:main} that states for some constant $C_W$, for all $N\geq 1$,
\begin{equation*}
    \sup_{H\in\M_N(\C)_{sa}^d}\ \norm{T_1^N(H) - T^0(H) } \leq \frac{C_W}{N^{2}}.
\end{equation*}
Thanks to well-known concentration estimate, see for example Equation \eqref{lvndlsvosndn}, there exists a constant $L$ such that almost surely, for $N$ large enough,
$$ \sup_{1\leq i\leq d} \norm{X_i^N} \leq L.$$
Since $f\circ T^0\in\cC^0(X_1,\dots,X_d)$ thanks to Theorem \ref{chainrule}, let $Q\in\Tr^1(X_1,\dots,X_d)$ be such that
$$ \norm{f\circ T^0-Q}_{\cC^0,L} \leq \varepsilon. $$
Thus almost surely,
$$ \limsup_{N\to\infty} \left| \norm{f\circ T^0(X^N)} - \norm{Q(X^N)} \right| \leq \varepsilon,$$
and since we can assume that $L\geq 2$,
$$ \left| \norm{f\circ T^0(x)} - \norm{Q(x)} \right| \leq \norm{f\circ T^0-Q}_{\cC^0,L} \leq \varepsilon,$$
all that remains to prove is that for any trace polynomials $Q\in\Tr^1(X_1,\dots,X_d)$, almost surely,
$$ \lim_{N\to\infty} \norm{Q(X^N)} = \norm{Q(x)}. $$
To prove the equality above, it is enough to show that for any polynomial $P$,
\begin{align*}
    &\lim_{N\to\infty} \tr(P(X^N)) = \tau(P(x)) \\
    &\lim_{N\to\infty} \norm{P(X^N)} = \norm{P(x)}.
\end{align*}
Hence, the claim is reduced to strong convergence in distribution for GUE matrices, which is a well-known result in random matrix theory; see e.g., \cite{haagerup2005new,collins2022operator,bandeira2023matrix,CGVTvT2026new}. \qed

\phantomsection
\addcontentsline{toc}{section}{References}
\bibliographystyle{amsplain}
\bibliography{DV.bib}
\end{document}